\documentclass[11pt,reqno]{amsart}
\usepackage[utf8]{inputenc}

\usepackage{amssymb,latexsym, mathtools,tikz}
\usepackage{enumerate}
\usepackage{graphicx}
\usepackage[mathscr]{euscript}
\usepackage{float}
\usepackage{placeins}
\usepackage{mdframed}
\usepackage{amssymb}
\usepackage{esint}
\usepackage{cool}
\usepackage[all,cmtip]{xy}
\usepackage{mathtools}
\usepackage{amstext} % for \text macro
\usepackage{array}   % for \newcolumntype macro
\usepackage[shortlabels]{enumitem}
\usepackage{ytableau}
\usepackage{tikz}
\usetikzlibrary{arrows,decorations.markings, cd}
\tikzstyle arrowstyle=[scale=1]
\tikzstyle directed=[postaction={decorate,
decoration={markings,mark=at position .65 with {\arrow[arrowstyle]{stealth}}}}]
\usepackage{mathdots}
\usepackage{quiver}
\usepackage{mfabacus}
\usepackage{rotating}

\newcolumntype{L}{>{$}l<{$}} % math-mode version of "l" column type
\newcolumntype{C}{>{$}c<{$}}
\newtheorem{theorem}{Theorem}[section]
\newtheorem{lemma}[theorem]{Lemma}

\newtheorem{cor}[theorem]{Corollary}
\newtheorem{prop}[theorem]{Proposition}
\newtheorem{conv}[theorem]{Convention}

\theoremstyle{definition}
\newtheorem{definition}[theorem]{Definition}
\newtheorem{example}[theorem]{Example}

\newtheorem{obs}[theorem]{Observation}
\newtheorem{notation}[theorem]{Notation}

\newtheorem{construction}[theorem]{Construction}

\theoremstyle{remark}
\newtheorem{remark}[theorem]{Remark}

\newtheorem{the context}[theorem]{The Context}
\newtheorem{question}[theorem]{Question}
\numberwithin{equation}{theorem}
\numberwithin{equation}{section}

\newcommand{\R}{\mathcal{R}}
\newcommand{\cat}[1]{\mathcal{#1}}

\newcommand{\OO}{\mathcal{O}}
\newcommand{\defi}[1]{{\bf\upshape\sffamily #1}}
\newcommand{\kk}{{\mathbf{k}}}
\newcommand{\op}{{\operatorname{op}}}

\newcommand{\ext}{\operatorname{Ext}}

\newcommand{\coker}{\operatorname{Coker}}

\newcommand{\tor}{\operatorname{Tor}}
\newcommand{\im}{\operatorname{Im}}

\newcommand{\cone}{\operatorname{Cone}}
\newcommand{\Ker}{\operatorname{Ker}}

\newcommand{\ideal}[1]{\mathfrak{#1}}
\newcommand{\m}{\ideal{m}}

\newcommand{\xx}{\mathbf{x}}
\newcommand{\bbs}{\mathbb{S}}
\newcommand{\bbz}{\mathbb{Z}}
\newcommand{\bbn}{\mathbb{N}}

\newcommand{\bbp}{\mathbb{P}}

\newcommand{\bbw}{\mathbb{W}}

\renewcommand{\geq}{\geqslant}
\renewcommand{\leq}{\leqslant}
\renewcommand{\ker}{\Ker}
\renewcommand{\hom}{\Hom}

\newcommand{\Hom}{\operatorname{Hom}}

\newcommand{\ch}{\operatorname{Ch}}

\newcommand{\gl}{\operatorname{{GL}}}

\newcommand{\maps}[5]{\xymatrix{#1 \ar[r]^-{#3} & #2 \\
#4 \ar@{|->}[r] & #5 \\}}

\renewcommand{\bar}{\operatorname{Bar}}
\newcommand{\cobar}{\operatorname{Cobar}}

\def\w{\wedge}

\def\im{\operatorname{im}}

\newcommand{\rev}{\operatorname{rev}}
\newcommand{\segpower}[2]{{{#1}^{[#2]}}}
\newcommand{\veralgg}[2]{{{#1}^{(#2)}}}
\newcommand{\veralg}[1]{{A^{(#1)}}}

\newcommand{\vermod}[2]{{A^{({\geq #2},{#1})}}}

\newcommand{\sgn}{\operatorname{sgn}}

\begin{document}

\title{Ribbon Schur Functors}
\author{Keller VandeBogert}
\email{kvandebo@nd.edu}
\address{Department of Mathematics, Notre Dame, IN 46556}

\maketitle

\begin{abstract}
    We investigate a generalization of the classical notion of a Schur functor associated to a ribbon diagram. These functors are defined with respect to an arbitrary algebra, and in the case that the underlying algebra is the symmetric/exterior algebra, we recover the classical definition of Schur/Weyl functors, respectively. In general, we construct a family of $3$-term complexes categorifying the classical concatenation/near-concatenation identity for symmetric functions, and one of our main results is that the exactness of these $3$-term complexes is equivalent to the Koszul property of the underlying algebra $A$. We further generalize these ribbon Schur functors to the notion of a multi-Schur functor and construct a canonical filtration of these objects whose associated graded pieces are described explicitly; one consequence of this filtration is a complete equivariant description of the syzygies of arbitrary Segre products of Koszul modules over the Segre product of Koszul algebras. Further applications to the equivariant structure of derived invariants, symmetric function identities, and Koszulness of certain classes of modules are explored at the end, along with a characteristic-free computation of the regularity of the Schur functor $\bbs^\lambda$ applied to the tautological subbundle on projective space.
\end{abstract}

\section{Introduction}

\subsection{Ribbon Schur Functors}

Schur functors are fundamental objects that lie at the intersection of representation theory, algebraic geometry, combinatorics, and commutative algebra. Representation theoretically, the Schur functors $\bbs^{\lambda}$ corresponding to a partition $\lambda$ are irreducible $\gl (V)$-representations in characteristic $0$, and (up to twists by the determinant representation) these constitute all irreducible representations of the general linear group (see, for instance, \cite[Chapter 2]{weyman2003}). From the algebro-geometric perspective, Schur functors corresponding to partitions may be constructed via the famous \defi{Borel--Weil theorem} (see \cite{BorelWeilThm,bott1957homogeneous}, or \cite{kempf1976collapsing}), which realizes these objects as the global sections of canonical line bundles on the complete flag variety. Combinatorially, Schur functors may be identified with their multigraded characters to obtain \defi{Schur polynomials}, which are a fundamental basis for the ring of symmetric functions (see, for instance, \cite{grinberg2014hopf} for more on this perspective). Finally, for a commutative algebraist, Schur modules over arbitrary commutative rings were constructed in the foundational work of Akin-Buchsbaum-Weyman \cite{akin1982schur}, where these objects may be described as the image of a map constructed using the (graded) commutative Hopf algebra structure on the symmetric/exterior powers.

In this paper, we take the perspective that Schur functors corresponding to \emph{ribbon diagrams} are ``more natural" than Schur functors corresponding to partitions. A \defi{ribbon diagram} is a skew shape for which consecutive rows overlap by exactly one block. 

\begin{figure}[H]
    \centering
    \begin{equation*}
        \ytableausetup{boxsize=1em}
        \begin{ytableau}
        \none & \none & \none & \\
        \none &  & & \\
        & & \\
        \end{ytableau} \qquad \qquad \qquad  \begin{ytableau}
        \none & \none &\none & \none & \none & \\
        \none & \none &\none &  & & \\
        & &  & \\
        \end{ytableau}
    \end{equation*}
    \caption{The skew shape on the left is \emph{not} a ribbon, since the bottom two rows have overlap size $2$. The shape on the right \emph{is} a ribbon.}
\end{figure}

From a representation-theoretic point of view, ribbon Schur functors may seem quite \emph{un}natural: they are highly reducible, even in characteristic $0$. However, experts in the theory of symmetric functions have long known that ribbon Schur polynomials also generate the ring of symmetric functions and are in many ways much more well-behaved than Schur polynomials corresponding to partitions (see \cite{lascoux1988ribbon}, \cite{billera2006decomposable}, \cite{reiner2007coincidences}, \cite{huang2016tableau}). For example, ribbon Schur polynomials satisfy a much simpler Pieri-type formula (the \defi{concatenation/near-concatenation identity}) than Schur polynomials corresponding to partitions, which require the Littlewood-Richardson rule to expand in general.

\begin{figure}[H]
    \centering
    \begin{equation*}
        \ytableausetup{boxsize=1em}
\begin{ytableau} 
\none & & \\
 & \\
 \end{ytableau} \cdot \begin{ytableau}
     \none & \none & \\
     & & \\
 \end{ytableau} = \begin{ytableau}
     \none & \none & \none & \none & \\
     \none & \none & & & \\
     \none & & \\
      & \\
 \end{ytableau} \quad + \quad  \begin{ytableau}
     \none & \none & \none & \none & \none & \\
     \none & & & & & \\
      & \\
 \end{ytableau}
    \end{equation*} 
    \caption{The concatenation/near-concatenation identity (here we are identifying the ribbon diagrams with their corresponding Schur polynomials).}
\end{figure}

One of the most fundamental reasons ribbon diagrams are desirable is that Schur functors corresponding to ribbons can be defined solely using the algebra structure on the symmetric algebra $S(V)$, whereas Schur functors for partitions almost always need the full Hopf algebra structure; this is immediate from the work of Akin-Buchsbaum-Weyman \cite{akin1982schur}, which establishes a canonical, characteristic-free presentation of Schur modules associated to skew partitions. This means that there is an evident way to generalize ribbon Schur functors to arbitrary algebras, and the main theme of this paper is that in order for this theory to work ``as it should", the underlying algebra needs to be Koszul.

\begin{example}
    In this example, identify the skew shapes with their corresponding Schur functors. Then:
    $$\ytableausetup{boxsize=1em}
\begin{ytableau} 
\, & \\
 & \\
 \end{ytableau} \quad  = \quad  \ker \left( S_2 (V) \otimes_k S_2 (V) \to \begin{matrix} S_4 (V) \\ \oplus \\ S_3 (V) \otimes_k V \end{matrix}\right), \quad \text{and}$$
 $$\ytableausetup{boxsize=1em}
\begin{ytableau} 
\none  & & \\
 & \\
 \end{ytableau} \quad = \quad \ker \left( S_2 (V) \otimes_k S_2 (V) \to S_4 (V) \right).$$
 In the first equality, the map $S_2 (V) \otimes_k S_2 (V) \to S_3 (V) \otimes_k V$ is the composition of \emph{both} comultiplication and multiplication on the symmetric algebra, whereas the second ribbon diagram is simply the kernel of the canonical multiplication map $S_2 (V) \otimes_k S_2 (V) \to S_4 (V)$. 
\end{example}

\subsection{Koszul Algebras and Backelin's Theorem}

A $\kk$-algebra $A$ is \defi{Koszul} if its residue field $\kk$ has a linear homogeneous minimal free resolution over $A$. Koszul algebras are one method of generalizing standard graded polynomial rings and make their appearance in a wide range of seemingly disconnected settings. Topologically, Koszul duality for quadratic algebras can be used to translate between facts about equivariant and standard cohomology (see \cite{goresky1997equivariant}). In the context of number theory, Positselski (see \cite{positselski2014galois}) has shown that certain classes of Milnor rings are Koszul algebras and relates the Milnor-Bloch-Kato conjecture to Koszulness of certain quotient rings. Combinatorially, early work of Backelin \cite{backelinThesis} showed that there was an equivalence between Koszul algebras and distributivity of certain associated subspace lattices, and since then Koszulness of rings associated to combinatorial objects has been an active and fruitful area of research (see for instance \cite{yuzvinsky2001orlik}, \cite{mastroeni2022chow}).

Suppose for the time being that $V$ is a vector space over a field $\kk$; all tensor products will be taken over $\kk$. Recall that a quadratic algebra $A$ is any quotient of the tensor algebra $T(V)$ by a quadratic ideal $(Q_2) \subset T_{\geq 2} (V)$, where $Q_2 \subset V \otimes V$. In particular, each graded piece of $A$ is the quotient
$$A_n = \frac{V^{\otimes i}}{Q_2 \otimes V^{\otimes n-2} + V \otimes Q_2 \otimes V^{\otimes n-3} + \cdots + V^{\otimes n-2} \otimes Q_2}.$$
Although Koszulness is often defined in terms of a homological property of the residue field of $A$, a well-known result of Backelin establishes the elegant fact that Koszulness has an equivalent formulation in terms of a purely combinatorial property of the collections $Q_2 \otimes V^{\otimes n-2},  V \otimes Q_2 \otimes V^{\otimes n-3},  \dots , V^{\otimes n-2} \otimes Q_2$:
\begin{theorem}[Backelin \cite{backelinThesis}, see also \cite{beilinson1996koszul}]
    Let $A$ be a quadratic algebra. Then $A$ is Koszul if and only if the collection $Q_2 \otimes V^{\otimes n-2},  V \otimes Q_2 \otimes V^{\otimes n-3},  \cdots , V^{\otimes n-2} \otimes Q_2$ generates a distributive subspace lattice for all $n \geq 2$. 
\end{theorem}
This alternative (equivalent) formulation of Koszulness is surprising because it implies that the Koszulness assumption gives you ``more than you bargained for"; the backwards implication of Backelin's theorem is evident, but a priori distributivity seems like a much stronger property than Koszulness. In this paper, we reinterpret this distributivity property as being equivalent to the exactness of a family of $3$-term sequences (see Proposition \ref{prop:distributivityComplexes}). A key observation here is that these $3$-term sequences yield a defining identity for a generalization of ribbon Schur functors to any Koszul algebra.

\subsection{Generalized Ribbon Schur Functors and the Concatenation/near-concatenation criterion for Koszulness}

Let $A$ be any standard graded quadratic $R$-algebra\footnote{the quadratic assumption is unnecessary, but since these objects are most well-behaved when the algebra $A$ is assumed to be Koszul, we will often be in this setting anyway.} and $\alpha = (\alpha_1 , \dots , \alpha_n)$ any sequence of positive integers. Then the \defi{ribbon Schur module} $\bbs^\alpha_A$ is defined as the kernel of the map
$$A_{\alpha_1} \otimes_R A_{\alpha_2} \otimes_R \cdots \otimes_R A_{\alpha_n} \to \bigoplus_{i=1}^{n-1} A_{\alpha_1} \otimes_R \cdots \otimes_R A_{\alpha_i+\alpha_{i+1}} \otimes_R \cdots \otimes_R A_{\alpha_n},$$
where each component of the map is induced by the natural multiplication maps $A_{\alpha_i} \otimes_R A_{\alpha_{i+1}} \to A_{\alpha_i + \alpha_{i+1}}$. As mentioned previously, when $A = S^\bullet (V)$ is the symmetric algebra, this definition recovers the Akin-Buchsbaum-Weyman \cite{akin1982schur} definition of the Schur functor associated to the ribbon diagram induced by $\alpha$. The importance of using ribbon Schur functors to canonically describe syzygies of modules over Veronese subalgebras of the polynomial ring was discovered in \cite{almousa2022equivariant}, and the goal of this paper is to fully expand upon this perspective in a much more general setting. 

One of our first main results is the following (for notation, see Definition \ref{def:operationsOfComps}). This statement may be interpreted as saying that the concatenation/near-concatenation identity observed in the classical theory of Schur polynomials is the shadow of a canonical short exact sequence \emph{of functors}, and that the concatenation/near-concatenation identity is actually a defining property of Koszulness.

\begin{theorem}\label{thm:introKoszulCriterion}
    Let $A$ be any Koszul $R$-algebra (where $R$ is any commutative ring) and $\alpha$, $\beta$ be any two compositions.
    \begin{enumerate}
        \item There is a canonical isomorphism of $R$-modules
        $$\left( \bbs^{\alpha}_A \right)^* \cong \bbs^{\alpha^t}_{A^!},$$
        where $(-)^!$ denotes the quadratic dual, $\alpha^t$ denotes the transposed ribbon diagram, and $(-)^* := \hom_R (- , R)$.
        \item There is a canonical short exact sequence
        $$0 \to \bbs^{\alpha \cdot \beta}_{A} \to \bbs^{\alpha}_A \otimes_R \bbs^{\beta}_A \to \bbs^{\alpha \odot \beta}_A \to 0.$$
        Conversely, this sequence is exact for \emph{all} compositions \emph{if and only if} $A$ is Koszul.
    \end{enumerate}
\end{theorem}
This is surprising since, intuitively, the classical concatenation/near-concatenation identity for symmetric functions may seem like a consequence of the combinatorial/representation-theoretic structure on the symmetric algebra $S^\bullet (V)$. Theorem \ref{thm:introKoszulCriterion} tells us that this additional structure is more of a red herring, and the classical identity is actually a consequence of a much more fundamental algebraic property. 

One could also ask if this equivalent formulation of Koszulness is actually useful for proving classes of modules are Koszul, and in Subsection \ref{subsec:buildingKoszulModules} we use this criterion to give quick proofs of the existence of large classes of Koszul modules over arbitrary Koszul algebras $A$. In the case of the symmetric/exterior algebras, this criterion for Koszulness gives a very simple proof that a general class of modules parametrized by arbitrary skew-partitions are Koszul. Previous proofs that these modules were Koszul only worked for diagrams corresponding to partitions, and resorted to geometric arguments that realized these objects as arising from taking cohomology of line bundles on flag varieties (see, for instance, \cite[Theorem 2.2]{gao2022cohomology}). Our proof is totally characteristic-independent, allows for the algebra to be over any commutative ring, and requires no machinery coming from algebraic geometry.

We also generalize this definition of Schur functors to allow for module inputs $M$ and $N$, denoted $\bbs^{\alpha}_{M,A,N}$; this is a generalization even in the classical case of Schur functors, and allows for elegant descriptions of the higher derived invariants associated to pairs of Koszul modules:

\begin{theorem}\label{thm:introTorAndExt}
    Let $A$ be a Koszul algebra and $N$ (resp. $M$) a left (resp. right) Koszul $A$-module. Then there is a canonical isomorphism of $A$-modules
    $$\tor_i^{A} (M, N) = \bbs^{(1^i)}_{M,A,N}.$$
    If $M$ is instead a left $A$-module, then there is a canonical isomorphism of $A$-modules
    $$\ext^i_A (M , N) = \bbs^{(i)}_{M^! , A^! , (N^*)^!}.$$
    If $A$, $M$, and $N$ have any ambient group action, then the above isomorphisms are equivariant.
\end{theorem}

\begin{remark}
    Note that there are other places in the literature where generalizations of Schur functors have been considered, such as the work of Sam-Snowden (\cite{sam2017infinite}, \cite{sam2019some}) which approaches the problem from a much more representation-theoretic perspective. These constructions are done in the standard type ABCD framework and do not apply to arbitrary Koszul algebras like the constructions in this paper do.
\end{remark}

% \begin{enumerate}
%     \item what is a schur functor: give the multiple different perspectives: the Kempf (and Magyar) POV, the RT POV (induced by Schur algebra), and the concrete commutative algebraic description by Akin-Buchsbaum-Weyman.
%     \item what is a ribbon (with a diagram)
%     \item ribbons are unnatural from some persp: highly reducible even in char 0, however, canonically defined in a way that doesn't use the Hopf algebra structure of exterior algebra
%     \item do an example of the difference between $(2,2)$ partition and $(2,2)$ ribbon
%     \item The seminal work of Backelin equated with distributivity: we introduce quasidistributivity instead (mention why distributivity doesn't get off the ground because of the modularity holding for free)
% \end{enumerate}

\subsection{Multi-Schur Functors}

Let us first motivate the construction of a multi-Schur functor. Segre subalgebras of the tensor product of Koszul algebras are well-known to be Koszul algebras. Given Koszul algebras $A$ and $B$, let $A \circ B$ denote the Segre product; the homogeneous components of the (dual) quadratic dual may be computed as the kernel of a map that applies the multiplication on $A$ and $B$ ``diagonally". For example, there is an equality
$$\left( (A \circ B)^! \right)^*_3 = \ker \left( A_1^{\otimes 3} \otimes B_1^{\otimes 3} \to \begin{matrix}
    A_2 \otimes A_1 \otimes B_2 \otimes B_1 \\
    \oplus \\
    A_1 \otimes A_2 \otimes B_1 \otimes B_2. \\
\end{matrix} \right).$$
This is the same thing as applying the defining relations for the Schur modules $\bbs^{(1^3)}_A$ and $\bbs^{(1^3)}_B$ diagonally, and leads us directly to the notion of a \defi{multi-Schur functor}. 

A multi-Schur functor $\bbs^{\underline{\alpha}}_{\underline{A}}$ takes as inputs \emph{tuples} of compositions and algebras:
$$\underline{\alpha} = (\alpha^1 , \dots , \alpha^n), \quad \underline{A} = (A^1 , \dots , A^n),$$
and is defined by taking the kernel of the defining relations of each of the ribbon Schur modules $\bbs^{\alpha^i}_{A^i}$ applied diagonally, exactly as above. Surprisingly, multi-Schur modules satisfy an identical concatenation/near-concatenation sequence (appropriately generalized; see Lemma \ref{lem:theMultiSchurSES}), and we can further extend this definition to allow for tuples of modules. This gives us a similar clean description of Tor modules over Segre products:

\begin{theorem}\label{thm:introTorForMulti}
    Consider tuples of the form
    $$\underline{A} = (A^1 , \dots , A^n),$$
    $$\underline{N} = (N^1 , \dots , N^n), \quad \underline{M} = (M^1 , \dots , M^n),$$
    where each $A^i$ is a Koszul $R$-algebra and $N^i$ (resp. $M^i$) is a left (resp. right) $A^i$-module. Then there is a canonical isomorphism of $A^1 \circ \cdots \circ A^n$-modules: 
    $$\tor_i^{A^1 \circ \cdots \circ A^n} (M^1 \circ \cdots \circ M^n, N^1 \circ \cdots \circ N^n) = \bbs^{(1^i)}_{\underline{M} , \underline{A} , \underline{N}}.$$
    This isomorphism is natural with respect to morphisms of algebras; in particular it is equivariant if each of the Koszul algebras has any ambient group action.
\end{theorem}

\subsection*{Canonical Filtrations}

Although Theorem \ref{thm:introTorForMulti} is simple to state, it still gives us very little information about the relationship between a multi-Schur functor (which is a priori a purely formal construction) and objects that we may better understand. A standard method of trying to understand an object is to construct a filtration whose associated graded objects are ``simple" in the appropriate sense. 

As it turns out, multi-Schur functors $\bbs^{\underline{\alpha}}_{\underline{M} , \underline{A} , \underline{N}}$ admit a canonical (that is, totally functorial) filtration whose associated graded pieces are tensor products of the Schur functors $\bbs^{\beta^i}_{M^i,A^i,N^i}$, where $1 \leq i \leq n$. The difficult part is determining the poset that parametrizes this filtration, and our main result related to multi-Schur functors is an explicit description of this parametrization.

Before stating the result, we need to introduce one piece of notation: given a composition $\alpha = (\alpha_1 , \dots , \alpha_\ell)$ and a subset $I \subset [ \ell -1]$, the notation $\sigma_I (\alpha)$ denotes the composition obtained by adding $\alpha_i$ and $\alpha_{i+1}$ for all $i \in I$. For example:
$$\sigma_{\{ 1,2,4,6,7\}} (2,1,3,5,3,6,5,3) = (6,8,14).$$
With this notation in hand, we have:
\begin{theorem}\label{thm:introMultiFiltration}
   Consider the tuple
   $$\underline{\alpha} = (\alpha^1 , \dots , \alpha^n),$$
   where each composition $\alpha^i$ has a fixed length $\ell$. With notation and hypotheses as in Theorem \ref{thm:introTorForMulti}, the multi-Schur module $\bbs^{\underline{\alpha}}_{\underline{M} , \underline{A} , \underline{N}}$ admits a canonical filtration with associated graded pieces of the form
        $$\bbs^{\sigma_{I_1} (\alpha^1)}_{M^1,A^1,N^1} \otimes_R \bbs^{\sigma_{I_2} (\alpha^2)}_{M^2, A^2, N^2} \otimes_R \cdots \otimes_R \bbs^{\sigma_{I_n} (\alpha^n)}_{M^n,A^n, N^n},$$ 
        where the subsets $I_1 , \dots , I_n \subset [\ell -1]$ range over all choices such that $I_1 \cap I_2 \cap \cdots \cap I_n = \varnothing$. This filtration is equivariant with respect to any kind of ambient group action.
\end{theorem}

\begin{example}
    If in the statement of Theorem \ref{thm:introMultiFiltration}, one has $n=2$ and each composition has length $3$, the poset parametrizing the filtration factors has Hasse diagram
    % https://q.uiver.app/?q=WzAsOSxbMSwyLCIoXFx2YXJub3RoaW5nLFxcdmFybm90aGluZykiXSxbMCwxLCIoXFx7IDEgXFx9ICwgXFx2YXJub3RoaW5nKSJdLFsxLDEsIiggXFx7IDIgXFx9ICwgXFx2YXJub3RoaW5nKSJdLFsyLDEsIihcXHZhcm5vdGhpbmcgLCBcXHsgMSBcXH0gKSJdLFszLDEsIihcXHZhcm5vdGhpbmcgLCBcXHsgMiBcXH0gKSJdLFswLDAsIihcXHsxLDJcXH0gLCBcXHZhcm5vdGhpbmcpIl0sWzEsMCwiKFxcezEgXFx9ICwgXFx7IDIgXFx9KSJdLFsyLDAsIihcXHsyIFxcfSAsIFxceyAxIFxcfSkiXSxbMywwLCIoXFx2YXJub3RoaW5nICwgXFx7MSwyIFxcfSkiXSxbMCwxLCIiLDAseyJzdHlsZSI6eyJoZWFkIjp7Im5hbWUiOiJub25lIn19fV0sWzAsMiwiIiwyLHsic3R5bGUiOnsiaGVhZCI6eyJuYW1lIjoibm9uZSJ9fX1dLFswLDMsIiIsMix7InN0eWxlIjp7ImhlYWQiOnsibmFtZSI6Im5vbmUifX19XSxbMCw0LCIiLDIseyJzdHlsZSI6eyJoZWFkIjp7Im5hbWUiOiJub25lIn19fV0sWzEsNSwiIiwwLHsic3R5bGUiOnsiaGVhZCI6eyJuYW1lIjoibm9uZSJ9fX1dLFsxLDYsIiIsMCx7InN0eWxlIjp7ImhlYWQiOnsibmFtZSI6Im5vbmUifX19XSxbNSwyLCIiLDAseyJzdHlsZSI6eyJoZWFkIjp7Im5hbWUiOiJub25lIn19fV0sWzIsNywiIiwwLHsic3R5bGUiOnsiaGVhZCI6eyJuYW1lIjoibm9uZSJ9fX1dLFszLDgsIiIsMix7InN0eWxlIjp7ImhlYWQiOnsibmFtZSI6Im5vbmUifX19XSxbNCw4LCIiLDEseyJzdHlsZSI6eyJoZWFkIjp7Im5hbWUiOiJub25lIn19fV0sWzQsNiwiIiwxLHsic3R5bGUiOnsiaGVhZCI6eyJuYW1lIjoibm9uZSJ9fX1dLFszLDcsIiIsMSx7InN0eWxlIjp7ImhlYWQiOnsibmFtZSI6Im5vbmUifX19XV0=
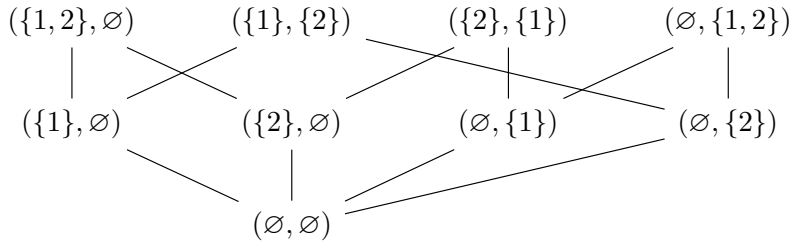
\[\begin{tikzcd}
	{(\{1,2\} , \varnothing)} & {(\{1 \} , \{ 2 \})} & {(\{2 \} , \{ 1 \})} & {(\varnothing , \{1,2 \})} \\
	{(\{ 1 \} , \varnothing)} & {( \{ 2 \} , \varnothing)} & {(\varnothing , \{ 1 \} )} & {(\varnothing , \{ 2 \} )} \\
	& {(\varnothing,\varnothing)}
	\arrow[no head, from=3-2, to=2-1]
	\arrow[no head, from=3-2, to=2-2]
	\arrow[no head, from=3-2, to=2-3]
	\arrow[no head, from=3-2, to=2-4]
	\arrow[no head, from=2-1, to=1-1]
	\arrow[no head, from=2-1, to=1-2]
	\arrow[no head, from=1-1, to=2-2]
	\arrow[no head, from=2-2, to=1-3]
	\arrow[no head, from=2-3, to=1-4]
	\arrow[no head, from=2-4, to=1-4]
	\arrow[no head, from=2-4, to=1-2]
	\arrow[no head, from=2-3, to=1-3]
\end{tikzcd}\]
For a concrete example of the explicit filtration factors, see Example \ref{ex:mSchurFiltrationEx1}.
\end{example}

% \begin{enumerate}
%     \item One generalization is to allow instead for tuples of ribbons.
%     \item The most obvious place this arises naturally is for describing derived invariants over Segre subalgebras, but we needed the general theory!
%     \item We build a canonical filtration compatible with any kind of equivariant structure that holds as a ridiculous level of generality
% \end{enumerate}

\subsection{Organization of Paper}

This paper is organized as follows. In Section \ref{sec:augmentedBarComplexes} we recall some background on augmented bar complexes and their homogeneous strands. We also establish conventions and notation that will be used throughout the paper. In Section \ref{sec:qdistributivity} we recall the notion of distributivity for collections of arbitrary $R$-submodules; much of this material is essentially contained in \cite{polishchuk2005quadratic}, but since we do not assume the ambient ring is a field, there are some additional details that need to be checked. We also introduce a collection of modules $L^I_{M_1 , \dots , M_n}$ associated to a distributive collection of $R$-submodules that will end up being an equivalent way to define ribbon Schur functors in the setting of arbitrary quadratic algebras. Under appropriate assumptions, we also show that distributivity behaves well with respect to ``dualization".

Sections \ref{sec:ribbonSchurs} and \ref{sec:multiSchurFunctors} develop the theory of ribbon Schur functors associated to arbitrary quadratic algebras/modules. In Section \ref{sec:ribbonSchurs}, we define ribbon Schur functors associated to arbitrary quadratic algebras/modules and prove the concatenation/near-concatenation Koszulness criterion. Subsection \ref{subsec:ribbonsForAlg} is devoted to giving explicit examples and illustrating the statements for concrete choices of Koszul algebras, and the main proofs of the results in full generality are given in Subsection \ref{subsec:ribbonsForMods}. We will see that many duality properties of Schur and Weyl functors in the classical setting as a consequence of Koszul duality for these more general ribbon Schur functors. 

In Section \ref{sec:multiSchurFunctors} we define multi-Schur functors as described earlier. The beginning of this section explains how to ugprade all of the statements of Section \ref{sec:ribbonSchurs} to this level of generality; it turns out that there is a large amount of additional bookkeeping needed when doing this. Subsections \ref{subsec:filtrations} and \ref{subsec:canonicalFiltrationsOfMultiSchur} prove all of the results necessary to construct the filtration as stated in Theorem \ref{thm:introMultiFiltration}.

Section \ref{sec:applications} is where we get to see the utility of the theory developed in Sections \ref{sec:ribbonSchurs} and \ref{sec:multiSchurFunctors} for giving canonical descriptions/filtrations to many of the invariants associated to Koszul algebras/modules. Subsection \ref{subsec:TorAndExt} proves the isomorphisms of Tor and Ext between Koszul modules as stated in Theorem \ref{thm:introTorAndExt}, and Subsections \ref{subsec:VeroneseCase} and \ref{subsec:SegreCase} give canonical descriptions of the derived invariants over Veronese/Segre subalgebras in terms of the original algebra(s), and even deduce an interesting symmetric function identity as a result taking a character count on the minimal free resolution over a Segre product. In Subsection \ref{subsec:buildingKoszulModules}, we construct a large class of Koszul modules over an arbitrary Koszul algebra and, in the case of symmetric/exterior algebras, we generalize a construction of Koszul modules associated to Schur functors corresponding to arbitrary skew partitions. We use this to give a characteristic free computation of the sheaf cohomology of $\bbs^\lambda (\R)$ on $\bbp (V)$, where $\R$ denotes the tautological subbundle on projective space.

Section \ref{sec:generalKoszulAlg} is an appendix that recalls the equivalence between Koszulness and distributivity (that is, Backelin's theorem) and interprets distributivity in terms of exactness properties of so-called \defi{refinement complexes}. Again, much of this material follows from straightforward generalizations of the material in \cite{polishchuk2005quadratic}, but we check the details for sake of completeness. In subsection \ref{subsec:generalitiesOnQuad} we first discuss some generalities and definitions related to general quadratic algebras. The bulk of this appendix is dedicated to the material of subsection \ref{subsec:refinementComplexes}, which defines \emph{refinement complexes} and establishes some important conventions for these refinement complexes. In subsections \ref{subsec:KoszulnessAndDistr} and \ref{subsec:KoszulMods}, we recall Backelin's theorem along with the appropriate analogs for Koszul modules and translate these results into statements about the family of refinement complexes. Finally, in Subsection \ref{subsec:PriddyCx} we recall the Priddy complex and some of its properties in the generality in which we are working.

\subsection*{Acknowledgments}

The author thanks Ayah Almousa, Vic Reiner, Steven Sam, and Jerzy Weyman for helpful conversations related to this material or comments/corrections on earlier drafts of this paper. The author also thanks the anonymous referee for a close reading with many helpful corrections/suggestions. The author was supported by NSF grant DMS-2202871.

\section{Augmented Bar Constructions}\label{sec:augmentedBarComplexes}

% \keller{make this section just a general section on augmented algebras}

% \keller{The constructions, but also: if $M$ is left/right $A$-module, then the bar of $M$ is a left/right DG bar of $A$ module. Note that it suffices to consider only left modules by replacing $A$ with its opposite algebra (one needs to prove that the Yoneda structure coincides with the opposite structure).}

In this section, we establish some conventions and notation to be used for the remainder of the paper. We also define one of the most important elements of the paper: the augmented bar complex associated to an augmented $R$-algebra. Unsurprisingly, understanding homogeneous strands of the augmented bar complex is equivalent to understanding Koszulness of the algebra $A$. 

\begin{definition}
    Let $R$ be a commutative Noetherian ring and $A$ any associative unital $R$-algebra.
    \begin{enumerate}
        \item The algebra $A$ is ($\bbz$-)\defi{graded} if $A = \bigoplus_{i \in \bbz} A_i$ with $A_i \cdot A_j \subset A_{i+j}$ for all $i,j \in \bbz$, and $1_A \in A_0$.
        \item The algebra $A$ is \defi{augmented} if $A = R \oplus A_+$, where $R = R \cdot 1_A$ and $A_+$ is a two-sided ideal in $A$.
        \item The algebra $A$ is \defi{quadratic} if $A = T(A_1) / (Q_2^A)$, where $T(-)$ denotes the tensor algebra functor, $Q_2^A \subset A_1 \otimes_R A_1$, and $(Q_2)$ denotes the two-sided ideal generated by $Q_2$.
        \item A left $A$-module $M$ is \defi{graded} if $M = \bigoplus_{i \in \bbz} M_i$ and $A_i M_j \subset M_{i+j}$ for all $i,j \in \bbz$.
    \end{enumerate}
    The notation $A^\op$ denotes the \defi{opposite algebra} of $A$; that is, the underlying set of $A^\op$ is the same as $A$ but with multiplication defined by $a \cdot b := b \cdot a$. Any right $A$-module is equivalently a left $A^\op$-module.
\end{definition}

\begin{remark}
    For quadratic $R$-algebras, we will always assume that the grading is induced by the standard grading on the tensor algebra. In particular, all quadratic algebras are nonnegatively graded.
\end{remark}

The following convention is extremely important, and without it most of the arguments given later in this paper can not even get off the ground.

\begin{conv}
Throughout this paper, all graded $R$-algebras $A$ will be assumed to be finitely-generated flat $R$-modules in each degree. Similarly, all graded $A$-modules will be assumed to be finitely-generated and flat in each degree. Sometimes we will assume that $A$ or $M$ is even $R$-projective in each degree, but it is \emph{always} true that they are at least finitely-generated and flat in each degree.
\end{conv}

There are two fundamental operations on algebras that will be particularly interesting for us:

\begin{definition}\label{def:VeroneseAndSegre}
Let $A$ and $B$ be graded $R$-algebras. The \defi{d-th Veronese power} $\veralg{d}$ of $A$ is defined to be the subalgebra
$$\veralg{d} := \bigoplus_{i \equiv 0 \mod d} A_i \subset A.$$
Let $M$ be a graded (left) $A$-module of initial degree $t$. Then the $d$th Veronese power $\veralgg{M}{d}$ is defined to be the $A^{(d)}$-module
$$M^{(d)} := \bigoplus_{i \equiv t \mod d} M_i \subset M.$$

The \defi{Segre product} $A \circ B$ of $A$ and $B$ is defined to be the subalgebra
$$A \circ B := \bigoplus_{i \geq 0} A_i \otimes_R B_i \subset A \otimes_R B.$$
The \defi{d-th Segre power} $\segpower{A}{d}$ of $A$ is defined to be the iterated Segre product
$$\underbrace{A \circ A \circ \cdots \circ A}_{d \ \text{times}} \subset A^{\otimes_R d}.$$
Given a graded left $A$ (resp. $B$)-module $M$ (resp. $N$) of initial degree $s$ (resp. $t$), the Segre product $M \circ N$ is the left $A \circ B$ module defined as 
$$M \circ N := \bigoplus_{i \geq 0} M_{i+s} \otimes_R N_{i+t}.$$
The Segre product of Koszul right modules is defined identically.

The $d$th Segre power $\segpower{M}{d}$ of $M$ is defined to be the Segre product
$$\underbrace{M \circ \cdots \circ M}_{d \ \text{times}}.$$
\end{definition}

\begin{remark}
    Often one uses the convention that the generators of the Veronese/Segre subalgebras have been rescaled to have degree $1$. We will actually have no need for this convention in the paper as long as we define an algebra to be Koszul if the associated Priddy complex (see Theorem \ref{thm:thePriddyCxWorks}) is a resolution; this is equivalent to the distributivity of a certain set of submodules (by Backelin's theorem), which is again a condition that still makes sense regardless of the generators having degree $1$. That being said, many of the results stated in this paper will be written as if the generators are in degree $1$, but again this condition is only required ``up to rescaling". 
\end{remark}

Likewise, the graded dual of an algebra will be useful for describing certain Ext modules. All duals in this paper are understood to be \emph{graded} duals; that is
$$\hom_A (M , R) := \bigoplus_{i \in \bbz} \hom_R (M_i , R).$$

\begin{definition}\label{def:gradedDualConv}
Let $M$ be a left $A$-module, where $A$ is any graded $R$-algebra. Then the graded dual $M^* := \hom_A (M,R)$ is canonically a right $A$-module via the action
    $$(m^* a)(n) := m^* (an), \quad m^* \in M^*, \ a \in A, \ n \in M.$$
    Moreover, there is a natural isomorphism of $A^\op$-modules $(M^\op)^* \cong (M^*)^\op$. 
\end{definition}

Recall that modules over (positively) graded algebras satisfy a strong form of Nakayama's lemma:

\begin{lemma}[Nakayama's Lemma]\label{lem:NAK}
    Let $A$ be any nonnegatively graded $R$-algebra and $M$ a graded left $A$-module with $M_i = 0$ for $i \ll 0$. If $A_+ M = M$, then $M = 0$. 
\end{lemma}

Now we define the augmented bar complex:

\begin{definition}
    Let $A$ be any augmented $R$-algebra and let $A_+ := A/A_0$. Given any left $A$-module $M$, the \defi{augmented Bar complex} $\bar^A (M)$ is the complex of $A$-modules with
    $$\bar^A_i (M) := A \otimes_R A_+^{\otimes i} \otimes_R M, \quad \text{with differential}$$
    $$d^B (a_0 \otimes a_1 \otimes \cdots \otimes a_i \otimes m) := \sum_{j=0}^{i-1} (-1)^j a_0 \otimes \cdots \otimes a_j \cdot a_{j+1} \otimes \cdots \otimes a_i \otimes m$$
    $$+ (-1)^i a_0 \otimes \cdots \otimes a_{i-1} \otimes a_i m.$$
    The \defi{augmented cobar complex} $\cobar^A (M)$ on $M$ is the graded $R$-dual of $\bar^A (M)$. The bar complex and cobar complex are both bigraded by the homological and internal degree. In other words:
    $$\bar^A_i (M)_j = \bigoplus_{k_0 + \cdots + k_i + \ell = j} A_{j_1} \otimes_R \cdots \otimes_R A_{j_k} \otimes_R M_\ell.$$
    With respect to the internal grading, the differential of $\bar^A (M)$ has degree $0$. 

    The notation $\bar^A (M)_n$ will denote the (internal) \defi{degree $n$ strand} of the bar complex.
\end{definition}

\begin{remark}
    For right $A$-modules $M$, there is a similar bar complex construction $\bar^A (M)^\op$ where $M$ appears as the leftmost tensor factor and the differential is defined analogously. This construction is compatible with taking the opposite algebra in the sense that there is an isomorphism of complexes
    $$\bar^{A^\op} (M) \cong \bar^A (M)^\op,$$
    where $M$ is viewed as a left $A^\op$-module on the lefthand side of this isomorphism. This means that throughout this section, it is of no loss of generality to assume that $M$ is a left $A$-module.
\end{remark}

\begin{example}
The degree $4$ strand $(R \otimes_A \bar^A (M))_4$ is the following complex of projective $R$-modules:
    % https://q.uiver.app/?q=WzAsOCxbMCwxLCJBXzFee1xcb3RpbWVzIDR9Il0sWzEsMCwiQV8yIFxcb3RpbWVzX1IgQV8xXntcXG90aW1lcyAyfSJdLFsxLDEsIkFfMSBcXG90aW1lc19SIEFfMiBcXG90aW1lc19SIEFfMSJdLFsxLDIsIkFfMV57XFxvdGltZXMgMn0gXFxvdGltZXNfUiBBXzIiXSxbMiwwLCJBXzMgXFxvdGltZXNfUiBBXzEiXSxbMiwxLCJBXzIgXFxvdGltZXNfUiBBXzIiXSxbMiwyLCJBXzMgXFxvdGltZXNfUiBBXzEiXSxbMywxLCJBXzQiXSxbMCwyXSxbMiw1XSxbNSw3XSxbMSwyLCJcXGJpZ29wbHVzIiwxLHsic3R5bGUiOnsiYm9keSI6eyJuYW1lIjoibm9uZSJ9LCJoZWFkIjp7Im5hbWUiOiJub25lIn19fV0sWzIsMywiXFxiaWdvcGx1cyIsMSx7InN0eWxlIjp7ImJvZHkiOnsibmFtZSI6Im5vbmUifSwiaGVhZCI6eyJuYW1lIjoibm9uZSJ9fX1dLFs0LDUsIlxcYmlnb3BsdXMiLDEseyJzdHlsZSI6eyJib2R5Ijp7Im5hbWUiOiJub25lIn0sImhlYWQiOnsibmFtZSI6Im5vbmUifX19XSxbNSw2LCJcXGJpZ29wbHVzIiwxLHsic3R5bGUiOnsiYm9keSI6eyJuYW1lIjoibm9uZSJ9LCJoZWFkIjp7Im5hbWUiOiJub25lIn19fV1d
\[\begin{tikzcd}
	& {A_2 \otimes_R A_1^{\otimes 2}} & {A_3 \otimes_R A_1} \\
	{A_1^{\otimes 4}} & {A_1 \otimes_R A_2 \otimes_R A_1} & {A_2 \otimes_R A_2} & {A_4} \\
	& {A_1^{\otimes 2} \otimes_R A_2} & {A_1 \otimes_R A_3}
	\arrow[from=2-1, to=2-2]
	\arrow[from=2-2, to=2-3]
	\arrow[from=2-3, to=2-4]
	\arrow["\bigoplus"{description}, draw=none, from=1-2, to=2-2]
	\arrow["\bigoplus"{description}, draw=none, from=2-2, to=3-2]
	\arrow["\bigoplus"{description}, draw=none, from=1-3, to=2-3]
	\arrow["\bigoplus"{description}, draw=none, from=2-3, to=3-3]
\end{tikzcd}\]
\end{example}

The following are some fundamental properties of (co)bar constructions that will be useful in later sections:

\begin{prop}
    Let $A$ be an augmented graded $R$-algebra and $M$ any left $A$-module.
    \begin{enumerate}
        \item The augmented bar complex $\bar^A (M)$ is a flat resolution of $M$.
        \item The augmented cobar complex $\cobar^A (A)$ is an associative DG-algebra via the standard tensor algebra product. Likewise, the cobar construction $\cobar^A (M)$ is a right DG-module over $\bar^A (A)$.
        \item There are isomorphisms:
        $$H_i (R \otimes_A \bar^A (M)) = \tor_i^A (R,M), \quad H^i (\bar^A (M)^*) = \ext^i_A (M,R).$$
        \item For any right $A$-module $N$, there are isomorphisms
        $$H_i (N \otimes_A \bar^A (M)) = \tor_i^A (N,M), \quad \tor_i^A (N,M)^* \cong \ext^i_A (M,N^*),$$
        where $N^*$ is a left $A$-module via the convention of Definition \ref{def:gradedDualConv}. If $N$ is instead a left $A$-module, there is an isomorphism
        $$H^i (\hom_A (\bar^A (M) , N)) = \ext^i_A (M,N).$$
    \end{enumerate}
\end{prop}

\begin{proof}
    Both $(1)$ and $(2)$ are well known. The statements $(3)$ and $(4)$ follow from the fact that Tor and Ext may be computed using flat resolutions instead of projective resolutions.
\end{proof}

We conclude this section with a straightforward observation that will be useful to write explicitly, since it will be cited many times in later sections.

\begin{obs}\label{obs:rightResByProj}
    Let $M$ be any $R$-module and
    $$0 \to M \to F_0 \to F_1 \to \cdots \to F_n \to 0$$
    any exact complex such that $F_i$ is a flat (resp. projective) $R$-module for each $i = 0 , \dots , n$. Then $M$ is flat (resp. projective). 
\end{obs}

\begin{proof}
    Proceed by induction on $n$, where the case $n=0$ implies $M = F_0$ is evidently flat (resp. projective). If $n > 0$, let $C := \im (F_0 \to F_1)$. By the inductive hypothesis, the module $C$ is flat (resp. projective) and there is a short exact sequence
    $$0 \to M \to F_0 \to C \to 0.$$
    If $C$ is projective, this sequence splits and hence $M$ is also projective. If $C$ is instead flat, apply the functor $- \otimes_R N$ for every $R$-module $N$ and employ the long exact sequence of homology to deduce that $M$ is flat.
\end{proof}

\section{Distributivity and Submodule Lattices}\label{sec:qdistributivity}

In this section, we recall a general family of complexes that may be associated to any collection of $R$-submodules $M_1 , \dots , M_n \subset M$ (see \cite[Chapter 2]{polishchuk2005quadratic} for the case over a field, though the theory is essentially identical here). The exactness properties of these complexes will be used to define distributivity, and under some additional hypotheses on the collection $M_1 , \dots , M_n$ we will see that distributivity satisfies a simple duality. The purpose of this section is to introduce the modules $L^I_{M_1 , \dots , M_n}$ of Definition \ref{def:theLModules} and study their flatness/projectivity/duality properties in the generality established in the previous section. 

All modules in this section are assumed to be finitely generated (including all submodules). For convenience, we recall the definition of a lattice.

\begin{definition}
    A poset is a set $S$ equipped with a partial order $\leq$. A poset is a \defi{lattice} if any two pairs of elements $a,b \in S$ have a well-defined meet and join, denoted $a \w b$ and $a \vee b$, respectively. 

    % A lattice $S$ is \defi{distributive} if for any three elements $a,b,c \in S$, the relation
    % $$a \w (b \vee c) = (a \w b) \vee (a \w c)$$
    % holds. A lattice $S$ is \defi{modular} if for any three elements $a,b,c \in S$ with $a \geq c$, the relation
    % $$a \w (b \vee c) = (a \w b) \vee c$$
    % holds.
\end{definition}

\begin{notation}\label{not:faceDegen}
    Let $[n] := \{ 1 , \dots , n \}$ for some integer $n$. The \defi{$j$th degeneracy map} $s_j : [n] \to [n-1]$ is defined to be the surjection
    $$s_j (i) := \begin{cases} i & \text{if} \ i \leq j, \\
    i-1 & \text{if} \ i > j. \end{cases}$$
    Likewise, the \defi{$j$th face map} $d_j : [n-1] \to [n]$ is defined to be the map
    $$d_j (i) := \begin{cases} i & \text{if} \ i < j, \\
    i+1 & \text{if} \ i \geq j. \end{cases}$$
    Likewise, given any set $I \subset [n]$, the notation $\sgn (j,I)$ for any $j \notin I$ denotes the sign of the permutation that reorders the set $(I,j)$ into ascending order.
\end{notation}

\begin{remark}
    The terminology ``face" and ``degeneracy" maps is borrowed from the terminology used in the simplicial category $\Delta$.
\end{remark}

% \begin{lemma}\label{lem:distributivitiyEquiv}
%     Let $u , x_1 , \dots , x_n$ be elements of a modular lattice $S$. Then the following are equivalent:
%     \begin{enumerate}
%         \item The set $u, x_1 , \dots , x_n$ is distributive.
%         \item The sets $x_1 , \dots , x_n$ and $x_1 \w u , \dots , x_n \w u$ are distributive and there is an equality
%         $$u \w \left( \vee_{i \in I} x_i \right) = \vee_{i \in I} (u \w x_i)$$
%         for all $I \subset [n]$.
%         \item The sets $x_1 , \dots , x_N$ and $x_1 \vee u , \dots , x_n \vee u$ are distributive and there is an equality
%         $$u \vee \left( \w_{i \in I} x_i \right) = \w_{i \in I} (u \vee x_i).$$
%     \end{enumerate}
% \end{lemma}

The following definition associates a lattice to any collection of $R$-submodules:

\begin{definition}
    Let $M$ be any $R$-module and $M_1 , \dots , M_n \subset M$ a collection of $R$-submodules. The collection of submodules $M_1 , \dots , M_n$ generates a lattice with operations
    $$M_i \w M_j := M_i \cap M_j, \quad M_i \vee M_j := M_i + M_j.$$
\end{definition}

The following construction introduces a fundamental family of complexes that can be associated to any collection of submodules. The terms of these complexes are built by taking ``intervals" above elements in the associated submodule lattices, and the exactness properties of these complexes will be very important in later sections.

\begin{construction}\label{cons:theDistributivityCx}
    Let $M$ be any $R$-module and $M_1 , \dots , M_n \subset M$ a collection of $R$-submodules. Given any subset $I \subset J$, let $\rho_{I,J}$ denote the canonical surjection 
    $$\rho_{I,J} : \frac{M}{\vee_{j \in I} M_j} \to \frac{M}{\vee_{j \in J} M_j}.$$
    Define the (cochain) complex $C^\bullet (M;M_1 , \dots , M_n)$ via
    $$C^i (M; M_1 , \dots , M_n) := \bigoplus_{|I| = i} \frac{M}{\vee_{j \in I} M_j},\quad \text{with differential}$$
    $$d^{C^\bullet}|_{\frac{M}{\vee_{i \in I} M_i}} := \sum_{j \notin I} \sgn(j,I) \rho_{I \cup j , I}.$$
    In the above, we use the convention that $C^0 (M ; M_1 , \dots , M_n) := M$. 

    Likewise, given any subset $I \subset J$, let $\iota_{J,I}$ denote the natural inclusion
    $$\iota_{J,I} : \wedge_{i \in J} M_i \to \wedge_{i \in I} M_i.$$
    Define the (chain) complex $C_\bullet (M ; M_1 , \dots , M_n)$ via
    $$C_i (M;M_1 , \dots , M_n) := \bigoplus_{|I|=i} \w_{j \in I} M_j, \quad \text{with differential}$$
    $$d^{C_\bullet}|_{\w_{i \in I} M_i} := \sum_{j \notin I} \sgn(j,I) \iota_{I,I\cup j}.$$

    If $I \subset [n]$ is any subset, define the complex $C^\bullet_I (M ; M_1 , \dots , M_n)$ by restricting to all direct summands of the form
    $$C^i_I (M ; M_1 , \dots , M_n) = \bigoplus_{\substack{I \subset J, \\ \ |J| - |I| = i}} \frac{M}{\vee_{j \in J} M_j}.$$
    Likewise, define the complex $C_\bullet^I (M ; M_1 , \dots , M_n)$  by restricting to all direct summands of the form
    $$C_i^I (M;M_1 , \dots , M_n) := \bigoplus_{\substack{I \subset J, \\  |J| - |I| =i}} \w_{j \in J} M_j.$$
\end{construction}

\begin{remark}
    Notice that by definition there are equalities 
    $$C_\bullet (M ; M_1 , \dots , M_n) = C^{\varnothing}_\bullet (M ; M_1 , \dots ,M_n),$$
    $$C^\bullet (M; M_1 , \dots , M_n) = C^\bullet_{\varnothing} (M ; M_1 , \dots , M_n).$$
\end{remark}

\begin{example}
If $M_1 , M_2 , M_3$ is a length $3$ collection of $R$-submodules of $M$, the associated complexes of Construction \ref{cons:theDistributivityCx} are explicitly given by
{\small
    % https://q.uiver.app/?q=WzAsMTgsWzAsMSwiQ19cXGJ1bGxldCAoTTsgTV8xICwgTV8yLCBNXzMpOiJdLFsxLDEsIk1fMSBcXGNhcCBNXzIgXFxjYXAgTV8zIl0sWzIsMCwiTV8xIFxcY2FwIE1fMiJdLFsyLDEsIk1fMSBcXGNhcCBNXzMiXSxbMiwyLCJNXzIgXFxjYXAgTV8zIl0sWzMsMCwiTV8xIl0sWzMsMSwiTV8yIl0sWzMsMiwiTV8zIl0sWzQsMSwiTSJdLFswLDQsIkNeXFxidWxsZXQgKE07IE1fMSAsIE1fMiAsIE1fMyk6Il0sWzEsNCwiTSJdLFsyLDMsIlxcZnJhY3tNfXtNXzF9Il0sWzIsNCwiXFxmcmFje019e01fMn0iXSxbMiw1LCJcXGZyYWN7TX17TV8zfSJdLFszLDMsIlxcZnJhY3tNfXtNXzErTV8yfSJdLFszLDQsIlxcZnJhY3tNfXtNXzErTV8zfSJdLFszLDUsIlxcZnJhY3tNfXtNXzIgKyBNXzN9Il0sWzQsNCwiXFxmcmFje019e01fMStNXzIrTV8zfSJdLFsxLDNdLFszLDZdLFs2LDhdLFsyLDMsIlxcYmlnb3BsdXMiLDEseyJzdHlsZSI6eyJib2R5Ijp7Im5hbWUiOiJub25lIn0sImhlYWQiOnsibmFtZSI6Im5vbmUifX19XSxbMyw0LCJcXGJpZ29wbHVzIiwxLHsic3R5bGUiOnsiYm9keSI6eyJuYW1lIjoibm9uZSJ9LCJoZWFkIjp7Im5hbWUiOiJub25lIn19fV0sWzUsNiwiXFxiaWdvcGx1cyIsMSx7InN0eWxlIjp7ImJvZHkiOnsibmFtZSI6Im5vbmUifSwiaGVhZCI6eyJuYW1lIjoibm9uZSJ9fX1dLFs2LDcsIlxcYmlnb3BsdXMiLDEseyJzdHlsZSI6eyJib2R5Ijp7Im5hbWUiOiJub25lIn0sImhlYWQiOnsibmFtZSI6Im5vbmUifX19XSxbMTAsMTJdLFsxMiwxNV0sWzE1LDE3XSxbMTEsMTIsIlxcYmlnb3BsdXMiLDEseyJzdHlsZSI6eyJib2R5Ijp7Im5hbWUiOiJub25lIn0sImhlYWQiOnsibmFtZSI6Im5vbmUifX19XSxbMTIsMTMsIlxcYmlnb3BsdXMiLDEseyJzdHlsZSI6eyJib2R5Ijp7Im5hbWUiOiJub25lIn0sImhlYWQiOnsibmFtZSI6Im5vbmUifX19XSxbMTQsMTUsIlxcYmlnb3BsdXMiLDEseyJzdHlsZSI6eyJib2R5Ijp7Im5hbWUiOiJub25lIn0sImhlYWQiOnsibmFtZSI6Im5vbmUifX19XSxbMTUsMTYsIlxcYmlnb3BsdXMiLDEseyJzdHlsZSI6eyJib2R5Ijp7Im5hbWUiOiJub25lIn0sImhlYWQiOnsibmFtZSI6Im5vbmUifX19XV0=
\[\begin{tikzcd}
	&& {M_1 \cap M_2} & {M_1} \\
	{C_\bullet (M; M_1 , M_2, M_3):} & {M_1 \cap M_2 \cap M_3} & {M_1 \cap M_3} & {M_2} & M \\
	&& {M_2 \cap M_3} & {M_3} \\
	&& {\frac{M}{M_1}} & {\frac{M}{M_1+M_2}} \\
	{C^\bullet (M; M_1 , M_2 , M_3):} & M & {\frac{M}{M_2}} & {\frac{M}{M_1+M_3}} & {\frac{M}{M_1+M_2+M_3}} \\
	&& {\frac{M}{M_3}} & {\frac{M}{M_2 + M_3}}
	\arrow[from=2-2, to=2-3]
	\arrow[from=2-3, to=2-4]
	\arrow[from=2-4, to=2-5]
	\arrow["\bigoplus"{description}, draw=none, from=1-3, to=2-3]
	\arrow["\bigoplus"{description}, draw=none, from=2-3, to=3-3]
	\arrow["\bigoplus"{description}, draw=none, from=1-4, to=2-4]
	\arrow["\bigoplus"{description}, draw=none, from=2-4, to=3-4]
	\arrow[from=5-2, to=5-3]
	\arrow[from=5-3, to=5-4]
	\arrow[from=5-4, to=5-5]
	\arrow["\bigoplus"{description}, draw=none, from=4-3, to=5-3]
	\arrow["\bigoplus"{description}, draw=none, from=5-3, to=6-3]
	\arrow["\bigoplus"{description}, draw=none, from=4-4, to=5-4]
	\arrow["\bigoplus"{description}, draw=none, from=5-4, to=6-4]
\end{tikzcd}\]}
In the above, note that $C_\bullet (M ; M_1, M_2 , M_3)$ is homologically indexed, whereas $C^\bullet (M ; M_1 , M_2 , M_3)$ is cohomologically indexed. By construction, both complexes have terms parametrized by the Boolean poset on $\{ 1 , 2 , 3 \}$. 
\end{example}

We can now define the notion of distributivity using the complexes of Construction \ref{cons:theDistributivityCx}.

\begin{definition}
    Let $M$ be an $R$-module and $M_1 , \dots , M_n$ a collection of $R$-submodules of $M$. The collection $M_1 , \dots , M_n$ is called \defi{distributive} if the complexes 
    $$C_\bullet^I (M ; M_1 , \dots , M_n) \quad \text{and} \quad C^\bullet_I (M ; M_1 , \dots , M_n)$$
    are exact in positive (co)homological degrees for all subsets $I \subset [n]$.
\end{definition}

\begin{remark}
    At first glance, this definition of distributive might seem quite different from the definition in terms of the submodule lattice generated by $M_1 , \dots , M_n$, but it turns out that these definitions are in fact equivalent by, for instance, \cite[Chapter 1, Proposition 7.2]{polishchuk2005quadratic} (the result here is stated over a field, but it holds over any commutative ring).
\end{remark}

The following is a trivial consequence of the inductive structure of the complexes $C^\bullet_I$ and $C^I_\bullet$:

\begin{obs}
    If $M_1 , \dots , M_n \subset M$ is a distributive collection, then every subcollection of $M_1 , \dots , M_n$ is distributive.
\end{obs}

Next, we introduce a collection of subquotients that may be associated to any collection of $R$-submodules. In later sections, we will see that all ribbon Schur functors arise as modules of this form for some appropriate collection of submodules:

\begin{definition}\label{def:theLModules}
    Let $M$ be an $R$-module and $M_1 , \dots , M_n$ a collection of $R$-submodules of $M$. Given an indexing set $I \subset [n]$, define the $R$-module
    $$L^I_{M_1 , \dots , M_n} := \frac{\w_{i \notin I} M_i}{\vee_{i \in I} M_i}.$$
\end{definition}

\begin{remark}
    By convention, we set
    $$L^{[n]}_{M_1 , \dots , M_n} := \frac{M}{M_1 + \cdots + M_n},$$
    $$L^{\varnothing}_{M_1 , \dots , M_n} := M_1 \cap \cdots \cap M_n.$$
    Moreover, notice that for the sake of conciseness of notation, the notation $N/L$ for two $R$-submodules $N,L \subset M$ is understood to mean $N / (L \cap N)$.
\end{remark}

The following observation is proved just by the definition of the complexes of Construction \ref{cons:theDistributivityCx} and the equivalence between the exactness of these complexes and distributivity:

\begin{obs}\label{obs:H0relatedToL}
    Let $M$ be an $R$-module and $M_1 , \dots , M_n$ any collection of $R$-submodules of $M$. For every $I \subset [n]$ the complexes
    $$C_\bullet^I (M ; M_1 , \dots , M_n) \quad \text{and} \quad C^\bullet_I (M ; M_1 , \dots , M_n)$$
    satisfy
    $$H_{0} (C_\bullet^I (M ; M_1 , \dots , M_n)) \twoheadrightarrow L^{[n] \backslash I}_{M_1 , \dots , M_n}, \quad \text{and}$$
    $$L^I_{M_1 , \dots , M_n} \hookrightarrow H^{0} (C^\bullet_I (M ; M_1 , \dots , M_n)) .$$
    If the collection $M_1 , \dots , M_n$ is distributive, then the above surjection/inclusion are equalities. 
\end{obs}

\begin{proof}
    By definition, the $0$th homology of $C_\bullet^I (M ; M_1 , \dots , M_n)$ is the cokernel of the map
    $$\bigoplus_{j \notin I} \w_{\ell \in I \cup j} M_\ell \to \w_{\ell \in I} M_\ell,$$
    which is precisely the quotient
    $$\frac{\w_{\ell \in I} M_\ell}{\sum_{j \notin I} \w_{\ell \in I \cup j} M_\ell }.$$
    There is always a containment
    $$\sum_{j \notin I} \w_{\ell \in I \cup j} M_\ell  \subset \left( \sum_{j \notin I} M_j \right) \cap \w_{\ell \in I} M_\ell,$$
    and thus there is a natural surjection $H_{0} (C_\bullet^I (M ; M_1 , \dots , M_n)) \twoheadrightarrow L^{[n] \backslash I}_{M_1 , \dots , M_n}$. If the collection $M_1 , \dots , M_n$ is distributive, then the above containment is an equality and thus the surjection is also an equality.

    Similarly, the $0$th cohomology of $C^\bullet_I (M ; M_1 , \dots , M_n)$ is the kernel of the map
    $$ \frac{M}{\vee_{\ell \in I} M_\ell} \to \bigoplus_{j \notin I} \frac{M}{\vee_{\ell \in I \cup j} M_\ell}.$$
    Again, there is always a containment
    $$\w_{j \notin I} \left( \vee_{\ell \in I \cup j} M_\ell \right) \supset \w_{j \notin I} M_j  + \vee_{\ell \in I} M_\ell $$
    and thus a natural inclusion $L^I_{M_1 , \dots , M_n} \hookrightarrow H^0 (C^\bullet_I (M ; M_1 , \dots , M_n))$. When the collection $M_1 , \dots , M_n$ is distributive the above containment is an equality in which case the inclusion is also an equality.
\end{proof}

The following proposition gives $3$ different equivalent conditions equivalent to distributivity. In particular, the exactness of $C_\bullet^I$ for all $I \subset [n]$ is equivalent to the exactness of $C^\bullet_I$, and this exactness is in turn equivalent to a family of short exact sequences. These short exact sequences will end up modeling the concatenation/near-concatenation sequences in later sections. 

Recall the notation $d_j$ and $s_j$ for face and degeneracy maps, respectively, of Notation \ref{not:faceDegen} in the statement of the following result:

\begin{prop}\label{prop:distributivityComplexes}
    Let $M$ be any $R$-module and $M_1 , \dots , M_n \subset M$ a collection of $R$-submodules. Then the following are equivalent:
    \begin{enumerate}
        \item For all nonempty $I \subset [n]$, the sequence
        $$0 \to L^{I \backslash j}_{M_1 , \dots , M_n} \to L^{s_j (I \backslash j)}_{M_1 , \dots , \widehat{M_j} , \dots , M_n} \to L^I_{M_1 , \dots , M_n} \to 0$$
        is exact.
        \item For all $I \subset [n]$, the cochain complex $C^\bullet_I (M ; M_1 , \dots , M_n)$ is exact in positive cohomological degrees.
        \item For all $I \subset [n]$, the chain complex $C_\bullet^I (M ; M_1 , \dots , M_n)$ is exact in positive homological degrees.
    \end{enumerate}
    In other words, to check distributivity it suffices to check exactness of \emph{either} the complex $C^I_\bullet (M; M_1 , \dots, M_n)$ or $C^\bullet_I (M; M_1 , \dots , M_n)$ for all $I \subset [n]$.
\end{prop}

\begin{remark}\label{rem:seqExactAtEnds}
    When $|I| = 1$ is a singleton set, the sequence above simply reads
    $$0 \to M_1 \to M \to M / M_1 \to 0.$$
    In general, the sequence
    $$0 \to L^{I \backslash j}_{M_1 , \dots , M_n} \to L^{s_j (I \backslash j)}_{M_1 , \dots , \widehat{M_j} , \dots , M_n} \to L^I_{M_1 , \dots , M_n} \to 0$$
    is a complex that is exact at the left and rightmost nontrivial terms by definition. This means that Proposition \ref{prop:distributivityComplexes} gives equivalent conditions for this sequence to be exact at the middle term.
\end{remark}

\begin{proof}
    Notice first that there are canonical short exact sequences of complexes (where $\widehat{-}$ denotes omission):
    \begin{equation}\label{eqn:SES1} 0 \to C^\bullet_{I} (M  ; M_1 , \dots , M_n )[-1] \qquad (\text{for} \ j \in I)\end{equation}
    $$\to C^\bullet_{I \backslash j} (M ; M_1 , \dots , M_n) \to C^\bullet_{s_j (I \backslash j)} (M ; M_1 , \dots, \widehat{M_j}, \dots , M_n) \to 0, \quad \text{and}$$
    \begin{equation}\label{eqn:SES2}0 \to C_\bullet^{s_j (I)} (M;M_1 , \dots , M_n) \to C_\bullet^I ( M ; M_1 , \dots , M_n)  \qquad (\text{for} \ j \notin I)\end{equation}
    $$\to  C_\bullet^{ I \cup j} (M ; M_1 , \dots , M_n)[-1] \to 0.$$

    \textbf{$(1) \iff (2)$:} By induction on $n$, we may assume that the collection $M_1 , \dots , \widehat{M_j} , \dots , M_n$ is distributive; in particular, the complex $C^\bullet_{s_j (I \backslash j)} (M ; M_1 , \dots, \widehat{M_j}, \dots , M_n)$ is exact in positive cohomological degrees, and moreover by a downward induction on $|I|$ (with base case $I = [n]$) we may also assume that $C^\bullet_{I} (M ; M_1 , \dots , M_n)$ is exact in positive cohomological degrees.
    
    Employing the long exact sequence of cohomology on the short exact sequence (\ref{eqn:SES1}) along with Observation \ref{obs:H0relatedToL} yields
    $$0 \to H^0 (C^\bullet_{I \backslash j} (M ; M_1 , \dots , M_n)) \to L^{s_j (I \backslash j)}_{M_1 , \dots , \widehat{M_j} , \dots , M_n} \to L^I_{M_1 , \dots , M_n} \to H^1 (C^\bullet_I (M ; M_1 , \dots , M_n)) $$
    $$\to H^1 (C^\bullet_{I \backslash j} (M ; M_1 , \dots , M_n)) \to H^1 (C^\bullet_{s_j (I \backslash j)} (M ; M_1 , \dots, \widehat{M_j}, \dots , M_n)) \to  \cdots.$$
    The inductive hypothesis immediately implies that
    $$H^i (C^\bullet_I (M ; M_1 , \dots , M_n)) = 0 \quad \text{for all} \ i > 1,$$
    and there is an exact sequence
    $$0 \to H^0 (C^\bullet_{I \backslash j} (M ; M_1 , \dots , M_n)) \to L^{s_j (I \backslash j)}_{M_1 , \dots , \widehat{M_j} , \dots , M_n} \to L^I_{M_1 , \dots , M_n} \to H^1 (C^\bullet_I (M ; M_1 , \dots , M_n)) \to 0.$$
    Assuming $(2)$, we may employ Observation \ref{obs:H0relatedToL} to deduce the short exact sequence of $(1)$. On the other hand, assuming $(1)$ there is a commutative diagram with exact rows:
    % https://q.uiver.app/#q=WzAsMTAsWzAsMCwiMCJdLFswLDEsIjAiXSxbMSwwLCJMXntJIFxcYmFja3NsYXNoIGp9X3tNXzEgLCBcXGRvdHMgLCBNX259Il0sWzIsMCwiTF57c19qIChJIFxcYmFja3NsYXNoIGopfV97TV8xICwgXFxkb3RzICwgXFx3aWRlaGF0e01fan0gLCAgXFxkb3RzICwgTV9ufSJdLFszLDAsIkxeSV97TV8xICwgXFxkb3RzICwgTV9ufSJdLFs0LDAsIjAiXSxbMSwxLCJIXjAgKENeXFxidWxsZXRfSSAoTSA7IE1fMSAsIFxcZG90cyAsIE1fbikpIl0sWzQsMSwiSF4xIChDXlxcYnVsbGV0X0kgKE0gOyBNXzEgLCBcXGRvdHMgLCBNX24pKSJdLFszLDEsIkxeSV97TV8xICwgXFxkb3RzICwgTV9ufSJdLFsyLDEsIkxee3NfaiAoSSBcXGJhY2tzbGFzaCBqKX1fe01fMSAsIFxcZG90cyAsIFxcd2lkZWhhdHtNX2p9ICwgIFxcZG90cyAsIE1fbn0iXSxbMCwyXSxbMSw2XSxbMiw2LCIiLDEseyJzdHlsZSI6eyJ0YWlsIjp7Im5hbWUiOiJob29rIiwic2lkZSI6InRvcCJ9fX1dLFsyLDNdLFszLDldLFszLDRdLFs0LDVdLFs4LDddLFs0LDhdLFs2LDldLFs5LDhdLFs1LDddXQ==
\[\begin{tikzcd}
	0 & {L^{I \backslash j}_{M_1 , \dots , M_n}} & {L^{s_j (I \backslash j)}_{M_1 , \dots , \widehat{M_j} ,  \dots , M_n}} & {L^I_{M_1 , \dots , M_n}} & 0 \\
	0 & {H^0 (C^\bullet_I (M ; M_1 , \dots , M_n))} & {L^{s_j (I \backslash j)}_{M_1 , \dots , \widehat{M_j} ,  \dots , M_n}} & {L^I_{M_1 , \dots , M_n}} & {H^1 (C^\bullet_I (M ; M_1 , \dots , M_n))}
	\arrow[from=1-1, to=1-2]
	\arrow[from=2-1, to=2-2]
	\arrow[hook, from=1-2, to=2-2]
	\arrow[from=1-2, to=1-3]
	\arrow[from=1-3, to=2-3]
	\arrow[from=1-3, to=1-4]
	\arrow[from=1-4, to=1-5]
	\arrow[from=2-4, to=2-5]
	\arrow[from=1-4, to=2-4]
	\arrow[from=2-2, to=2-3]
	\arrow[from=2-3, to=2-4]
	\arrow[from=1-5, to=2-5]
\end{tikzcd}\]
    Since the inner two maps are isomorphisms, both of the outer two inclusions are isomorphisms and thus $C^\bullet_{I \backslash j} (M ; M_1 , \dots , M_n)$ is also exact in positive degree for all $j \in I$. It follows that $C^\bullet_I (M ; M_1 , \dots , M_n)$ is exact in positive cohomological degrees for all $I$ if and only if the sequence of $(1)$ is exact for all $I$.

    \textbf{$(1) \iff (3)$:} This proof proceeds similarly: by induction on $n$, we may assume that the collection $M_1 , \dots , \widehat{M_j} , \dots , M_n$ is distributive and $C_\bullet^{s_j (I)} (M ; M_1 , \dots  ,  M_n)$ is exact in positive homological degrees, and likewise by a downward induction on $|I|$ we may assume that $H_1 (C_\bullet^{I \cup j} (M;M_1 , \dots , M_n))$ is exact in positive homological degrees.
    
    Employing the long exact sequence of homology on the second short exact sequence (\ref{eqn:SES2}) along with Observation \ref{obs:H0relatedToL} yields
    $$\cdots \to H_1 (C_\bullet^{ I \cup j} (M ; M_1 , \dots , M_n)) \to H_1 (C_\bullet^{s_j (I)} (M;M_1 , \dots , M_n))$$
    $$\to H_1 (C_\bullet^I ( M ; M_1 , \dots , M_n)) \to L^{[n] \backslash (I \cup j)}_{M_1 , \dots , M_n} \to L^{[n-1] \backslash s_j (I )}_{M_1 , \dots , \widehat{M_j} , \dots , M_n} \to L^{[n] \backslash I}_{M_1 , \dots , M_n} \to 0.$$
    Set $I' := [n] \backslash I$ and notice that 
    $$[n] \backslash (I \cup j) = I' \backslash j, \quad \text{and} \quad  [n-1] \backslash s_j (I) = s_j ([n] \backslash (I \cup j)) = s_j (I' \backslash j).$$
    Again, this at least implies that $H_i (C_\bullet^I ( M ; M_1 , \dots , M_n)) = 0$ for $i>1$ and there is an exact sequence
    $$0 \to H_1 (C_\bullet^I ( M ; M_1 , \dots , M_n)) \to L^{I' \backslash j}_{M_1 , \dots , M_n} \to L^{s_j (I' \backslash j)}_{M_1 , \dots , \widehat{M_j} , \dots , M_n} \to H_0 (C_\bullet^I (M ; M_1 , \dots , M_n) \to 0.$$
    The result thus follows by identical reasoning as in $(1) \iff (2)$.
\end{proof}

% \begin{proof}
%     Proceed by induction on $|I|$, with base case $|I| = 1$ being handled by the short exact sequence
%     $$0 \to M_i \to M \to M / M_i \to 0.$$
%     Let $|I| > 1$ and write $I = (i_1 < \cdots < i_k)$. The collection of submodules $M_{i_1}  , \dots , M_{i_k}$ generates a distributive lattice and the complex $C^\bullet (M ; M_{i_1} , \dots , M_{i_k})$ has no cohomology in positive degrees. Moreover, by the inductive hypothesis each term of this complex is a projective $R$-module, whence the statement follows from Observation \ref{obs:rightResByProj}.
% \end{proof}

The following lemma shows that the property of distributivity ``dualizes well". This will be essential to prove the appropriate analogs of Koszul duality.

\begin{lemma}\label{lem:dualDistributivity}
    Let $M$ be a flat $R$-module and $M_1 , \dots , M_n \subset M$ a sequence of submodules with $M / M_i$ a flat $R$-module for all $1 \leq i \leq n$. For a submodule $N \subset M$, use the notation $N^\vee := (M/N)^* := \hom_R (M/N , R)$. Then:
    \begin{enumerate}
        \item Assume that $M /\sum_{i \in I} M_i$ is a flat $R$-module for all $I \subset [n]$. Then
        $$\text{the collection} \ M_1 , \dots , M_n \subset M \  \text{is distributive}$$ 
    $$\iff$$
    $$\text{the collection} \ M_1^\vee , \dots , M_n^\vee \subset M^* \  \text{is distributive}$$
     (notice that each $M_i^\vee$ may be viewed as a submodule of $M^*$ by dualizing the surjection $M \to M/M_i$). In this case, each of the modules $L^I_{M_1 , \dots , M_n}$ are $R$-flat.
     \item Assume that $M_1 , \dots , M_n$ is a distributive collection. Then there are isomorphisms
    $$\left( \frac{M}{\vee_{i \in I} M_i} \right)^* \cong \w_{i \in I} M_i^\vee, \quad \text{and}$$
    $$\left( \w_{i \in I} M_i \right)^* \cong \frac{M^*}{\vee_{i \in I} M_i^\vee},$$
    where $I \subset [n]$.
    \item Assume that $M_1 , \dots , M_n$ is a distributive collection. Then there are isomorphisms of complexes
    $$C^\bullet_I (M ; M_1 , \dots , M_n)^* \cong C_\bullet^I (M^* ; M_1^\vee , \dots , M_n^\vee), \quad \text{and}$$
    $$C_\bullet^I (M ; M_1 , \dots , M_n)^* \cong C^\bullet_I (M^* ; M_1^\vee , \dots , M_n^\vee).$$
    In particular, for all $I \subset [n]$ there is an isomorphism
    $$\left( L^I_{M_1 , \dots , M_n} \right)^* \cong L^{[n] \backslash I}_{M_1^\vee , \dots , M_n^\vee}.$$
    \end{enumerate}
\end{lemma}

\begin{proof}
The progression of the proof actually follows by first proving $(2)$ and $(3)$, then noting that $(1)$ is an immediate consequence of $(3)$. 

\textbf{Proof of (2):} Proceed by induction on $n$, where the base case is for $n=1$. In this case, the short exact sequence
$$0 \to M_i \to M \to M/ M_i \to 0$$
implies that $M_i$ is a flat $R$-module for all $i$, and hence dualizing yields the short exact sequence
$$0 \to (M / M_i)^* := M_i^\vee \to M^* \to M_i^* \to 0.$$
This implies that $M_i^* = M^* / M_i^\vee$, yielding the base case.

Let $n > 1$ and consider the complexes $C^\bullet (M; M_1 , \dots , M_n)$ and $C_\bullet (M ; M_1 , \dots , M_n)$. Each term of these complexes falls within the inductive hypothesis, in which case we may dualize to obtain the isomorphisms
    $$C^\bullet (M ; M_1 , \dots , M_n)^* \cong C_\bullet (M^* ; M_1^\vee , \dots , M_n^\vee), \quad \text{and}$$
    $$C_\bullet (M ; M_1 , \dots , M_n)^* \cong C^\bullet (M^* ; M_1^\vee , \dots , M_n^\vee).$$
    Taking $0$th (co)homology of each of the above complexes and employing Proposition \ref{prop:distributivityComplexes} yields the isomorphisms
    $$\left( \frac{M}{\vee_{i \in [n]} M_i} \right)^* \cong \w_{i \in [n]} M_i^\vee, \quad \text{and}$$
    $$\left( \w_{i \in [n]} M_i \right)^* \cong \frac{M^*}{\vee_{i \in [n]} M_i^\vee}.$$

\textbf{Proof of (3):} Since each term of the complex $C^\bullet_I (M ; M_1 , \dots , M_n)$ is of the form $M/\vee_{i \in J} M_i$ for some subset $J \subset [n]$, dualizing and using part $(2)$ implies that $C^\bullet_I(M ; M_1 , \dots , M_n)^*$ has terms $\w_{i \in J} M_i^\vee$, and it is clear that the dualized differentials are the same as those of $C^I_\bullet (M^* ; M_1^\vee , \dots , M_n^\vee)$. The proof for $C_\bullet^I (M; M_1 , \dots , M_n)$ is identical and the isomorphism $\left( L^I_{M_1 , \dots , M_n} \right)^* \cong L^{[n] \backslash I}_{M_1^\vee , \dots , M_n^\vee}$ follows upon taking $0$th (co)homology and using Proposition \ref{prop:distributivityComplexes}. 

\textbf{Proof of (1):} Assume that $M_1 , \dots , M_n$ is distributive. The assumption that each quotient $M / \sum_{i \in I} M_i$ is flat implies that $\w_{i \in I} M_i$ is flat for all $I = (i_1 , \dots , i_k) \subset [n]$, since each of the augmented complexes
$$\w_{i \in I} M_i \to C^\bullet (M ; M_{i_1} , \dots , M_{i_k})$$
is exact and by assumption $C^\bullet (M ; M_{i_1} , \dots , M_{i_k})$ is a complex of flat $R$-modules. Thus $\w_{i \in I} M_i$ is flat by Observation \ref{obs:rightResByProj}.

Using this, it follows that for every $I \subset [n]$ the chain complex $C_\bullet^I (M ; M_1 , \dots , M_n)$ is exact in positive homological degrees and the $0$th homology is a flat $R$-module. This means that the dual $C_\bullet^I (M; M_1 , \dots , M_n)^*$ is a cochain complex with cohomology concentrated in degree $0$, and by the isomorphism of $(3)$ combined with Proposition \ref{prop:distributivityComplexes}, the collection $M_1^\vee , \dots , M_n^\vee$ is distributive. This argument is inherently symmetric in the roles of $M_i$ and $M_i^\vee$, whence the result follows.
\end{proof}

\begin{cor}\label{cor:projectivesInDistr}
    Let $M$ be any projective $R$-module and $M_1 , \dots , M_n \subset M$ a distributive collection of $R$-submodules. Assume that 
    $$M /\sum_{i \in I} M_i$$
    is a projective $R$-module for all $I \subset [n]$. Then for all subsets $I \subset [n]$, the $R$-module
    $$L^I_{M_1 , \dots , M_n}$$
    is a projective $R$-submodule.
\end{cor}

\begin{proof}
    By $(3)$ of Lemma \ref{lem:dualDistributivity}, each of the modules $L^I_{M_1 , \dots , M_n}$ has a right resolution by projective $R$-modules given by the complex $C^\bullet_I (M;M_1 , \dots , M_n)$. By Observation \ref{obs:rightResByProj}, the $R$-module $L^I_{M_1 , \dots , M_n}$ must itself be projective.
\end{proof}

\section{A Generalization of Schur Modules for Ribbon Diagrams}\label{sec:ribbonSchurs}

In this section, we introduce ribbon Schur functors associated to arbitrary Koszul algebra (and module) inputs. We show that the exactness of the concatenation/near-concatenation sequence is actually equivalent to Koszulness and establish a set of properties generalizing many well-known properties for classically-defined Schur functors corresponding to ribbon diagrams.

\subsection{Standard Operations Between Compositions}\label{subsec:compositionOps}

Before defining ribbon Schur functors, it will be helpful to establish multiple conventions and define certain natural operations on the Boolean/refinement posets. These operations are standard in the combinatorial literature, but for the sake of establishing conventions, we define things explicitly with examples here.

\begin{definition}[Operations on Compositions]\label{def:operationsOfComps}
    Given any integer $d > 0$, a \defi{composition} of $d$ is a tuple $\alpha = (\alpha_1 , \dots, \alpha_k)$ with $\alpha_i > 0$ for each $1 \leq i \leq k$ and $|\alpha| := \alpha_1 + \cdots + \alpha_k = d$. The integer $k$ is the \defi{length} of $\alpha$, denoted $\ell (\alpha)$. 

    Given two compositions $\alpha$ and $\beta$, the \defi{concatenation} of $\alpha$ and $\beta$, denoted $\alpha \cdot \beta$, is defined as the composition
    $$\alpha \cdot \beta := (\alpha_1 , \dots , \alpha_k , \beta_1 , \dots , \beta_j).$$
    The \defi{near-concatenation} of $\alpha$ and $\beta$, denoted $\alpha \odot \beta$, is defined as the composition
    $$\alpha \odot \beta := (\alpha_1 , \dots , \alpha_{k-1} , \alpha_k + \beta_1 , \beta_2 , \dots , \beta_j).$$
    Finally, given any integer $d > 0$, the notation $\alpha^{(d)}$ is defined to be the composition
    $$\alpha^{(d)} := (d \alpha_1 , \dots , d \alpha_k).$$
\end{definition}

\begin{definition}[Ribbon Diagrams Associated to Compositions]
     Given any composition $\alpha$, one can associate the \defi{ribbon diagram} by building a diagram whose row lengths (read from bottom to top) are given by $\alpha_1 , \dots , \alpha_n$, and such that consecutive rows have overlap size precisely $1$. 

    The \defi{transpose} of a composition, denoted $\alpha^t$, is the composition obtained by transposing the ribbon diagram associated to $\alpha$.
\end{definition}

% \begin{remark}
%     \keller{any comment about the terminology here}
% \end{remark}

\begin{example}
    The ribbon diagram associated to the composition $(3,1,1,2,4)$ is the diagram
    $$
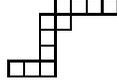
\ytableausetup{boxsize=0.5em}
\begin{ytableau} 
\none& \none& \none & & & &\\ 
 \none& \none&  & \\  
 \none& \none&   \\  
 \none& \none&    \\  
   & &    \\  
\end{ytableau}
$$
The transpose of this shape is given by the shape
$$\ytableausetup{boxsize=0.5em}
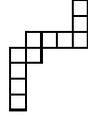
\begin{ytableau} 
\none & \none & \none & \none & \\
\none & \none & \none & \none & \\
\none & & & & \\
 &  \\
  \\
  \\
  \\
\end{ytableau}$$
and hence, upon reading the row lengths from bottom to top we find
$$(3,1,1,2,4)^t = (1,1,1,2,4,1,1).$$
Likewise, if $\alpha = (1,1,2)$ and $\beta = (2,3,1)$ then
$$\alpha \cdot \beta = 
\ytableausetup{boxsize=0.5em}
\begin{ytableau} 
 \, &  \\
  \\
  \\
\end{ytableau} 
\cdot
\begin{ytableau}
\none & \none & \none & \\
\none & & & \\
 & \\
\end{ytableau}
=
\begin{ytableau} 
\none & \none & \none & \none & \\
\none & \none & & & \\
\none & & \\
 & \\
 \\
 \\
\end{ytableau} \quad \text{and}$$
$$\alpha \odot \beta = 
\ytableausetup{boxsize=0.5em}
\begin{ytableau} 
 \, &  \\
  \\
  \\
\end{ytableau} 
\cdot
\begin{ytableau}
\none & \none & \none & \\
\none & & & \\
 & \\
\end{ytableau}
=
\begin{ytableau}
\none & \none & \none & \none & \none & \\
\none & \none & \none & & & \\
 & & & \\
  \\
   \\
\end{ytableau}$$
With the ribbon diagrams, we see that concatenation corresponds to ``stacking" the diagrams, and near concatenation corresponds to merging the diagrams along their rows.
\end{example}

\begin{definition}\label{def:booleanAndRefinement}
    Let $n \in \bbn$ be any nonnegative integer and $C(n)$ denote the set of compositions of $n$ into positive parts. Let $[n] := \{ 1 , \dots , n \}$ with the convention that $[0] := \varnothing$, and let $[a,b] := \{ a , a+1 , \dots , b \}$. The set $2^{[n-1]}$ is a poset with the standard Boolean poset structure, and $C(n)$ is also a poset with the standard refinement poset structure. Whenever the notation $\alpha \leq \beta$ is used for two compositions $\alpha$ and $\beta$, the partial order $\leq$ is understood to be the refinement order.

    There is moreover a standard isomorphism of posets
    $$\phi : 2^{[n-1]} \to C(n).$$
    This isomorphism may be given explicitly as follows: given $I \subset [n-1]$, write $I = [a_1 , b_1] \cup [a_1 , b_2] \cup \cdots \cup [a_k,b_k]$ for $a_i < b_i$ and $b_i < a_{i+1} - 1$ for each $i$. Then
    $$\phi (I) = (1^{a_1-1} , b_1-a_1+1, 1^{a_2-b_1-1} , b_2-a_2+1, \dots , b_k-a_k+1,n-b_k).$$
\end{definition}

\begin{example}
The following example shows the Boolean poset $2^{\{1,2,3 \}}$ and the refinement poset $C(4)$ side by side. The map $\phi$ is even a morphism of distributive lattices.
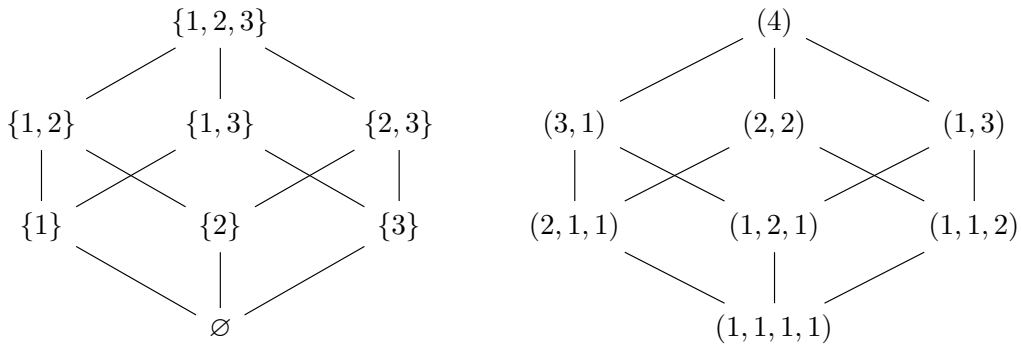
\[\begin{tikzcd}
	& {\{ 1 , 2 , 3 \}} &&& {(4)} \\
	{\{ 1 , 2 \}} & {\{ 1 , 3 \}} & {\{2 , 3 \}} & {(3,1)} & {(2,2)} & {(1,3)} \\
	{\{1\}} & {\{2 \}} & {\{ 3 \}} & {(2,1,1)} & {(1,2,1)} & {(1,1,2)} \\
	& \varnothing &&& {(1,1,1,1)}
	\arrow[no head, from=4-2, to=3-1]
	\arrow[no head, from=4-2, to=3-2]
	\arrow[no head, from=4-2, to=3-3]
	\arrow[no head, from=3-1, to=2-1]
	\arrow[no head, from=3-1, to=2-2]
	\arrow[no head, from=2-2, to=3-3]
	\arrow[no head, from=3-2, to=2-3]
	\arrow[no head, from=2-3, to=3-3]
	\arrow[no head, from=2-1, to=3-2]
	\arrow[no head, from=2-1, to=1-2]
	\arrow[no head, from=1-2, to=2-2]
	\arrow[no head, from=1-2, to=2-3]
	\arrow[no head, from=4-5, to=3-4]
	\arrow[no head, from=4-5, to=3-5]
	\arrow[no head, from=3-6, to=4-5]
	\arrow[no head, from=2-5, to=3-4]
	\arrow[no head, from=2-5, to=3-6]
	\arrow[no head, from=2-6, to=3-5]
	\arrow[no head, from=2-4, to=3-5]
	\arrow[no head, from=1-5, to=2-4]
	\arrow[no head, from=1-5, to=2-5]
	\arrow[no head, from=1-5, to=2-6]
	\arrow[no head, from=2-4, to=3-4]
	\arrow[no head, from=2-6, to=3-6]
\end{tikzcd}\]
\end{example}

\begin{obs}\label{obs:transposeProps}
    Let $\alpha$ and $\beta$ be two compositions and $I \subset [n]$. Then:
    $$(\alpha \cdot \beta)^t = \beta^t \odot \alpha^t, \quad (\alpha \odot \beta)^t = \beta^t \cdot \alpha^t, \quad \text{and}$$
    $$\operatorname{rev} ( \phi (I)^t ) = \phi( [n] \backslash I),$$
    where $\operatorname{rev}$ denotes the \defi{reversal} operator, which simply reverses the order of the entries of a composition.
\end{obs}

\begin{definition}[Partitioned Compositions]
    Let $\alpha = (\alpha_1 , \dots , \alpha_n)$ be a composition of some integer $d$. A \defi{partition} of $\alpha$ is any choice of decomposition of $\alpha$ as a concatenation of sub-compositions of $\alpha$. The composition $\alpha$ is $\ell$-partitioned if $\alpha$ is endowed with a partition that decomposes $\alpha$ into $\ell$ parts.

    Given an $\ell$-partitioned composition $\alpha$, the notation $p_i (\alpha)$ denotes the $i$th piece of partition of $\alpha$. In other words, every $\ell$-partitioned composition $\alpha$ may be written as the concatenation
    $$\alpha = p_1 (\alpha) \cdot p_2 (\alpha) \cdots p_\ell (\alpha).$$
\end{definition}

\begin{remark}
    We will employ a minor abuse of notation when dealing with partitioned compositions, since the chosen partition will often not be specified. More precisely, the data of a partitioned composition includes both the data of a composition $\alpha$ and a chosen partition $P$, which can be encoded by any chosen subset of $[\ell (\alpha) - 1]$. The word ``composition" without any adjective will only refer to a composition as defined in Definition \ref{def:operationsOfComps}.
\end{remark}

\begin{conv}
    By convention, any $\ell$-partitioned composition $\alpha$ may be viewed as a $j$-partitioned composition for any $j < \ell$ by concatenating the last $\ell - j$ pieces of $\alpha$. In other words, if
    $$\alpha = p_1 (\alpha) \cdot \cdots \cdot p_\ell (\alpha),$$
    then $\alpha$ may be viewed as the $j$-partitioned composition
    $$\alpha = p_1 (\alpha) \cdot \cdots p_{j-1} (\alpha) \cdot \left( p_j (\alpha) \cdot p_{j+1} (\alpha) \cdots p_\ell (\alpha) \right).$$
    This convention will become essential when we deal with multi-Schur modules, as defined later.
\end{conv}

\begin{example}\label{ex:partnExample}
    Let $\alpha = (2,2,1,4,3)$ and consider the $4$-partition of $\alpha$ induced by the subset $\{ 1,3,4\} \subset [4]$. This decomposes $\alpha$ as the concatenation
    $$(2) \cdot (2,1) \cdot (4) \cdot (3).$$
    Via the convention of \ref{conv:concatenationConventions}, $\alpha$ may also be viewed as the $3$-partitioned composition
    $$(2) \cdot (2,1) \cdot (4,3).$$
\end{example}

\begin{definition}\label{def:fundamentalOperationsForPtnComps}
    Let $\alpha$ be an $\ell$-partitioned composition and let $I = \{ i_1 < \cdots < i_k \} \subset [\ell-1 ]$ be any subset.
    \begin{enumerate}
        \item The notation $\sigma_I (\alpha)$ denotes the composition obtained by near-concatenating $p_j (\alpha)$ and $p_{j+1} (\alpha)$ for every $j \in I$. 
        \item The notation $\nu_I (\alpha)$ denotes the ribbon diagram obtained by disconnecting $p_j (\alpha)$ and $p_{j+1} (\alpha)$ for every $j \in I$. 
        \item The notation $\mu_I (\alpha)$ denotes the composition $\nu_{[\ell-1-|I|]} \circ \sigma_I (\alpha)$.
    \end{enumerate}
\end{definition}

\begin{example}
    Let $\alpha := (2) \cdot (2,1) \cdot (4) \cdot (3)$ be the $4$-partitioned shape of Example \ref{ex:partnExample}. As a ribbon diagram:
    $$\alpha = \ytableausetup{boxsize=0.5em}
\begin{ytableau} 
\none & \none & \none & \none & \none & & & \\
\none & \none & & & & \\
\none & \none & \\
\none & & \\
 & \\
\end{ytableau}$$
Let $I :=  \{ 1 ,3 \}$. Then:
$$\sigma_I (\alpha) = (4,1) \cdot (7) \quad \text{(viewed as 2-partitioned)}$$
$$= \ytableausetup{boxsize=0.5em}
\begin{ytableau} 
\none & \none & \none & & & & & & & \\
\none & \none & \none & \\
 & & & \\
\end{ytableau}$$
$$\nu_I (\alpha) = \ytableausetup{boxsize=0.5em}
\begin{ytableau} 
\none & \none & \none & \none & \none & \none & \none & & & \\
\none & \none & \none  & & & & \\
\none & \none & \none  & \\
\none & \none & & \\
 & \\
\end{ytableau},$$
$$\mu_I (\alpha) = \nu_{\{ 1 \}} (\sigma_I (\alpha)) $$
$$= \ytableausetup{boxsize=0.5em}
\begin{ytableau} 
\none & \none & \none & \none & & & & & & & \\
\none & \none & \none & \\
 & & & \\
\end{ytableau}$$
\end{example}

\begin{remark}
    The operation $\mu_I$ can be reformulated in words as the operation that near-concatenates all elements $p_i (\alpha)$ and $p_{i+1} (\alpha)$ for $i \in I$, then disconnects everything else.
\end{remark}

\begin{conv}\label{conv:concatenationConventions}
    Let $\alpha$ and $\beta$ be $j$ and $k$-partitioned compositions, respectively. By convention, the concatenation $\alpha \cdot \beta$ will be viewed as a $j+k$-partitioned composition with
    $$p_s (\alpha \cdot \beta ) = \begin{cases}
    p_s (\alpha) & \text{if} s \leq j, \\
    p_{s-j} (\beta) & \text{if} \ s > j. 
    \end{cases}$$
    Likewise, the near-concatenation $\alpha \odot \beta$ will be viewed as a $j+k-1$-partitioned composition with
    $$p_s (\alpha \odot \beta) = \begin{cases}
    p_s (\alpha) & \text{if} \ s < k, \\
    p_{k} (\alpha) \odot p_{1} (\beta) & \text{if} \ s = k, \\
    p_{s-j+1} (\beta) & \text{if} \ s > k. 
    \end{cases}$$
\end{conv}

\begin{definition}\label{def:gradedCompsConventions}
    Let $A$ be any quadratic $R$-algebra. The notation $A_d$ for any integer $d \in \bbz$ denotes the degree $d$ component of $A$. Given a tuple $\alpha = (\alpha_1 , \dots , \alpha_n)$, use the notation
    $$A_\alpha := A_{\alpha_1} \otimes_R \cdots \otimes A_{\alpha_n}.$$
    Likewise, given a quadratic left (resp. right) $A$-module $N$ (resp. $M$), use the notation
$$(M \otimes_R A)_{\alpha} := M_{\alpha_1} \otimes_R A_{\alpha_2} \otimes_R \cdots \otimes_R A_{\alpha_n},$$
$$(A \otimes_R N)_\alpha := A_{\alpha_1} \otimes_R \cdots \otimes_R A_{\alpha_{n-1}} \otimes_R N_{\alpha_n}.$$
We use the convention that $(M \otimes_R A)_{\alpha} = M_{\alpha_1}$ and $(A \otimes_R N)_{\alpha} := N_{\alpha_1}$ if $\ell (\alpha) = 1$. 

Similarly, for $\ell (\alpha) \geq 2$, use the notation
$$(M \otimes_R A \otimes_R N)_\alpha := M_{\alpha_1} \otimes_R A_{\alpha_2} \otimes_R \cdots \otimes_R A_{\alpha_{n-1}} \otimes_R N_{\alpha_n},$$
where we use the convention that $(M \otimes_R A \otimes_R N)_\alpha = M_{\alpha_1} \otimes_R N_{\alpha_2}$ if $\ell (\alpha) = 2$.
\end{definition}

\subsection{Ribbon Schur Modules for Koszul Algebras}\label{subsec:ribbonsForAlg}

In this subsection, we will hold off on proving the statements until the next subsection (where the statements will be proved in more generality). This subsection will mainly be used to illustrate the construction of ribbon Schur functors and their various properties with examples. 

\begin{definition}
    Let $A$ be a Koszul $A$-module and $\alpha$ any composition. Define the \defi{ribbon Schur module} $\bbs^{\alpha}_A$ as the kernel of the natural map
$$\bbs^{\alpha}_A = \ker \left( A_\alpha \to \bigoplus_{\alpha \lessdot \beta} A_\beta \right).$$
where $\lessdot$ denotes the covering relation in the refinement poset.
\end{definition}

\begin{remark}
    Since the ribbon Schur module $\bbs^{\alpha}_A$ is defined only using the algebra structure on $A$, this definition is canonical and is well-defined for any Koszul algebra $A$ (in fact, it is well-defined for any algebra, but the Koszul property will ensure uniformity in the properties of these objects).
\end{remark}

\begin{example}
    Let $\alpha = (3,2,1,3)$. Then in the refinement poset, $\alpha$ is covered by the compositions
    $$(5,1,3), \quad (3,3,3), \quad \text{and} \quad (3,2,4).$$
    Thus for any Koszul algebra $\bbs_{A}^{(3,2,1,3)}$ is defined to be the kernel of the map
    $$A_3 \otimes_R A_2 \otimes_R A_1 \otimes_R A_3 \to \begin{matrix}
    A_5 \otimes_R A_1 \otimes_R A_3 \\
    \oplus \\
    A_3 \otimes_R A_3 \otimes_R A_3 \\
    \oplus \\
    A_3 \otimes_R A_2 \otimes_R A_4 \\
    \end{matrix}.$$
\end{example}

The following proposition justifies usage of the terminology ``ribbon Schur functor":

\begin{prop}\label{prop:theFunctoriality}
    Formation of the ribbon Schur module $\bbs^{\alpha}_A$ is functorial in the ring argument; that is, given a morphism of $R$-algebras $\psi : A \to B$, there is an induced morphism
    $$\bbs^{\alpha}_\psi : \bbs^{\alpha}_A \to \bbs^{\alpha}_B.$$
\end{prop}

\begin{example}
    When $A = S(V)$, the symmetric algebra on a free $R$-module $V$, the ribbon Schur module coincides with the classical Schur module definition corresponding to the skew ribbon shape defined by $\alpha$.

    Likewise, for $A = \bigwedge^\bullet V$ the exterior algebra on some free $R$-module, the ribbon Schur module coincides with the \defi{Weyl} module associated to the ribbon diagram of $\alpha$. 

    The functoriality of the classical Schur and Weyl functors is precisely the functoriality of Proposition \ref{prop:theFunctoriality}, and there are isomorphisms of functors
    $$\bbs^{\alpha}_{S(-)} = \bbs^{\alpha} (-), \quad \bbs^{\alpha}_{\bigwedge^\bullet (-)} = \bbw^{\alpha} (-).$$
\end{example}

\begin{example}
    When $A = T(V)$, the tensor algebra on some flat $R$-module $V$, the ribbon Schur modules are particularly simple:
    $$\bbs_{T(V)}^\alpha = \begin{cases} V^{\otimes \alpha_1} & \text{if} \ \ell (\alpha) = 1, \\
    0 & \text{otherwise}. \end{cases}$$
\end{example}

\begin{example}
    If $\alpha = (1^i)$ for some integer $i \geq 1$, then by definition
    $$\bbs^{(1^i)}_A = (A^!)^*_i.$$
    Thus the Priddy complex may be reformulated as the complex
    $$\cdots \to A \otimes_R \bbs^{(1^{i})}_A \to A \otimes_R \bbs^{(1^{i-1})}_A \to \cdots \to A \otimes_R \bbs^{(1)}_A \to A \to 0.$$
\end{example}

% We will instead prove the dual assertion that
% $$\bbs^{\sigma(\alpha)^t}_A = \coker \left( \bigoplus_{\alpha \lessdot \beta} (A^!)^*_\beta \to (A^!)^*_\alpha \right).$$
% Proceed by induction on the number of columns, with base case being given by the $2$ column case. In this case, suppose that $\alpha = (a,b)$; the claim reduces to showing
% $$\bbs^{\sigma (a,b)^t}_A = \coker ( (A^!)^*_{a+b} \to (A^!)^*_a \otimes_R (A^!)^*_b ).$$
% Notice that the composition
% $$(A^!)^*_{a+b} \to (A^!)^*_a \otimes_R (A^!)^*_b \xrightarrow{d^{\sigma(a,b)^t}_A} \bbs^{\sigma(a,b)^t}_A$$
% is evidently $0$, which implies that $(A^!)^*_{a+b} \subset \ker d^{\sigma (a,b)^t}_A$. On the other hand, the $A$-Schur module $\bbs^{\sigma(a,b)^t}_A$ may be written as a quotient of $(A^!)^*_{a-1} \otimes_R A_2 \otimes_R (A^!)^*_{b-1}$, which implies that
% $$\ker (d^{\sigma(a,b)^t}_A) \subset \left( (A^!)^*_a \otimes_R (A^!)^*_b \right) \cap \left( (A^!)^*_{a-1} \otimes_R (A^!)^*_2 \otimes_R (A^!)^*_{b-1} \right)$$
% $$= (A^!)^*_{a+b}.$$

% Now proceed inductively. It is again evident that the composition
% $$\bigoplus_{\alpha \lessdot \beta} (A^!)^*_\beta \to (A^!)^*_{\alpha} \xrightarrow{d^{\sigma(\alpha)^t}_A} \bbs^{\sigma(\alpha)}_A$$
% is $0$. By \keller{obs} there are canonical surjections
% $$(A^!)^*_{\alpha_1} \otimes_R \bbs^{\sigma(\alpha_2 , \dots , \alpha_n)^t}_A \twoheadrightarrow \bbs^{\sigma (\alpha)^t}_A,$$
% $$\bbs^{\sigma(\alpha_1 , \dots , \alpha_{n-1})^t}_A \otimes_R (A^!)^*_{D_n} \twoheadrightarrow \bbs^{\sigma (\alpha)^t}_A.$$
% By induction the result follows \keller{finish this}

\begin{lemma}\label{lem:SESandDuality}
Let $A$ be a Koszul $R$-algebra. Then:
\begin{enumerate}
    \item For all compositions $\alpha$, the ribbon Schur module $\bbs^{\alpha}_{A}$ is a flat $R$-module.
    \item For any two compositions $\alpha$ and $\beta$ there is a canonical short exact sequence of $R$-modules
$$0 \to \bbs^{\alpha \cdot \beta}_{A} \to \bbs^{\alpha}_{A} \otimes_R \bbs^{\beta}_{A} \to \bbs^{\alpha \odot \beta}_{A} \to 0.$$
Moreover, this sequence is exact for all compositions if and only if $A$ is a Koszul algebra.
\item There is a canonical isomorphism of $R$-modules
$$\left( \bbs^{\alpha}_A \right)^* \cong \bbs^{\alpha^t}_{A^!}.$$
\end{enumerate}
If $A$ is assumed $R$-projective, then the ribbon Schur module $\bbs^{\alpha}_A$ is $R$-projective for all compositions $\alpha$. 
\end{lemma}

\begin{remark}
    As mentioned in the Introduction, in the theory of symmetric functions the short exact sequence of (2) is an explicit realization of the so-called \defi{concatenation/near-concatenation} identity.
\end{remark}

\begin{remark}
    Notice that Observation \ref{obs:transposeProps} implies that the sequences
    $$0 \to \bbs^{\alpha \cdot \beta}_{A} \to \bbs^{\alpha}_{A} \otimes_R \bbs^{\beta}_{A} \to \bbs^{\alpha \odot \beta}_{A} \to 0 \quad \text{and}$$
    $$0 \to \bbs^{\beta^t \cdot \alpha^t}_{A^!} \to \bbs^{\beta^t}_{A^!} \otimes_R \bbs^{\alpha^t}_{A^!} \to \bbs^{\beta^t \odot \alpha^t}_{A^!} \to 0$$
    are naturally dual to each other.
\end{remark}

\begin{example}
    Let $\bbs^\alpha (-)$ and $\bbw^\beta (-)$ denote the classical Schur and Weyl functors corresponding to the skew ribbon shapes defined by the compositions $\alpha$ and $\beta$, respectively. Then by definition, for any free $R$-module $V$ there are equalities
    $$\bbs^\alpha (V) = \bbs^\alpha_{S(V)}, \quad \bbw^{\alpha^t} (V^*) = \bbs^{\alpha}_{\bigwedge^\bullet V^*}.$$
    The duality mentioned in the statement of Lemma \ref{lem:SESandDuality} is thus a generalization of the well-known isomorphism
    $$\bbs^{\alpha} (V)^* \cong \bbw^{\alpha^t} (V^*).$$
\end{example}

Next, we introduce a class of complexes whose terms are tensor products of ribbon Schur functors; this complex generalizes the short exact sequence of Lemma \ref{lem:SESandDuality} (2); more defining this complex we make a simple observation:

\begin{obs}
    Recall the operator $\mu_I$ as in Definition \ref{def:fundamentalOperationsForPtnComps}. Let $A$ be a Koszul algebra and $\alpha$ any $\ell$-partitioned composition. Then for all $I \subset J$ there is a natural surjection
    $$\rho_{I,J}: \bbs^{\mu_I (\alpha)}_A \to \bbs^{\mu_J (\alpha)}_A.$$
    This surjection is defined by taking the tensor products of the surjections as in the right map of the short exact sequence of Lemma \ref{lem:SESandDuality} (as dictated by the subsets $I$ and $J$).
\end{obs}

\begin{definition}\label{def:AlgeHGComplex}
\rm
Let $\alpha$ be any $\ell$-partitioned composition. Define the cochain complex $(\cat{H}_A (\alpha),\delta)$
whose $i^{th}$ term is 
$$
\cat{H}^i_A(\alpha )
:=\bigoplus_{\substack{I \subseteq [\ell-1]:\\ |I|=i}} \bbs^{\mu_I (\alpha) }_A, \quad \text{with differential}
$$
$$d^{\cat{H}_A}|_{\bbs^{\mu_I (\alpha)}} := \sum_{j \notin I} \sgn (j,I) \rho_{I , I \cup j}.$$
\end{definition}

\begin{theorem}
For all $\ell$-partitioned compositions $\alpha$, the complex $\cat{H}_A (\alpha)$ is a cochain complex with
$$
H^0(\cat{H}^i_A(\alpha)) \cong
\bbs^{\alpha}_A,
$$
and which is exact in positive cohomological degrees. If $A$ also has a compatible action by a group $G$, then the complex $\cat{H} (\alpha)$ is also $G$-equivariant.
\end{theorem}

\begin{remark}
 The complexes $\cat{H}^\bullet_A (\alpha)$ for an $\ell$-partitioned composition $\alpha$ may be viewed as a categorification of the \defi{Hamel-Goulden identities} (see \cite{hamel1995planar}) associated to the (horizontal) ribbon decomposition induced by the partitioning data of $\alpha$.
\end{remark}

\begin{example}\label{ex:HamelGoulExample}
    Let $\alpha = (3,1,2,2,4,1,5)$, viewed as a $4$-partitioned composition with $p_1 (\alpha) = (3,1)$, $p_2 (\alpha) = (2,2)$, $p_3(\alpha) = (4)$, and $p_4 (\alpha) = (1,5)$. Then the complex $\cat{H}_A (\alpha)$ takes the form
    % https://q.uiver.app/?q=WzAsOCxbMSwwLCJcXGJic157KDMsMSl9X0EgXFxvdGltZXNfUiBcXGJic157KDIsMil9X0EgXFxvdGltZXNfUiBcXGJic157KDQpfV9BIFxcb3RpbWVzX1IgXFxiYnNeeygxLDUpfV9BIl0sWzEsMSwiIFxcYmJzXnsoMywxKX1fQSBcXG90aW1lc19SIFxcYmJzXnsoMiw2KX1fQSBcXG90aW1lc19SIFxcYmJzXnsoMSw1KX1fQSAiXSxbMSwyLCIgIFxcYmJzXnsoMywzLDIpfV9BIFxcb3RpbWVzX1IgXFxiYnNeeyg1LDUpfV9BICJdLFsxLDMsIlxcYmJzXnsoMywzLDcsNSl9X0EiXSxbMCwxLCJcXGJic157KDMsMywyKX1fQSBcXG90aW1lc19SIFxcYmJzXnsoNCl9X0EgXFxvdGltZXNfUiBcXGJic157KDEsNSl9X0EiXSxbMiwxLCJcXGJic157KDMsMSl9X0EgXFxvdGltZXNfUiBcXGJic157KDIsMil9X0EgXFxvdGltZXNfUiBcXGJic157KDUsNSl9X0EiXSxbMCwyLCJcXGJic157KDMsMyw2KX1fQSBcXG90aW1lc19SIFxcYmJzXnsoMSw1KX1fQSJdLFsyLDIsIlxcYmJzXnsoMywxKX1fQSBcXG90aW1lc19SIFxcYmJzXnsoMiw3LDUpfV9BIl0sWzQsMSwiXFxiaWdvcGx1cyIsMSx7InN0eWxlIjp7ImJvZHkiOnsibmFtZSI6Im5vbmUifSwiaGVhZCI6eyJuYW1lIjoibm9uZSJ9fX1dLFsxLDUsIlxcYmlnb3BsdXMiLDEseyJzdHlsZSI6eyJib2R5Ijp7Im5hbWUiOiJub25lIn0sImhlYWQiOnsibmFtZSI6Im5vbmUifX19XSxbNiwyLCJcXGJpZ29wbHVzIiwxLHsic3R5bGUiOnsiYm9keSI6eyJuYW1lIjoibm9uZSJ9LCJoZWFkIjp7Im5hbWUiOiJub25lIn19fV0sWzIsNywiXFxiaWdvcGx1cyIsMSx7InN0eWxlIjp7ImJvZHkiOnsibmFtZSI6Im5vbmUifSwiaGVhZCI6eyJuYW1lIjoibm9uZSJ9fX1dLFswLDFdLFsxLDJdLFsyLDNdXQ==
\[\begin{tikzcd}
	& {\bbs^{(3,1)}_A \otimes_R \bbs^{(2,2)}_A \otimes_R \bbs^{(4)}_A \otimes_R \bbs^{(1,5)}_A} \\
	{\bbs^{(3,3,2)}_A \otimes_R \bbs^{(4)}_A \otimes_R \bbs^{(1,5)}_A} & { \bbs^{(3,1)}_A \otimes_R \bbs^{(2,6)}_A \otimes_R \bbs^{(1,5)}_A } & {\bbs^{(3,1)}_A \otimes_R \bbs^{(2,2)}_A \otimes_R \bbs^{(5,5)}_A} \\
	{\bbs^{(3,3,6)}_A \otimes_R \bbs^{(1,5)}_A} & {  \bbs^{(3,3,2)}_A \otimes_R \bbs^{(5,5)}_A } & {\bbs^{(3,1)}_A \otimes_R \bbs^{(2,7,5)}_A} \\
	& {\bbs^{(3,3,7,5)}_A}
	\arrow["\bigoplus"{description}, draw=none, from=2-1, to=2-2]
	\arrow["\bigoplus"{description}, draw=none, from=2-2, to=2-3]
	\arrow["\bigoplus"{description}, draw=none, from=3-1, to=3-2]
	\arrow["\bigoplus"{description}, draw=none, from=3-2, to=3-3]
	\arrow[from=1-2, to=2-2]
	\arrow[from=2-2, to=3-2]
	\arrow[from=3-2, to=4-2]
\end{tikzcd}\]
\end{example}

\subsection{Ribbon Schur Modules with Koszul Module Inputs}\label{subsec:ribbonsForMods}

In this section, we generalize the ribbon Schur functors to allow for Koszul module inputs as well. The same type of concatenation/near-concatenation sequence may be used to detect Koszulness of modules. In the following, recall the notation established in Definition \ref{def:gradedCompsConventions}.

\begin{definition}\label{def:RibbonSchur}
    Let $\alpha$ be any composition and recall the conventions of Definition \ref{def:gradedCompsConventions}. Given a Koszul left (resp. right) $A$-module $M$ with initial degree $t$, define the ribbon Schur module $\bbs^\alpha_{A,M}$ (resp. $\bbs^\alpha_{M,A}$) as the kernel of the natural map
$$\bbs^\alpha_{A,M} := \ker \left( (A \otimes_R M)_{\alpha \cdot (t)} \to \bigoplus_{\alpha \cdot (t) \lessdot \beta} (A \otimes_R M)_\beta \right),$$
$$\text{resp.} \quad \bbs^\alpha_{M,A} := \ker \left( (M \otimes_R A)_{(t) \cdot \alpha} \to \bigoplus_{(t) \cdot \alpha \lessdot \beta} (M \otimes_R A)_\beta \right).$$
Given a Koszul left (resp. right) $A$-module $N$ (resp. $M$) of initial degree $s$ (resp. $t$), define the ribbon Schur module $\bbs^\alpha_{M,A,N}$ as the kernel of the natural map
$$\bbs^\alpha_{M,A,N} := \ker \left( (M \otimes_R A \otimes_R N)_{(t) \cdot \alpha \cdot (s)} \to \bigoplus_{(t) \cdot \alpha \cdot (s) \lessdot \beta} (M \otimes_R A \otimes_R N)_\beta \right).$$
\end{definition}

\begin{remark}
    Notice that if $\alpha$ is the empty partition, we use the convention that
    $$\bbs^{()}_{A,M} = M_t, \quad \bbs^{()}_{M,A,N} = M_t \otimes_R N_s.$$
\end{remark}

\begin{notation}\label{not:disconnectedRibbons}
    The definition of $\bbs^{\alpha}_{M,A,N}$ may be extended to disconnected ribbon diagrams by using the convention that disconnected portions of the diagram correspond to tensor products of the respective Schur modules. We will use this convention for the remainder of the paper, since it drastically simplifies notational issues in the following proofs/constructions.
\end{notation}

\begin{example}
    Let $\alpha := (3,1,1)$ and $\beta = (2,2)$. Then by the convention of Notation \ref{not:disconnectedRibbons}, there is an equality
    $${\begin{matrix} \\ \bbs_A \end{matrix}}^{ \ytableausetup{boxsize=0.25em}
\begin{ytableau} 
\none & \none & \none & \none & & \\
\none & \none & \none & & \\
\none & \none & \\
\none & \none & \\
 & & \\
\end{ytableau}} = 
{\begin{matrix} \quad \ \ {\ytableausetup{boxsize=0.25em}
\begin{ytableau}
\none & \none & \\
\none & \none & \\
 & & \\
\end{ytableau}} \\ \bbs_A \end{matrix}} 
{\begin{matrix} \qquad \ \ {\ytableausetup{boxsize=0.25em}
\begin{ytableau}
 \none & & \\
  & \\
\end{ytableau}}  \\ \otimes_R \  \bbs_A \end{matrix}} $$
$$=\bbs^{(3,1,1)}_A \otimes_R \bbs^{(2,2)}_A.
$$
\end{example}

\begin{example}
    Let $A$ be any Koszul algebra and $\alpha$ a composition. Recall that every truncation of $A$ is Koszul by Corollary \ref{cor:truncationsAreKoszul}, whence $A_+^r$ is a Koszul (left and right) $A$-module for all $r \geq 0$. By definition, there are equalities
    $$\bbs^{\alpha}_{A_+^r , A , A_+^{r'}} = \bbs^{(r) \cdot \alpha \cdot (r')}_{A}.$$
\end{example}

The following definition is an evident extension of Definition \ref{def:KoszulSubmodcollection} to allow for both left and right $A$-modules: 

\begin{definition}\label{def:doubleModuleDef}
    Given a Koszul left (resp. right) $A$-module $N$ (resp. $M$) with initial degrees $s$ and $t$, respectively, define the collection $S_{M,A,N,1}^n , \dots , S_{M,A,N,n-1}^n \subset M_t \otimes_R A_1^{\otimes n-1} \otimes_R N_s$ for $n \geq 2$ via
    $$S_{M,A,N,i}^n := \begin{cases}
    Q_{t+1}^M \otimes_R A_1^{\otimes n-2} \otimes_R N_s & \text{if} \ i = 1, \\
    Q_t^M \otimes_R A_1^{\otimes i-3} \otimes_R Q_2^A \otimes_R A_1^{\otimes n-i+1} \otimes_R N_s & \text{if} \ 2 \leq i \leq n-2, \\
    M_t \otimes_R A_1^{\otimes n-2} \otimes_R Q_{s+1}^N & \text{if} \ i= n-1. 
    \end{cases}$$
    If $n = 1$, use the convention that $S_{M,A,N,0}^1 := 0$. 
\end{definition}

Recall the notation of $\mu_I$ for a subset $I$ as established in Definition \ref{def:fundamentalOperationsForPtnComps} in the following:

\begin{definition}\label{def:MuITermsForAM}
    Let $A$ be a Koszul algebra and $M$ any left $A$-module of initial degree $s$. Given any $\ell$-partitioned composition $\alpha$, view the concatenation $\alpha \cdot (s)$ as $(\ell+1)$-partitioned via Convention \ref{conv:concatenationConventions}. Then, for any subset $I \subset [\ell]$ write $\mu_I (\alpha \cdot (s)) = \alpha^1 \cup \cdots \alpha^{\ell - |I|}$ as a disjoint union of compositions $\alpha^1 , \dots , \alpha^{\ell  -|I|}$, where $\alpha^{\ell - |I|}$ has length $g$. Then we define
    $$\bbs^{\mu_I (\alpha)}_{A, M} := \bbs^{\alpha^1}_A \otimes_R \bbs^{\alpha^2}_A \otimes_R \cdots \otimes_R \bbs^{\alpha^{\ell - |I|}_{\leq g-1}}_{A , M_{\geq \alpha^{\ell - |I|}_g + s}}.$$
\end{definition}

As a quick example let $\alpha = (2,1)\cdot(3,4)\cdot(4,3,3)$ be a $3$-partitioned composition and $I = \{ 1, 3 \} \subset [3]$. Then for any Koszul algebra $A$ and left $A$-module $M$ of initial degree $s$, there is an equality
$$\bbs^{\mu_I (\alpha)}_{A , M} := \bbs^{(2,4,4)}_A \otimes_R \bbs^{(4,3)}_{A , M_{\geq 3 + s}}.$$

\begin{remark}
    The slight technicality in defining $\bbs^{\mu_I (\alpha)}_{A,M}$ when taking account of a module input stems from the fact that $M$ has its own initial degree, in which case we should be defining the corresponding ribbons with respect to the composition $\alpha \cdot (s)$. 
\end{remark}

The following observation is an immediate consequence of the identification of the ribbon Schur functors with the modules of Definition \ref{def:theLModules}, noted in Observation \ref{obs:SchursAsLs}.

\begin{obs}
    Let $A$ be a Koszul algebra and $M$ any left $A$-module. Given an $\ell$-partitioned composition $\alpha$ and subsets $I \subset J \subset [\ell]$, there is a canonical morphism of $R$-modules
    $$\bbs^{\mu_I (\alpha)}_A \to \bbs^{\mu_J (\alpha)}_A.$$
\end{obs}

The following observation is a direct consequence of the definition of a ribbon Schur functor, but will be very useful later on:

\begin{obs}\label{obs:intersectingSchurs}
    Let $\alpha$, $\beta$, and $\gamma$ be any (possibly empty) compositions. Given a Koszul algebra $A$ and a Koszul left (resp. right) $A$-module $N$ (resp. $M$), there is an equality
    $$\bbs^{\alpha \cdot \beta \cdot \gamma}_{M , A , N} = \left( \bbs^{\alpha \cdot \beta}_{M,A} \otimes_R \bbs^{\gamma}_{A,N} \right) \cap \left( \bbs^{\alpha}_{M,A} \otimes_R \bbs^{\beta \cdot \gamma}_{A,N} \right),$$
    where the intersection is being viewed as taking place in $(M \otimes_R A \otimes_R N)_{\alpha \cdot \beta \cdot \gamma}$.
\end{obs}

Finally, we are able to tie ribbon Schur functors to the modules introduced in Definition \ref{def:theLModules}. This reformulation combined with the equivalence of Proposition \ref{prop:distributivityComplexes} will yield quick proofs of the concatenation/near-concatenation formulation of Koszulness.

\begin{obs}\label{obs:SchursAsLs}
    Let $\alpha$ be any composition of length $\ell$. Given a Koszul algebra $A$ and a Koszul left $A$-module $M$, there is an isomorphism of $R$-modules
    $$\bbs^{\alpha}_{A,M} = L^{\phi^{-1} (\alpha)}_{S_{A,M,1}^n , \dots , S_{A,M,n-1}^n},$$
where the map $\phi$ is the isomorphism of posets of Definition \ref{def:booleanAndRefinement} and the module $L^{I}_{M_1 , \dots , M_n}$ is from Definition \ref{def:theLModules}. The analogous statement for right Koszul modules also holds.

Likewise, given a Koszul left (resp. right) $A$-module $N$ (resp. $M$), there is an equality
$$\bbs^{\alpha}_{M,A,N} = L^{\phi^{-1} (\alpha)}_{S_{M,A,N,1}^n , \dots , S_{M,A,N,n-1}^n}.$$
\end{obs}

We arrive at the statement and proof of the relevant properties for ribbon Schur functors associated to Koszul algebras/modules; again, the proof here is deceptively short, but the proof combines all of the machinery developed thus far. Recall the notation for the reversal operator $\rev$ from Observation \ref{obs:transposeProps}.

\begin{lemma}\label{lem:SESandDualityMods}
    Let $A$ be a Koszul $R$-algebra and $M$ (resp. $N$) a Koszul right (resp. left) $A$-module. Then:
\begin{enumerate}
    \item For all compositions $\alpha$, the ribbon Schur module $\bbs^{\alpha}_{A,M}$ (resp. $\bbs^{\alpha}_{N,A}$) is a flat $R$-module. If $A$ and $M$ are $R$-projective, then $\bbs^{\alpha}_{A,M}$ (resp. $\bbs^{\alpha}_{N,A}$) is also $R$-projective.
    \item For any two compositions $\alpha$ and $\beta$ there is a short exact sequence of $R$-modules
$$0 \to \bbs^{\alpha \cdot \beta}_{A,M} \to \bbs^{\alpha}_{A} \otimes_R \bbs^{\beta}_{A,M} \to \bbs^{\alpha \odot \beta}_{A,M} \to 0$$
$$\text{resp.} \quad 0 \to \bbs^{\alpha \cdot \beta}_{N,A} \to \bbs^{\alpha}_{N,A} \otimes_R \bbs^{\beta}_{A} \to \bbs^{\alpha \odot \beta}_{N,A} \to 0.$$
Conversely, if the sequences 
$$0 \to \bbs^{\alpha \cdot \beta}_{A,M_{\geq d}} \to \bbs^{\alpha}_{A} \otimes_R \bbs^{\beta}_{A,M_{\geq d}} \to \bbs^{\alpha \odot \beta}_{A,M_{\geq d}} \to 0$$
$$\text{resp.} \quad 0 \to \bbs^{\alpha \cdot \beta}_{N_{\geq d},A} \to \bbs^{\alpha}_{N_{\geq d},A} \otimes_R \bbs^{\beta}_{A} \to \bbs^{\alpha \odot \beta}_{N_{\geq d},A} \to 0.$$
are exact for all integers $d \geq 0$ and compositions $\alpha, \beta$, then $M$ (resp. $N$) is a Koszul left (resp. right) $A$-module.
\item There is a canonical isomorphism of $R$-modules
$$\left( \bbs^{\alpha}_{A,M} \right)^* \cong \bbs^{\rev(\alpha^t) }_{M^!,A^!}$$
$$\text{resp.} \quad \left( \bbs^{\alpha}_{N,A} \right)^* \cong \bbs^{\rev(\alpha^t)}_{A^!,N^!}.$$
\end{enumerate}
\end{lemma}

\begin{remark}
    There is a convention here that is important to take note of in the short exact sequence of $(2)$: if the composition $\beta$ is the empty composition, write $\alpha = \alpha' \cdot (\alpha_n)$; the short exact sequence $(2)$ then reads
    $$0 \to \bbs^{\alpha}_{A,M} \to \bbs^{\alpha}_{A} \otimes_R M_t \to \bbs^{\alpha'}_{A,M_{\geq t + \alpha_n}} \to 0.$$
    There is a similar convention in the right module case: if $\alpha$ is empty, write $\beta = (\beta_1) \cdot \beta'$. Then the short exact sequence reads
    $$0 \to \bbs^{\beta}_{N,A} \to N_t \otimes_R \bbs^{\beta}_A \to \bbs^{\beta'}_{N_{\geq t + \beta_1} , A} \to 0.$$
    Notice moreover that the reason Lemma \ref{lem:SESandDuality}(3) does not have the reversal $\rev(\alpha^t)$ is because there is by construction a canonical isomorphism $\bbs^{\alpha^t}_A \cong \bbs^{\rev(\alpha^t)}_A$ for any composition $\alpha$. 
\end{remark}

\begin{proof}
    \textbf{Proof of (1):} This just combines Lemma \ref{lem:dualDistributivity}(1), Theorem \ref{thm:koszulModuleDistr}, and Observation \ref{obs:SchursAsLs}.

    \textbf{Proof of (2):} This is precisely the short exact sequence of Proposition \ref{prop:distributivityComplexes}, so again the result follows from Theorem \ref{thm:koszulModuleDistr} combined with Observation \ref{obs:SchursAsLs}.

    \textbf{Proof of (3):} This just combines Lemma \ref{lem:dualDistributivity}(3), Theorem \ref{thm:koszulModuleDistr}, and Observation \ref{obs:SchursAsLs}.
\end{proof}

The case of having double module inputs in the associated Schur functor has to be treated separately, since the distributivity criterion cannot be used directly:

\begin{cor}\label{cor:doubleModuleInputProperties}
Let $A$ be a Koszul $R$-algebra and $N$ (resp. $M$) any left (resp. right) Koszul $A$-module. Then:
\begin{enumerate}
    \item For all compositions $\alpha$, the ribbon Schur module $\bbs^{\alpha}_{M,A,N}$ is a flat $R$-module. If $A$, $M$, and $N$ are $R$-projective, then $\bbs^{\alpha}_{M,A,N}$ is $R$-projective.
    \item For any two compositions $\alpha$ and $\beta$ there is a short exact sequence of $R$-modules
$$0 \to \bbs^{\alpha \cdot \beta}_{M,A,N} \to \bbs^{\alpha}_{M,A} \otimes_R \bbs^{\beta}_{A,N} \to \bbs^{\alpha \odot \beta}_{M,A,N} \to 0.$$
\item There is a canonical isomorphism of $R$-modules
$$\left( \bbs^{\alpha}_{M,A,N} \right)^* \cong \bbs^{\rev(\alpha^t)}_{N^!,A^!,M^!}.$$
\end{enumerate}
\end{cor}

\begin{proof}
Notice that $(3)$ follows from Lemma \ref{lem:dualDistributivity}(3) combined with Observation \ref{obs:SchursAsLs}.

\textbf{Proof of (2):} As noted in Remark \ref{rem:seqExactAtEnds}, the sequence
$$0 \to \bbs^{\alpha \cdot \beta}_{M,A,N} \to \bbs^{\alpha}_{M,A} \otimes_R \bbs^{\beta}_{A,N} \to \bbs^{\alpha \odot \beta}_{M,A,N} \to 0$$
is exact at the left and rightmost terms, so it suffices to prove exactness for the middle term. Assume first that $\ell(\alpha) = 1$ and $\ell (\beta) = 0$. Then there is a commutative diagram:
% https://q.uiver.app/?q=WzAsMTgsWzEsMSwiXFxiYnNeeyhcXGFscGhhXzEpfV97TSxBLE59Il0sWzIsMSwiXFxiYnNeeyhcXGFscGhhXzEpfV97TSxBfSBcXG90aW1lc19SIE5fe3N9Il0sWzMsMSwiTV90IFxcb3RpbWVzX3IgTl97cytcXGFscGhhXzF9Il0sWzEsMiwiTV90IFxcb3RpbWVzX1IgXFxiYnNeeyhcXGFscGhhXzEpfV97QSxOfSJdLFsyLDIsIk1fdCBcXG90aW1lc19SIEFfe1xcYWxwaGFfMX0gXFxvdGltZXNfUiBOX3MiXSxbMywyLCJNX3QgXFxvdGltZXNfciBOX3tzK1xcYWxwaGFfMX0iXSxbMSwzLCJNX3t0K1xcYWxwaGFfMX0gXFxvdGltZXNfUiBOX3MiXSxbMiwzLCJNX3t0K1xcYWxwaGFfMX0gXFxvdGltZXNfUiBOX3MiXSxbMCwxLCIwIl0sWzAsMiwiMCJdLFswLDMsIjAiXSxbMSw0LCIwIl0sWzIsNCwiMCJdLFs0LDIsIjAiXSxbNCwxLCIwIl0sWzMsMCwiMCJdLFsyLDAsIjAiXSxbMSwwLCIwIl0sWzAsMV0sWzEsMl0sWzE1LDJdLFsxNiwxXSxbMTcsMF0sWzIsMTRdLFs1LDEzXSxbMiw1LCIiLDAseyJsZXZlbCI6Miwic3R5bGUiOnsiaGVhZCI6eyJuYW1lIjoibm9uZSJ9fX1dLFs2LDcsIiIsMCx7ImxldmVsIjoyLCJzdHlsZSI6eyJoZWFkIjp7Im5hbWUiOiJub25lIn19fV0sWzAsM10sWzEsNF0sWzQsNV0sWzgsMF0sWzksM10sWzEwLDZdLFs0LDddLFszLDZdLFs2LDExXSxbNywxMl0sWzMsNF1d
\[\begin{tikzcd}
	& 0 & 0 & 0 \\
	0 & {\bbs^{(\alpha_1)}_{M,A,N}} & {\bbs^{(\alpha_1)}_{M,A} \otimes_R N_{s}} & {M_t \otimes_R N_{s+\alpha_1}} & 0 \\
	0 & {M_t \otimes_R \bbs^{(\alpha_1)}_{A,N}} & {M_t \otimes_R A_{\alpha_1} \otimes_R N_s} & {M_t \otimes_R N_{s+\alpha_1}} & 0 \\
	0 & {M_{t+\alpha_1} \otimes_R N_s} & {M_{t+\alpha_1} \otimes_R N_s} \\
	& 0 & 0
	\arrow[from=2-2, to=2-3]
	\arrow[from=2-3, to=2-4]
	\arrow[from=1-4, to=2-4]
	\arrow[from=1-3, to=2-3]
	\arrow[from=1-2, to=2-2]
	\arrow[from=2-4, to=2-5]
	\arrow[from=3-4, to=3-5]
	\arrow[Rightarrow, no head, from=2-4, to=3-4]
	\arrow[Rightarrow, no head, from=4-2, to=4-3]
	\arrow[from=2-2, to=3-2]
	\arrow[from=2-3, to=3-3]
	\arrow[from=3-3, to=3-4]
	\arrow[from=2-1, to=2-2]
	\arrow[from=3-1, to=3-2]
	\arrow[from=4-1, to=4-2]
	\arrow[from=3-3, to=4-3]
	\arrow[from=3-2, to=4-2]
	\arrow[from=4-2, to=5-2]
	\arrow[from=4-3, to=5-3]
	\arrow[from=3-2, to=3-3]
\end{tikzcd}\]
The middle row and column is exact by Lemma \ref{lem:SESandDualityMods}(2), and the last row and column are evidently exact. Let $\Theta$ denote the surjection
$$\Theta : \bbs^{(\alpha_1)}_{M,A} \otimes_R N_s \to M_t \otimes_R N_{s + \alpha_1}.$$
A quick diagram chase shows that
$$\ker (\Theta) \subset \left( \bbs^{(\alpha_1)}_{M,A} \otimes_R N_{s} \right) \cap \left( M_t \otimes_R \bbs^{(\alpha_1)}_{A,N} \right),$$
and by Observation \ref{obs:intersectingSchurs} this intersection is precisely the Schur module $\bbs^{(\alpha_1)}_{M,A,N}$. Since the reverse inclusion evidently holds, the top row is exact.

Assume now that $\ell (\alpha), \ell(\beta) > 0$. Write $\beta = \beta' \cdot (\beta_k)$ for some composition $\beta'$ with $\ell (\beta') = \ell(\beta) - 1$. Then there is a commutative diagram:

% https://q.uiver.app/?q=WzAsMjEsWzEsMSwiXFxiYnNee1xcYWxwaGFcXGNkb3QgXFxiZXRhfV97TSxBLE59Il0sWzIsMSwiXFxiYnNee1xcYWxwaGF9X3tNLEF9IFxcb3RpbWVzX1IgXFxiYnNee1xcYmV0YX1fe0EsTn0iXSxbMywxLCJcXGJic157XFxhbHBoYSBcXG9kb3QgXFxiZXRhfV97TSxBLE59Il0sWzEsMiwiXFxiYnNee1xcYWxwaGEgXFxjZG90IFxcYmV0YSd9X3tNLEF9IFxcb3RpbWVzX1IgXFxiYnNeeyhcXGJldGFfayl9X3tBLE59Il0sWzIsMiwiXFxiYnNee1xcYWxwaGF9X3tNLEF9IFxcb3RpbWVzX1IgXFxiYnNee1xcYmV0YSd9X0EgXFxvdGltZXNfUiBcXGJic157KFxcYmV0YV9rKX1fe0EsTn0iXSxbMywyLCJcXGJic157XFxhbHBoYSBcXG9kb3QgXFxiZXRhJ31fe00sQX0gXFxvdGltZXNfUiBcXGJic157KFxcYmV0YV9rKX1fe0EsTn0iXSxbMSwzLCJcXGJic157XFxhbHBoYSBcXGNkb3QgXFxiZXRhJyBcXG9kb3QgKFxcYmV0YV9rKX1fe00sQSxOfSJdLFsyLDMsIlxcYmJzXntcXGFscGhhfV97TSxBfSBcXG90aW1lc19SIFxcYmJzXntcXGJldGEnXFxvZG90IChcXGJldGFfayl9X3tBLk59Il0sWzMsMywiXFxiYnNee1xcYWxwaGEgXFxvZG90IFxcYmV0YScgXFxvZG90IChcXGJldGFfayl9X3tNLEEsTn0iXSxbMSwwLCIwIl0sWzIsMCwiMCJdLFszLDAsIjAiXSxbNCwxLCIwIl0sWzQsMiwiMCJdLFs0LDMsIjAiXSxbMyw0LCIwIl0sWzIsNCwiMCJdLFsxLDQsIjAiXSxbMCwzLCIwIl0sWzAsMiwiMCJdLFswLDEsIjAiXSxbMjAsMF0sWzksMF0sWzEwLDFdLFsxMSwyXSxbMCwxXSxbMSwyXSxbMiwxMl0sWzIsNV0sWzEsNF0sWzAsM10sWzE5LDNdLFszLDRdLFs0LDVdLFs3LDhdLFs4LDE0XSxbNSwxM10sWzUsOF0sWzgsMTVdLFs3LDE2XSxbNCw3XSxbMyw2XSxbMTgsNl0sWzYsN10sWzYsMTddXQ==
\[\begin{tikzcd}
	& 0 & 0 & 0 \\
	0 & {\bbs^{\alpha\cdot \beta}_{M,A,N}} & {\bbs^{\alpha}_{M,A} \otimes_R \bbs^{\beta}_{A,N}} & {\bbs^{\alpha \odot \beta}_{M,A,N}} & 0 \\
	0 & {\bbs^{\alpha \cdot \beta'}_{M,A} \otimes_R \bbs^{(\beta_k)}_{A,N}} & {\bbs^{\alpha}_{M,A} \otimes_R \bbs^{\beta'}_A \otimes_R \bbs^{(\beta_k)}_{A,N}} & {\bbs^{\alpha \odot \beta'}_{M,A} \otimes_R \bbs^{(\beta_k)}_{A,N}} & 0 \\
	0 & {\bbs^{\alpha \cdot \beta' \odot (\beta_k)}_{M,A,N}} & {\bbs^{\alpha}_{M,A} \otimes_R \bbs^{\beta'\odot (\beta_k)}_{A.N}} & {\bbs^{\alpha \odot \beta' \odot (\beta_k)}_{M,A,N}} & 0 \\
	& 0 & 0 & 0
	\arrow[from=2-1, to=2-2]
	\arrow[from=1-2, to=2-2]
	\arrow[from=1-3, to=2-3]
	\arrow[from=1-4, to=2-4]
	\arrow[from=2-2, to=2-3]
	\arrow[from=2-3, to=2-4]
	\arrow[from=2-4, to=2-5]
	\arrow[from=2-4, to=3-4]
	\arrow[from=2-3, to=3-3]
	\arrow[from=2-2, to=3-2]
	\arrow[from=3-1, to=3-2]
	\arrow[from=3-2, to=3-3]
	\arrow[from=3-3, to=3-4]
	\arrow[from=4-3, to=4-4]
	\arrow[from=4-4, to=4-5]
	\arrow[from=3-4, to=3-5]
	\arrow[from=3-4, to=4-4]
	\arrow[from=4-4, to=5-4]
	\arrow[from=4-3, to=5-3]
	\arrow[from=3-3, to=4-3]
	\arrow[from=3-2, to=4-2]
	\arrow[from=4-1, to=4-2]
	\arrow[from=4-2, to=4-3]
	\arrow[from=4-2, to=5-2]
\end{tikzcd}\]
The bottom two rows and the rightmost two columns are exact by induction (combined with the flatness proved in Lemma \ref{lem:SESandDualityMods}), and an identical diagram chase as above shows that the top row must also be exact.

\textbf{Proof of (1):} Proceed by induction on the length $\ell (\alpha)$, where the base cases $\ell (\alpha)=0$ or $1$ are evident since $\bbs^{\varnothing}_{M,A,N} := M_t \otimes_R N_s$ and there is a short exact sequence
$$0 \to \bbs^{(\alpha_1)}_{M,A,N} \to M_t \otimes_R \bbs^{(\alpha_1)}_{A,N} \to M_{t + \alpha_1} \otimes_R N_s \to 0.$$
The latter two terms in this sequence are flat (resp. projective) by Lemma \ref{lem:SESandDualityMods}, so $\bbs^{(\alpha_1)}_{M,A,N}$ is flat (resp. projective) by Observation \ref{obs:rightResByProj}.

Assuming now that $\ell (\alpha) \geq 2$, write $\alpha = \beta \cdot \gamma$ for two compositions $\beta$, $\gamma$ with $\ell(\beta), \ell(\gamma) > 0$. Then by $(2)$ there is a short exact sequence
$$0 \to \bbs^{\alpha}_{M,A,N} \to \bbs^{\beta}_{M,A} \otimes_R \bbs^{\gamma}_{A,M} \to \bbs^{\beta \odot \gamma}_{M,A,N} \to 0.$$
The latter two terms are flat (resp. projective) by the induction hypothesis, whence Observation \ref{obs:rightResByProj} again implies that $\bbs^{\alpha}_{M,A,N}$ is flat (resp. projective). 
\end{proof}

\begin{definition}\label{def:HGcomplexes}
\rm
Let $\alpha$ be any $\ell$-partitioned composition. Given a Koszul algebra $A$ and a Koszul left $A$-module $M$, define the cochain complex $(\cat{H}_{A,M} (\alpha),\delta)$
whose $i^{th}$ term is 
$$
\cat{H}^i_A(\alpha )
:=\bigoplus_{\substack{I \subseteq [\ell]:\\ |I|=i}} \bbs^{\mu_I (\alpha) }_{A,M}, \quad \text{with differential}
$$
$$d^{\cat{H}_{A,M}}|_{\bbs^{\mu_I (\alpha)}} := \sum_{j \notin I} \sgn(j,I) \rho_{I , I \cup j}.$$
\end{definition}

\begin{remark}
    Note the difference between the components of the modules for the complex of Definition \ref{def:HGcomplexes} versus Definition \ref{def:AlgeHGComplex}: in the first case, the direct sums are parametrized by subsets of $[\ell]$ because of the additional component coming from the module $M$, whereas the components of Definition \ref{def:AlgeHGComplex} only range over subsets of $[\ell-1]$ since there is no additional module $M$. 
\end{remark}

Finally, we conclude this section with the analogue of the Hamel--Goulden type complexes of Definition \ref{def:AlgeHGComplex} for ribbon Schur functors admitting a module input. Recall the notation of Definition \ref{def:MuITermsForAM} in the below:

\begin{theorem}
\label{lem:HGcatComplex}
For all $\ell$-partitioned composition $\alpha$, the complex $\cat{H}_{A,M} (\alpha)$ is a cochain complex with
$$
H^0(\cat{H}_{A,M} (\alpha)) \cong
\bbs^{\alpha}_{A,M},
$$
and which is exact in positive cohomological degrees.
\end{theorem}

\begin{proof}
    First, observe that the statement that $H^0(\cat{H}_{A,M} (\alpha)) \cong
\bbs^{\alpha}_{A,M}$ is evidently true. Proceed by induction on $\ell$, with the base case $\ell = 2$. In this case, write $\alpha = \beta \cdot \gamma$; the complex $\cat{H}_{A,M} (\alpha)$ simply becomes the length $1$ complex
    $$\cat{H}_{A,M} (\alpha) : \quad \bbs^{\beta}_{A} \otimes_R \bbs^{\gamma}_{A,M} \to \bbs^{\beta \odot \gamma}_{A,M} \to 0.$$
    These are the last two terms of the short exact sequence from Lemma \ref{lem:SESandDualityMods}(2), so the statement holds.

    Assume now that $\ell (\alpha) > 2$. Write $\alpha = \alpha' \cdot p_\ell (\alpha)$. Then there is a short exact sequence of complexes:
    $$0 \to \cat{H}_{A,M} (\alpha ' \odot p_\ell (\alpha)) [-1] \to \cat{H}_{A,M} (\alpha) \to \cat{H}_A (\alpha') \otimes_R \bbs^{p_\ell (\alpha)}_{A,M} \to 0.$$
    By the inductive hypothesis, both of the complexes $\cat{H}_{A,M} (\alpha ' \odot p_\ell (\alpha))$ and $\cat{H}_A (\alpha') \otimes_R \bbs^{p_\ell (\alpha)}_{A,M}$ are exact in positive cohomological degrees; taking the long exact sequence of cohomology yields:
    $$0 \to \bbs^{\alpha}_{A,M} \to \bbs^{\alpha'}_{A} \otimes_R \bbs^{p_\ell (\alpha)}_{A,M} \to \bbs^{\alpha' \odot p_\ell (\alpha)}_{A,M} \to H^1 (\cat{H}_{A,M} (\alpha)) \to 0.$$
    The map $ \bbs^{\alpha'}_{A} \otimes_R \bbs^{p_\ell (\alpha)}_{A,M} \to \bbs^{\alpha' \odot p_\ell (\alpha)}_{A,M}$ is precisely the map of Lemma \ref{lem:SESandDualityMods}(2), which is of course surjective. Thus $ H^1 (\cat{H}_{A,M} (\alpha))= 0$ and the result follows.
\end{proof}

Recall that Example \ref{ex:HamelGoulExample} gives an explicit example of the complex of Theorem \ref{lem:HGcatComplex} when just using the Schur modules $\bbs^\alpha_A$.

\begin{example}
    As a more concrete example of Theorem \ref{lem:HGcatComplex}, let $A := S(V) \cong \kk [ x_1 , \dots , x_n]$ be the symmetric algebra (viewed as a polynomial ring) and $\m := (x_1 , \dots , x_n)$ the homogeneous maximal ideal. Given any composition $\alpha$ and integer $d \geq 1$, notice that by definition there is an equality
    $$\bbs^{\alpha}_{A , \m^d} = \bbs^{\alpha \cdot (d)}_A.$$
    Consider the composition $\alpha := (1^4)$ for any integer. Then we can see the difference in the complex of Theorem \ref{lem:HGcatComplex} based on how we view the composition $\alpha$. Viewed as a $3$-partitioned composition, the complex of Theorem \ref{lem:HGcatComplex} yields
    % https://q.uiver.app/#q=WzAsNCxbMCwwLCJBXzEgXFxvdGltZXNfUiBBXzEgXFxvdGltZXMgQV8xIFxcb3RpbWVzX1IgQV9kIl0sWzIsMCwiXFxiZWdpbnttYXRyaXh9IEFfMiBcXG90aW1lc19SIEFfMSBcXG90aW1lc19SIEFfZCBcXFxcIFxcb3BsdXMgXFxcXCBBXzEgXFxvdGltZXNfUiBBXzIgXFxvdGltZXNfUiBBX2QgXFxcXCBcXG9wbHVzIFxcXFwgQV8xIFxcb3RpbWVzX1IgQV8xIFxcb3RpbWVzX1IgQV97ZCsxfSBcXGVuZHttYXRyaXh9Il0sWzQsMCwiXFxiZWdpbnttYXRyaXh9IEFfMyBcXG90aW1lc19SIEFfZCBcXFxcIFxcb3BsdXMgXFxcXCBBXzIgXFxvdGltZXNfUiBBX3tkKzF9IFxcXFwgXFxvcGx1cyBcXFxcIEFfMSBcXG90aW1lc19SIEFfe2QrMn0gXFxlbmR7bWF0cml4fSJdLFs2LDAsIkFfe2QrM30iXSxbMCwxXSxbMSwyXSxbMiwzXV0=
\[\begin{tikzcd}
	{A_1 \otimes_R A_1 \otimes A_1 \otimes_R A_d} && {\begin{matrix} A_2 \otimes_R A_1 \otimes_R A_d \\ \oplus \\ A_1 \otimes_R A_2 \otimes_R A_d \\ \oplus \\ A_1 \otimes_R A_1 \otimes_R A_{d+1} \end{matrix}} && {\begin{matrix} A_3 \otimes_R A_d \\ \oplus \\ A_2 \otimes_R A_{d+1} \\ \oplus \\ A_1 \otimes_R A_{d+2} \end{matrix}} && {A_{d+3}}
	\arrow[from=1-1, to=1-3]
	\arrow[from=1-3, to=1-5]
	\arrow[from=1-5, to=1-7]
\end{tikzcd}\]
    This is simply the degree $d+3$ homogeneous strand of the Bar complex on $\m^d$. If we instead view $\alpha$ as $2$-partitioned via $(1^3) = (1^2) \cdot (1)$, we obtain something distinct from a strand of the Bar complex:
    % https://q.uiver.app/#q=WzAsMyxbMCwwLCJcXGJic157KDFeMil9X0EgXFxvdGltZXNfUiBBXzEgXFxvdGltZXNfUiBBX2QiXSxbMiwwLCJcXGJlZ2lue21hdHJpeH0gXFxiYnNeeygxXjIpfV9BIFxcb3RpbWVzX1IgQV97ZCsxfSBcXFxcIFxcb3BsdXMgXFxcXCBcXGJic157KDEsMil9X0EgXFxvdGltZXNfUiBBX2QgXFxlbmR7bWF0cml4fSJdLFs0LDAsIlxcYmJzXnsoMSxkKzIpfV9BIl0sWzAsMV0sWzEsMl1d
\[\begin{tikzcd}
	{\bbs^{(1^2)}_A \otimes_R A_1 \otimes_R A_d} && {\begin{matrix} \bbs^{(1^2)}_A \otimes_R A_{d+1} \\ \oplus \\ \bbs^{(1,2)}_A \otimes_R A_d \end{matrix}} && {\bbs^{(1,d+2)}_A}
	\arrow[from=1-1, to=1-3]
	\arrow[from=1-3, to=1-5]
\end{tikzcd}\]
    Finally, we can also view $(1^3)$ as $1$-partitioned; recalling that $\bbs^{(1^3)}_A = \bigwedge^3 V$ in this setting, Theorem \ref{lem:HGcatComplex} implies that there is a short exact sequence
    % https://q.uiver.app/#q=WzAsNSxbMSwwLCJcXGJic157KDFeMyxkKX0iXSxbMywwLCJcXGJpZ3dlZGdlXjMgViBcXG90aW1lc19SIFNeZCAoVikiXSxbNSwwLCJcXGJic157KDFeMixkKzEpfV9BIl0sWzYsMCwiMCJdLFswLDAsIjAiXSxbMCwxXSxbMSwyXSxbNCwwXSxbMiwzXV0=
\[\begin{tikzcd}
	0 & {\bbs^{(1^3,d)}} && {\bigwedge^3 V \otimes_R S^d (V)} && {\bbs^{(1^2,d+1)}_A} & 0
	\arrow[from=1-2, to=1-4]
	\arrow[from=1-4, to=1-6]
	\arrow[from=1-1, to=1-2]
	\arrow[from=1-6, to=1-7]
\end{tikzcd}\]
    Note that $\bbs^{(1^3 , d)}_A$ and $\bbs^{(1^2 , d+1)}_A$ are equal to the classical defined Schur modules $\bbs^{(d,1^3)} (V)$ and $\bbs^{(d+1,1^2)} (V)$, respectively, in which case the above short exact sequence recovers the well-known description of hook Schur modules as arising from homogeneous strands of the Koszul complex (see \cite{buchsbaum1975generic}). Thus Theorem \ref{lem:HGcatComplex} yields a family of complexes that interpolates between the full strands of the Bar complex and homogeneous strands of the Koszul complex.
\end{example}

\section{Multi-Schur Functors}\label{sec:multiSchurFunctors}

In this section, we define multi-Schur functors. As mentioned in the introduction, the intuition behind these objects is that they are obtained by taking kernels of the defining ribbon Schur module relations diagonally; they will be particularly helpful for describing canonical equivariant decompositions of the derived invariants over Segre subalgebras.

\begin{definition}\label{def:multiSchurForSameLength}
    Let $\alpha^1 , \dots , \alpha^n$ be a sequence of compositions of fixed length $\ell$ and $A^1 , \dots , A^n$ be a sequence of Koszul $R$-algebras. The \defi{multi-Schur module} $\bbs^{\alpha^1 , \dots , \alpha^n}_{A^1 , \cdots , A^n}$ is defined to be the kernel of the natural map
    $$(A^1 \otimes_R \cdots \otimes_R A^n)_{(\alpha^1 , \dots , \alpha^n)} \to \bigoplus_{i=1}^{\ell-1} (A^1 \otimes_R \cdots \otimes_R A^n)_{(\sigma_i (\alpha^1) , \dots , \sigma_i (\alpha^n))}.$$
    Given a sequence of Koszul left $A^i$-modules $M^i$ of initial degree $t^i$ for $1 \leq i \leq n$, the \defi{multi-Schur module} $\bbs^{\alpha^1 , \dots , \alpha^n}_{(A^1,M^1) , \cdots  , (A^n,M^n)}$ is defined to be the kernel of the natural map
    $$(( A^1 \otimes_R M^1) \otimes_R \cdots \otimes_R (A^n \otimes_R M^n))_{(\alpha^1 \cdot (t^1) , \dots , \alpha^n \cdot (t^n))}$$
    $$\to \bigoplus_{i=1}^{\ell-1} (( A^1 \otimes_R M^1) \otimes_R \cdots \otimes_R (A^n \otimes_R M^n))_{(\sigma_i (\alpha^1 \cdot (t^1)) , \dots , \sigma_i (\alpha^n \cdot (t^n)))}.$$
    The definition for a right $A$-module is analogous. Assume now that $\ell (\alpha^1) \geq 2$. Given Koszul right (resp. left) $A$-modules $M^1 , \dots , M^n$ (resp. $N^1 , \dots , N^n$) of initial degrees $t^i$ (resp. $s^i$), the multi-Schur module
    $\bbs^{\alpha^1 , \dots , \alpha^n}_{(M^1,A^1,N^1) , \cdots , (M^n,A^n,N^n)}$ is defined to be the kernel of the natural map
    $$(( M^1 \otimes_R  A^1 \otimes_R N^1) \otimes_R \cdots \otimes_R (M^n \otimes_R A^n \otimes_R N^n))_{((t^1) \cdot \alpha^1 \cdot (s^1) , \dots , (t^n) \cdot \alpha^n \cdot (s^n))}$$
    $$\to \bigoplus_{i=1}^{\ell-1} (( M^1 \otimes_R  A^1 \otimes_R N^1) \otimes_R \cdots \otimes_R (M^n \otimes_R A^n \otimes_R N^n))_{(\sigma_i ((t^1) \cdot \alpha^1 \cdot (s^1)) , \dots , \sigma_i ((t^n) \cdot \alpha^n \cdot (s^n)))}.$$
\end{definition}

In a similar way to the ribbon Schur functors for single inputs, the multi-Schur modules can also be described as modules of the form $L^I_{M_1 , \dots , M_n}$ for some collection of submodules, but the translation is a little more subtle:

\begin{construction}\label{cons:turningMultiSchurIntoLs}
    Adopt notation and hypotheses as in Definition \ref{def:multiSchurForSameLength}, and recall the notation for the $R$-submodules $S_{A,N,i}^j$ and $S_{M,A,N,i}^j$ as introduced in Definitions \ref{def:koszulModSubmodCollection} and \ref{def:doubleModuleDef}, respectively. For each of the compositions $\alpha^i$, there are isomorphisms of $R$-modules
    $$(A^i \otimes_R N^i)_{\alpha^i \cdot (s^i)} \cong \frac{A_1^{i \otimes |\alpha^i|} \otimes_R N_{s^i}}{\vee_{j \in \phi^{-1} (\alpha^i)} S_{A,N,j}^{|\alpha^i|}},$$
    $$(M^i \otimes_R A^i \otimes_R N^i)_{(t^i) \cdot \alpha^i \cdot (s^i)} \cong \frac{M_{t^i} \otimes_R A_1^{i \otimes |\alpha^i|} \otimes_R N_{s^i}}{\vee_{j \in \phi^{-1} (\alpha^i)} S_{M,A,N,j}^{|\alpha^i|}}.$$
    By the assumption that each $\alpha^i$ has length $\ell$, each of the sets
    $$[|\alpha^i|-1] \backslash \phi^{-1} (\alpha^i)$$
    has length $\ell - 1$.\footnote{Put more informally: quotienting by the module $S_{A,i}^n$ has the effect of deleting the $i$th comma in the composition $(1^{|\alpha|})$ and replacing it with addition; if $\alpha$ is $\ell$-partitioned, that means there are $\ell-1$ commas \emph{not} deleted from $(1^{|\alpha|})$.} Let $\psi_{\alpha^i} : [|\alpha^i|-1] \backslash \phi^{-1} (\alpha^i) \to [\ell - 1]$ denote the unique order-preserving isomorphism between these two sets and consider the induced collections
    $$S_{\underline{A},\underline{N},j}^{\underline{\alpha}} := S_{A^1,N^1,\psi_{\alpha^1}^{-1} (j)}^{|\alpha^1|} + S_{A^2,N^2,\psi_{\alpha^2}^{-1} (j)}^{|\alpha^2|} + \cdots + S_{A^n,N^n,\psi_{\alpha^n}^{-1} (j)}^{|\alpha^n|}$$
    $$
    \subset (A_1^{1 \otimes |\alpha^1|} \otimes_R N^1_{s^1}) \otimes_R ( A_1^{2 \otimes |\alpha^2|} \otimes_R N^2_{s^2}) \otimes_R \cdots \otimes_R (A_1^{n \otimes |\alpha^n|} \otimes_R N^n_{s^n}) ,$$
    $$S_{\underline{M},\underline{A},\underline{N},j}^{\underline{\alpha}} := S_{M^1,A^1,N^1,\psi_{\alpha^1}^{-1} (j)}^{|\alpha^1|} + S_{M^2,A^2,N^2,\psi_{\alpha^2}^{-1} (j)}^{|\alpha^2|} + \cdots + S_{M^n,A^n,N^n,\psi_{\alpha^n}^{-1} (j)}^{|\alpha^n|}$$
    $$
    \subset (M^1_{t^1} \otimes_R A_1^{1 \otimes |\alpha^1|} \otimes_R N^1_{s^1}) \otimes_R (M^2_{t^2} \otimes_R A_1^{2 \otimes |\alpha^2|} \otimes_R N^2_{s^2}) \otimes_R \cdots \otimes_R (M^n_{t^n} \otimes_R A_1^{n \otimes |\alpha^n|} \otimes_R N^n_{s^n}) .$$
    In the above, we are abusing notation for sake of clarity: the module $S_{A^i,N^i,j}^{|\alpha^i|}$ is defined as an $R$-submodule of $A_1^{i \otimes |\alpha^i|} \otimes_R N^i_{s^i}$, but in the above expressions we are vieweing each of these submodules as the $i$th tensor factor of
    $$(A_1^{1 \otimes |\alpha^1|} \otimes_R N^1_{s^1}) \otimes_R ( A_1^{2 \otimes |\alpha^2|} \otimes_R N^2_{s^2}) \otimes_R \cdots \otimes_R (A_1^{n \otimes |\alpha^n|} \otimes_R N^n_{s^n}).$$
    The running flatness assumption implies that inclusion into the $i$th tensor factor is actually a well-defined injection. The same abuse of notation is used for $S_{M^i,A^i,N^i,j}^{|\alpha^i|}$.
\end{construction}

\begin{obs}\label{obs:muliSchursAsL's}
Adopt notation and hypotheses as in Construction \ref{cons:turningMultiSchurIntoLs}. Then there are isomorphisms of $R$-modules:
$$\bbs^{\alpha^1 , \dots , \alpha^n}_{(A^1,N^1) , \dots , (A^n , N^n)} \cong$$
$$\frac{\w_{j \in [\ell -1 ]} S^{\underline{\alpha}}_{\underline{A} , \underline{M}}}{\vee_{j \notin \phi^{-1} (\alpha^1)} S_{A^1,N^1,j}^{|\alpha^1|} + \cdots + \vee_{j \notin \phi^{-1} (\alpha^n)} S_{A^n,N^n,j}^{|\alpha^n|}},$$
$$\bbs^{\alpha^1 , \dots , \alpha^n}_{(M^1,A^1,N^1) , \dots , (M^n, A^n , N^n)} \cong$$
$$\frac{\w_{j \in [\ell -1 ]} S^{\underline{\alpha}}_{\underline{M}, \underline{A} , \underline{M}}}{\vee_{j \notin \phi^{-1} (\alpha^1)} S_{M^1,A^1,N^1,j}^{|\alpha^1|} + \cdots + \vee_{j \notin \phi^{-1} (\alpha^n)} S_{M^n,A^n,N^n,j}^{|\alpha^n|}}.$$
\end{obs}

An interesting aspect of multi-Schur modules is that one can define the multi-Schur modules when the input compositions have different sizes; to make sense of this, the inputs need to be $\ell$-partitioned into the same number of parts, instead:

\begin{definition}
    Let $\alpha^1 , \dots , \alpha^n$ be a sequence of $\ell$-partitioned compositions and $A^1 , \dots , A^n$ be a sequence of Koszul $R$-algebras. The \emph{multi-Schur module} $\bbs^{\alpha^1 , \dots , \alpha^n}_{A^1 , \cdots , A^n}$ is defined as the kernel of the natural map:
    $$\bbs^{\mu_\varnothing (\alpha^1)}_{A^1} \otimes_R \bbs^{\mu_\varnothing (\alpha^2)}_{A^2} \otimes_R \cdots \otimes_R \bbs^{\mu_\varnothing (\alpha^n)}_{A^n} \to \bigoplus_{i=1}^{\ell -1 } \bbs^{\mu_i (\alpha^1)}_{A^1} \otimes_R \bbs^{\mu_i (\alpha^2)}_{A^2} \otimes_R \cdots \otimes_R \bbs^{\mu_i (\alpha^n)}_{A^n}.$$
    If $\alpha := \alpha^1 = \alpha^2 = \cdots = \alpha^n$ then the more concise notation $\bbs^{\alpha}_{A^1 , \dots , A^n}$ will be used to denote $\bbs^{\alpha^1 , \dots , \alpha^n}_{A^1 , \dots , A^n}$.

    Likewise, given a Koszul left $A$-module $M$, define the multi-Schur module $\bbs^{\alpha^1 , \dots , \alpha^n}_{(A^1, M^1) , \cdots , (A^n,M^n)}$ as the kernel of the natural map:
    $$\bbs^{\mu_\varnothing (\alpha^1)}_{(A^1,M^1)} \otimes_R \bbs^{\mu_\varnothing (\alpha^2)}_{(A^2,M^2)} \otimes_R \cdots \otimes_R \bbs^{\mu_\varnothing (\alpha^n)}_{(A^n,M^n)} \to \bigoplus_{i=1}^\ell \bbs^{\mu_i (\alpha^1)}_{(A^1,M^1)} \otimes_R \bbs^{\mu_i (\alpha^2)}_{(A^2,M^2)} \otimes_R \cdots \otimes_R \bbs^{\mu_i (\alpha^n)}_{(A^n,M^n)}.$$
    
    Finally, given a Koszul left (resp. right) $A$-module $N$ (resp. $M$), define the multi-Schur module 
    $$\bbs^{\alpha^1 , \dots , \alpha^n}_{(M^1,A^1, N^1) , \cdots , (M^n,A^n,N^n)}$$ as the kernel of the natural map:
    $$\bbs^{\mu_\varnothing (\alpha^1)}_{(M^1,A^1,N^1)} \otimes_R  \cdots \otimes_R \bbs^{\mu_\varnothing (\alpha^n)}_{(M^n,A^n,N^n)} \to \bigoplus_{i=1}^\ell \bbs^{\mu_i (\alpha^1)}_{(M^1,A^1,N^1)}  \otimes_R \cdots \otimes_R \bbs^{\mu_i (\alpha^n)}_{(M^n,A^n,N^n)}.$$
\end{definition}

% \begin{remark}
%     By convention, if any of the compositions $\alpha^i$ is empty, then the module $\bbs^{\alpha^1 , \dots , \alpha^n}_{A^1 , \dots , A^n}$ is defined to be $0$. \keller{add the convention when there are module inputs too}
% \end{remark}

The following is the evident analogue of Observation \ref{obs:intersectingSchurs} for multi-Schur modules:

\begin{obs}\label{obs:intersectMultischurs}
    Let $\alpha^i, \beta^i$, and $\gamma^i$ be $\ell$ (resp. $j$, $k$)-partitioned partitions. Given a collection of Koszul algebras $A^i$ and left (resp. right) Koszul $A^i$-modules $N_i$ (resp. $M_i$) for $1 \leq i \leq n$, there is an equality
    $$\bbs^{\alpha^1\cdot \beta^1 \cdot \gamma^1 , \dots , \alpha^n \cdot \beta^n \cdot \gamma^n}_{(M^1 , A^1 , N^1) , \cdots , (M^n , A^n , N^n)} = $$
    $$\left( \bbs^{\alpha^1\cdot \beta^1  , \dots , \alpha^n \cdot \beta^n }_{(M^1 , A^1 , N^1) , \cdots , (M^n , A^n , N^n)} \otimes_R \bbs^{ \gamma^1 , \dots ,  \gamma^n}_{(M^1 , A^1 , N^1) , \cdots , (M^n , A^n , N^n)} \right) $$
    $$\cap \left( \bbs^{\alpha^1 , \dots , \alpha^n }_{(M^1 , A^1 , N^1) , \cdots , (M^n , A^n , N^n)} \otimes_R \bbs^{ \beta^1 \cdot \gamma^1 , \dots ,  \beta^n \cdot \gamma^n}_{(M^1 , A^1 , N^1) , \cdots , (M^n , A^n , N^n)} \right).$$
\end{obs}

The following properties are immediate from the definition of multi-Schur modules:

\begin{prop}
    Let $\alpha^1 , \dots , \alpha^n$ be a sequence of $\ell$-partitioned compositions and $A^1 , \dots , A^n$ be a sequence of Koszul $R$-algebras. Then:
    \begin{enumerate}
        \item Let $\tau \in \Sigma_n$ be any permutation of $[n]$. Then there is a natural isomorphism
        $$\bbs^{\alpha^1 , \dots , \alpha^n}_{A^1 , \dots , A^n} \cong \bbs^{\alpha^{\tau(1)} , \dots , \alpha^{\tau(n)} }_{A^{\tau(1)} , \dots , A^{\tau(n)}}.$$
        \item If $\ell = 1$, there is an isomorphism
        $$\bbs^{\alpha^1 , \dots , \alpha^n}_{A^1 , \dots , A^n} \cong \bbs^{\alpha^1}_{A^1} \otimes_R \cdots \otimes_R \bbs^{\alpha^n}_{A^n}.$$
        \item If each $A^i$ admits the structure of a $R[G^i]$-module for some group $G^i$ (where $1 \leq i \leq n$), then the multi-Schur module $\bbs^{\alpha^1 , \dots , \alpha^n}_{A^1 , \dots , A^n}$ admits the structure of a $R[G^1 \times \cdots \times G^n]$-module, and all of the above isomorphisms are $G^1 \times \cdots \times G^n$-equivariant.
    \end{enumerate}
\end{prop}

In view of Observation \ref{obs:intersectMultischurs}, the notation for multi-Schur modules can quickly become overwhelming. For this reason, we introduce the following shorthand notation:

\begin{notation}\label{not:multiIndexForSchur}
    Let $\alpha^1 , \dots , \alpha^n$ be a collection of $\ell$-partitioned compositions. Given a collection of Koszul algebras $A^i$ and left (resp. right) Koszul $A^i$-modules $N_i$ (resp. $M_i$) for $1 \leq i \leq n$, use the more concise notation
    $$\bbs^{\underline{\alpha}}_{\underline{M} , \underline{A}, \underline{N}} := \bbs^{\alpha^1 , \dots , \alpha^n}_{(M^1 , A^1 , N^1) , \cdots , (M^n , A^n , ^n)}.$$
    For given tuples of compositions $\underline{\alpha}$ and $\underline{\beta}$, extend the operations of Definitions \ref{def:booleanAndRefinement} and \ref{def:operationsOfComps} by applying them coordinate-wise to the tuples. With this identification, we have the equalities
    $$\underline{\alpha} := \{ \alpha^1 , \dots  , \alpha^n \}, \quad \underline{M} := \{ M^1 , \dots , M^n \},$$
    $$\underline{A} := \{ A^1 , \dots , A^n \}, \quad \underline{N} := \{ N^1 , \dots , N^n \}.$$
\end{notation}

The following lemma is the multi-Schur analog of Lemma \ref{lem:SESandDualityMods}, but the proof is actually a little more subtle due to the added difficulty of allowing compositions of different sizes.

\begin{lemma}\label{lem:theMultiSchurSES}
    Let $\underline{\alpha}$ and $\underline{\beta}$ be sequences of $k$-partitioned and $\ell-k$-partitioned compositions, respectively, for some fixed $1 \leq k \leq \ell$. Let $\underline{A} = \{ A^1 , \dots , A^n \}$ be a sequence of Koszul $R$-algebras and $\underline{M}$ a sequence of Koszul left $\underline{A}$-modules. Then:
    \begin{enumerate}
        \item Every multi-Schur module is $R$-flat.
        \item There is a canonical short exact sequence 
        $$0 \to \bbs^{\underline{\alpha} \cdot \underline{\beta}}_{\underline{A} , \underline{M}} \to \bbs^{\underline{\alpha} }_{\underline{A} } \otimes_R \bbs^{ \underline{\beta}}_{\underline{A} , \underline{M}} \to \bbs^{\underline{\alpha} \odot \underline{\beta}}_{\underline{A} , \underline{M}} \to 0.$$
    \end{enumerate}
    If all Koszul algebras/modules are $R$-projective, then every multi-Schur module $\bbs^{\underline{\alpha}}_{\underline{A} , \underline{M}}$ is also $R$-projective. The analogous statement for right modules holds as well.
\end{lemma}

\begin{proof}

It suffices to prove $(2)$, since $(1)$ is a consequence of $(2)$ combined with Observation \ref{obs:rightResByProj}.

\textbf{Proof of (2):} Let $p := \max \{ \ell (p_i(\alpha^j)) \mid 1 \leq i \leq k, \ 1 \leq j \leq n \}$ and define $q := |\{ i \mid \ell (p_i (\alpha^j)) = p \}|$. The proof is by a double induction on the values $p$ and $q$. Notice that when $p =1$ and $q$ is arbitrary, recall that $A^1 \otimes_R \cdots \otimes_R A^n$ is a Koszul algebra and the tensor product $M^1 \otimes_R \cdots \otimes_R M^n$ is a left Koszul module over $A^1 \otimes_R \cdots \otimes_R A^n$ by Corollary \ref{cor:tensorsAreKoszul}. By Observation \ref{obs:muliSchursAsL's} combined with Proposition \ref{prop:distributivityComplexes}, the sequence of $(2)$ is exact.

Assume now that $p > 1$ and $q = 1$. This means that there is some element of one of the compositions $\underline{\alpha}$ or $\underline{\beta}$ whose largest part is strictly greater than $1$. Reversing the order of $\alpha$ and $\beta$, it is of no loss of generality to assume that $\underline{\alpha}$ contains a part of size $p$. Moreover, for simplicity of notation, let us assume that the first element of $\underline{\alpha}$ has size $p$ (the general case is identical but notationally more cumbersome). Write $\underline{\alpha} = \underline{(a)} \cdot \underline{\alpha'}$, so that by construction $\underline{\alpha'}$ has all parts of size $< p$. There is a commutative diagram:

\[\begin{tikzcd}
	& 0 & 0 & 0 \\
	0 & {\bbs^{ \underline{\alpha} \cdot \underline{\beta}}_{\underline{A} , \underline{M}}} & {\bbs^{\underline{\alpha} }_{\underline{A} } \otimes_R \bbs^{ \underline{\beta}}_{\underline{A} , \underline{M}}} & {\bbs^{ \underline{\alpha} \odot \underline{\beta}}_{\underline{A} , \underline{M}}} & 0 \\
	0 & {\bbs^{\underline{(a)} }_{\underline{A} } \otimes_R \bbs^{\underline{\alpha'} \cdot \underline{\beta}}_{\underline{A} , \underline{M}}} & {\bbs^{\underline{(a)}}_{\underline{A} } \otimes_R \bbs^{ \underline{\alpha'} }_{\underline{A} } \otimes_R \bbs^{\underline{\beta}}_{\underline{A} , \underline{M}}} & {\bbs^{\underline{(a)} }_{\underline{A} } \otimes_R \bbs^{ \underline{\alpha'} \odot \underline{\beta}}_{\underline{A} , \underline{M}}} & 0 \\
	0 & {\bbs^{\underline{(a)} \odot \underline{\alpha'} \cdot \underline{\beta}}_{\underline{A} , \underline{M}}} & {\bbs^{\underline{(a)} \odot \underline{\alpha'}}_{\underline{A} } \otimes_R \bbs^{ \underline{\beta}}_{\underline{A} , \underline{M}}} & {\bbs^{\underline{(a)} \odot \underline{\alpha'} \odot \underline{\beta}}_{\underline{A} , \underline{M}}} & 0 \\
	& 0 & 0 & 0
	\arrow[from=2-1, to=2-2]
	\arrow[from=1-2, to=2-2]
	\arrow[from=1-3, to=2-3]
	\arrow[from=1-4, to=2-4]
	\arrow[from=2-3, to=2-4]
	\arrow[from=2-4, to=2-5]
	\arrow[from=3-4, to=3-5]
	\arrow[from=3-3, to=3-4]
	\arrow[from=2-3, to=3-3]
	\arrow[from=2-2, to=3-2]
	\arrow[from=2-2, to=2-3]
	\arrow[from=3-2, to=3-3]
	\arrow[from=3-1, to=3-2]
	\arrow[from=4-1, to=4-2]
	\arrow[from=4-2, to=4-3]
	\arrow[from=4-3, to=4-4]
	\arrow[from=4-4, to=4-5]
	\arrow[from=4-4, to=5-4]
	\arrow[from=4-3, to=5-3]
	\arrow[from=4-2, to=5-2]
	\arrow[from=3-2, to=4-2]
	\arrow[from=3-3, to=4-3]
	\arrow[from=3-4, to=4-4]
	\arrow[from=2-4, to=3-4]
\end{tikzcd}\]
    
In the above diagram, notice that the bottom two rows and the last two columns are exact by the inductive hypothesis. We may also assume by induction on $p$ that the map 
$$\Theta: \bbs^{\underline{(a)} }_{\underline{A} } \otimes_R \bbs^{\underline{\alpha'} \cdot \underline{\beta}}_{\underline{A} , \underline{M}} \to \bbs^{\underline{(a)} \odot \underline{\alpha'} \cdot \underline{\beta}}_{\underline{A} , \underline{M}}$$
is a surjection. We claim that with this information the first column of the above diagram is exact. It is evident just by definition that the map
$$ \bbs^{\underline{(a)} \cdot \underline{\alpha'} \cdot \underline{\beta}}_{\underline{A} , \underline{M}} \to \bbs^{\underline{\alpha} }_{\underline{A} } \otimes_R \bbs^{ \underline{\beta}}_{\underline{A} , \underline{M}}$$
is always an injection, so it remains to prove exactness at the middle term $\bbs^{\underline{\alpha} }_{\underline{A} } \otimes_R \bbs^{ \underline{\beta}}_{\underline{A} , \underline{M}} $. However, a diagram chase employing exactness of the middle column shows that
$$\ker \Theta \subset \left(\bbs^{\underline{\alpha} }_{\underline{A} } \otimes_R \bbs^{ \underline{\beta}}_{\underline{A} , \underline{M}}  \right) \cap \left( \bbs^{\underline{(a)} }_{\underline{A} } \otimes_R \bbs^{\underline{\alpha'} \cdot \underline{\beta}}_{\underline{A} , \underline{M}} \right),$$
and this latter intersection is precisely equal to $\bbs^{ \underline{\alpha} \cdot \underline{\beta}}_{\underline{A} , \underline{M}}$ by Observation \ref{obs:intersectMultischurs}. Since the reverse containment evidently holds, exactness of the first column follows, and hence all columns of the above diagram are exact. Employing the long exact sequence of homology, it follows that the first row is exact.

For the inductive step on $q$, the argument is actually identical. After choosing a similar decomposition of $\alpha$, there is an identical commutative diagram. By construction, the bottom two rows and rightmost two columns are exact by induction on $q$. A verbatim argument works for showing that the first column is exact, and hence the long exact sequence of homology yields the statement in general. 
\end{proof}

\begin{cor}
    Let $\underline{\alpha}$ and $\underline{\beta}$ be a collection of $k$ and $\ell-k$-partitioned compositions, respectively. Let $A^i$ be a collection of Koszul algebras and $N_i$ (resp. $M_i$) a collection of left (resp. right) Koszul $A^i$-modules for $1 \leq i \leq n$. Then:
    \begin{enumerate}
        \item Every multi-Schur module $\bbs^{\underline{\alpha}}_{\underline{M} , \underline{A} , \underline{N}}$ is $R$-flat.
        \item There is a canonical short exact sequence
        $$0 \to \bbs^{\underline{\alpha} \cdot \underline{\beta}}_{\underline{M} , \underline{A} , \underline{N}} \to \bbs^{\underline{\alpha} }_{\underline{M}, \underline{A} } \otimes_R \bbs^{ \underline{\beta}}_{\underline{A} , \underline{N}} \to \bbs^{\underline{\alpha} \odot \underline{\beta}}_{\underline{M} , \underline{A} , \underline{N}} \to 0.$$
    \end{enumerate}
    If all the Koszul algebras/modules are $R$-projective, then every multi-Schur module $\bbs^{\underline{\alpha}}_{\underline{A} , \underline{M}}$ is also $R$-projective.
\end{cor}

\begin{proof}
    This proof is formally identical to the proof of Corollary \ref{cor:doubleModuleInputProperties} but with the multi-index notation of Notation \ref{not:multiIndexForSchur} used instead.
\end{proof}

Finally, we have the generalization of the complexes $\cat{H} (\alpha)$ of Definition \ref{def:HGcomplexes} for the multi-Schur setting:

\begin{definition}
    Let $\alpha^1 , \dots , \alpha^n$ be a sequence of $\ell$-partitioned compositions. Given Koszul algebras $A^i$ and left Koszul modules $M^i$ for $1 \leq i \leq n$, define the cochain complex $\cat{H}_{(A^1,M^1),\dots , (A^n , M^n)}^\bullet (\alpha^1 , \dots,  \alpha^n)$ via
    $$\cat{H}_{(A^1,M^1),\dots , (A^n , M^n)}^i (\alpha^1 , \dots,  \alpha^n) := \bigoplus_{|I| = i} \bbs^{\mu_I (\alpha^1) , \dots , \mu_I (\alpha^n)}_{(A^1, M^1) , \cdots , (A^n , M^n)}, \quad \text{with differential}$$
    $$d^{\cat{H}_{(A^1,M^1), \dots , (A^n , M^n)}}|_{\bbs^{\mu_I (\alpha^1) , \dots , \mu_I (\alpha^n)}_{A^1 , \cdots , A^n}} := \sum_{j \notin I} \sgn (j,I) \rho_{I , I \cup j}.$$
\end{definition}

\begin{cor}
    Let $\alpha^1 , \dots , \alpha^n$ be a sequence of $\ell$-partitioned compositions. Given Koszul algebras $A^i$ and left Koszul modules $M^i$ for $1 \leq i \leq n$, the cochain complex $\cat{H}^\bullet_{(A^1,M^1) , \dots , (A^n , M^n)} (\alpha^1 , \dots , \alpha^n)$ is exact in positive cohomological degrees and
    $$H^0 \left( \cat{H}^\bullet_{(A^1,M^1) , \dots , (A^n , M^n)} (\alpha^1 , \dots , \alpha^n) \right) = \bbs^{\alpha^1 , \dots , \alpha^n}_{A^1 , \cdots , A^n}.$$
\end{cor}

\subsection{Some Generalities on Filtrations}\label{subsec:filtrations}

We now turn our attention to the task of filtering the multi-Schur modules $\bbs^{\underline{\alpha}}_{\underline{M} , \underline{A} , \underline{N}}$ for given choices of $\underline{\alpha}$. This section collects a few general results on ``splicing" filtrations together; these results are essentially trivial, but it will be useful to refer to them explicitly. First, let us recall the definition of a filtration.

\begin{definition}
    Let $M$ be an $R$-module. A (ascending) \defi{filtration} of $M$ is a chain of $R$-submodules
    $$0 = F_0 \subset F_1 \subset \cdots \subset F_{n-1} \subset F_n = M.$$
    The successive quotients $F_i / F_{i-1}$ are called the \defi{associated graded pieces}. 
\end{definition}

The following observation is likely a common first exercise on properties of filtrations, but we state and prove it here for completeness. The intuition here is that filtrations whose graded pieces admit further filtrations may be refined to a single filtration of the larger object with the same graded pieces.

\begin{obs}\label{obs:refiningCompositions}
Let $M$ be an $R$-module equipped with a finite filtration
$$F : \quad 0 = F_0 \subset F_1 \subset \cdots \subset F_{n-1} \subset F_n = M.$$
Suppose that for all $i \geq 1$, the $i$th graded piece $G_i := F_i / F_{i-1}$ admits a finite filtration
$$F^i : \quad 0 = F_0^i \subset F_1^i \subset \cdots \subset F_{n_i-1}^i \subset F_{n_i}^i = G_i$$
with graded pieces $G_j^i := F_j^i/F_{j-1}^i$ for each $j \geq 1$. Then the filtration $F$ of $M$ may be refined to a filtration $F'$ of $M$ with associated graded pieces $G_j^i$ for $1 \leq i \leq n$, $1 \leq j \leq n_i$. 
\end{obs}

The above observation may be used to understand how to filter tensor products of modules, each equipped separately with their own filtrations; this is the content of the following corollary.

\begin{cor}\label{cor:splicingFiltrations}
    Let $M^1, M^2 , \dots , M^\ell$ be a collection of $R$-modules equipped with filtrations
$$F^i : \quad \quad 0 = F_0^i \subset F_1^i \subset \cdots \subset F_{n_i-1}^i \subset F_{n_i}^i = M^i,$$
and assume that each associated graded piece $G_j^i := F_j^i / F_{j-1}^i$ is a flat $R$-module. Then each of the $R$-modules $M^i$ is flat and there is a filtration of $M^1 \otimes_R M^2 \otimes_R \cdots \otimes_R M^\ell$ with associated graded pieces of the form 
$$G_{i_1}^1 \otimes_R G_{i_2}^2 \otimes_R \cdots \otimes_R G_{i_\ell}^\ell,$$
where $1 \leq i_k \leq n_k$ for each $1 \leq k \leq \ell$. 
\end{cor}

\begin{proof}
    Proceed by induction on $\ell$, with the base case $\ell = 1$ being vacuous. Assume $\ell > 1$ and let $F'$ be a filtration of $M^1 \otimes_R M^2 \otimes_R \cdots \otimes_R M^{\ell-1}$ with associated graded pieces of the form
    $$G_{i_1}^1 \otimes_R \cdots \otimes_R G_{i_{\ell-1}}^{\ell-1}.$$
    By the flatness assumption, the induced filtration $F' \otimes_R M^\ell$ has associated graded pieces of the form
    $$G_{i_1}^1 \otimes_R \cdots \otimes_R G_{i_{\ell-1}}^{\ell-1} \otimes_R M^\ell.$$
    Further filtering each graded piece by the filtration $F^\ell$ of $M^\ell$, we may employ Observation \ref{obs:refiningCompositions} to deduce the result.
\end{proof}

\subsection{A Canonical Filtration of Multi-Schur Functors}\label{subsec:canonicalFiltrationsOfMultiSchur}

The following short exact sequence, when combined with some of the other filtration results proved in this subsection, will be the essential ingredient for proving Theorem \ref{thm:multiSchurFactors}. It will be used to place an object we want to understand in the middle of a short exact sequence whose outer terms have a well-understood filtration.

\begin{lemma}\label{lem:multiSchuriltrationSES}
    Let $\alpha^1 , \dots , \alpha^n$ and $\beta^1 , \dots , \beta^n$ be sequences of $k$-partitioned and $\ell-k$-partitioned compositions, respectively, for some fixed $0 \leq k < \ell$. Let $A^1 , \dots , A^n$ be a sequence of Koszul $R$-algebras and $N^i$ (resp. $M^i$) be Koszul left (resp. right) $A^i$-modules for $1 \leq i \leq n$. Then there is a canonical short exact sequence
    $$0 \to \bbs^{\alpha^1 \cdot \beta^1 , \dots , \alpha^{n-1} \cdot \beta^{n-1} , \alpha^n}_{\underline{M} , \underline{A} , \underline{N}} \otimes_R \bbs^{\beta^n}_{M^n,A^n,N^n} $$
    $$\to \bbs^{\alpha^1 \cdot \beta^1 , \dots , \alpha^n \cdot \beta^n}_{\underline{M} , \underline{A} , \underline{N}} \to \bbs^{\alpha^1 \odot \beta^1, \dots , \alpha^{n-1} \odot \beta^{n-1} , \alpha^n \cdot \beta^n}_{\underline{M} , \underline{A} , \underline{N}} \to 0.$$
    In the above, the sequences of $\alpha^1 \cdot \beta^1 , \dots , \alpha^{n-1} \cdot \beta^{n-1} , \alpha^{n}$ and $\alpha^1 \odot \beta^1, \dots , \alpha^{n-1} \odot \beta^{n-1} , \alpha^n \cdot \beta^n$ are being viewed as $k$ and $\ell-1$-partitioned compositions via the convention of Convention \ref{conv:concatenationConventions}.
\end{lemma}

\begin{example}\label{ex:theSESExample}
    Let $A$ and $B$ be two Koszul $R$-algebras and $\alpha^1 = \beta^1 = (1^3)$, viewed as a $3$-partitioned composition. Then the short exact sequence of Lemma \ref{lem:multiSchuriltrationSES} takes the form
    $$\bbs^{(1^3),(1^2)}_{A , B} \otimes_R B_1 \to \bbs^{(1^3),(1^3)}_{A , B} \to \bbs^{(1,2),(1^3)}_{A, B} \to 0.$$
    In the above, $(1^3)$ is being viewed as the $2$-partitioned composition $(1) \cdot (1^2)$ in the first term of the sequence, and likewise for the last term in the sequence.
\end{example}

\begin{proof}
    For simplicity of notation, we will prove the theorem only for Koszul algebra inputs (otherwise the relevant diagrams are too large to display). The proof follows by examining the commutative diagram of Figure \ref{fig:proofOfMultiSES}; the base case and the inductive step are outlined in the caption.
    \begin{sidewaysfigure}
    \centering
    {\vspace{7in}

\[\begin{tikzcd}
	& 0 & 0 \\
	& {\bbs^{\alpha^1 \cdot \beta^1 , \dots , \alpha^{n-1} \cdot \beta^{n-1} , \alpha^n}_{A^1 , \cdots , A^n} \otimes_R \bbs^{\beta^n}_{A^n}} & {\bbs^{\alpha^1 \cdot \beta^1 , \dots , \alpha^{n-1} \cdot \beta^{n-1} , \alpha^1}_{A^1 , \cdots , A^n} \otimes_R \bbs^{\beta^n}_{A^n}} & 0 \\
	0 & {\bbs^{\alpha^1 \cdot \beta^1 , \dots , \alpha^{n-1} \cdot \beta^{n-1} , \alpha^n \cdot \beta^n}_{A^1 , \cdots , A^n}} & {\bbs^{\alpha^1  , \dots  , \alpha^n }_{A^1 , \cdots , A^n} \otimes_R \bbs^{\beta^1  , \dots  , \beta^n }_{A^1 , \cdots , A^n}} & {\bbs^{\alpha^1 \odot \beta^1 , \dots , \alpha^{n-1} \odot \beta^{n-1} , \alpha^n \odot \beta^n}_{A^1 , \cdots , A^n}} & 0 \\
	0 & {\bbs^{\alpha^1 \odot \beta^1 , \dots , \alpha^{n-1} \odot \beta^{n-1} , \alpha^n \cdot \beta^n}_{A^1 , \cdots , A^n}} & {\bbs^{\alpha^1 \odot \beta^1 , \dots , \alpha^{n-1} \odot \beta^{n-1}}_{A^1 , \cdots , A^{n-1}} \otimes_R \bbs^{\alpha^n}_{A^n} \otimes_R \bbs^{\beta^n}_{A^n}} & {\bbs^{\alpha^1 \odot \beta^1 , \dots , \alpha^{n-1} \odot \beta^{n-1} , \alpha^n \odot \beta^n}_{A^1 , \cdots , A^n}} & 0 \\
	& 0 & 0
	\arrow[from=2-3, to=2-4]
	\arrow[from=2-4, to=3-4]
	\arrow[Rightarrow, no head, from=3-4, to=4-4]
	\arrow[Rightarrow, no head, from=2-2, to=2-3]
	\arrow[from=2-3, to=3-3]
	\arrow[from=3-3, to=4-3]
	\arrow[from=1-3, to=2-3]
	\arrow[from=4-3, to=5-3]
	\arrow[from=3-3, to=3-4]
	\arrow[from=4-3, to=4-4]
	\arrow[from=2-2, to=3-2]
	\arrow[from=3-2, to=4-2]
	\arrow[from=4-1, to=4-2]
	\arrow[from=3-1, to=3-2]
	\arrow[from=1-2, to=2-2]
	\arrow[from=4-2, to=5-2]
	\arrow[from=4-2, to=4-3]
	\arrow[from=3-2, to=3-3]
	\arrow[from=4-4, to=4-5]
	\arrow[from=3-4, to=3-5]
\end{tikzcd}\]

\footnotesize
\[\begin{tikzcd}
	& 0 & 0 & 0 \\
	0 & {\bbs^{\alpha^1 \cdot \beta^1 \cdot \gamma^1 , \dots , \alpha^{n-1} \cdot \beta^{n-1} \cdot \gamma^{n-1} , \alpha^n \cdot \beta^n}_{A^1 , \cdots , A^n} \otimes_R \bbs^{\gamma^n}_{A^n}} & {\begin{matrix} \bbs^{\alpha^1 , \dots , \alpha^n}_{A^1 , \cdots , A^n} \otimes_R \\ \bbs^{\beta^1 \cdot \gamma^1 , \dots , \beta^{n-1} \cdot \gamma^{n-1},\beta^n}_{A^1 , \cdots , A^n} \otimes_R \bbs^{\gamma^n}_{A^n} \end{matrix}} & {\bbs^{\alpha^1 \odot \beta^1 \cdot \gamma^1 , \dots , \alpha^{n-1} \odot \beta^{n-1} \cdot \gamma^{n-1} , \alpha^n \odot \beta^n}_{A^1 , \cdots , A^n} \otimes_R \bbs^{\gamma^n}_{A^n}} & 0 \\
	0 & {\bbs^{\alpha^1 \cdot \beta^1 \cdot \gamma^1 , \dots , \alpha^{n-1} \cdot \beta^{n-1} \cdot \gamma^{n-1} , \alpha^n \cdot \beta^n \cdot \gamma^n}_{A^1 , \cdots , A^n} } & {\begin{matrix} \bbs^{\alpha^1 , \dots , \alpha^n}_{A^1 , \cdots , A^n} \otimes_R \\ \bbs^{\beta^1 \cdot \gamma^1 , \dots , \beta^{n-1} \cdot \gamma^{n-1},\beta^n\cdot \gamma^n}_{A^1 , \cdots , A^n}  \end{matrix}} & {\bbs^{\alpha^1 \odot \beta^1 \cdot \gamma^1 , \dots , \alpha^{n-1} \odot \beta^{n-1} \cdot \gamma^{n-1} , \alpha^n \odot \beta^n \cdot \gamma^n}_{A^1 , \cdots , A^n} } & 0 \\
	0 & {\bbs^{\alpha^1 \cdot \beta^1 \odot \gamma^1 , \dots , \alpha^{n-1} \cdot \beta^{n-1} \odot \gamma^{n-1} , \alpha^n \cdot \beta^n \cdot \gamma^n}_{A^1 , \cdots , A^n} } & {\begin{matrix} \bbs^{\alpha^1 , \dots , \alpha^n}_{A^1 , \cdots , A^n} \otimes_R \\ \bbs^{\beta^1 \odot \gamma^1 , \dots , \beta^{n-1} \odot \gamma^{n-1},\beta^n\cdot \gamma^n}_{A^1 , \cdots , A^n}  \end{matrix}} & {\bbs^{\alpha^1 \odot \beta^1 \odot \gamma^1 , \dots , \alpha^{n-1} \odot \beta^{n-1} \odot \gamma^{n-1} , \alpha^n \odot \beta^n \cdot \gamma^n}_{A^1 , \cdots , A^n} } & 0 \\
	& 0 & 0 & 0
	\arrow[from=2-2, to=2-3]
	\arrow[from=2-3, to=2-4]
	\arrow[from=3-2, to=3-3]
	\arrow[from=3-3, to=3-4]
	\arrow[from=4-2, to=4-3]
	\arrow[from=4-3, to=4-4]
	\arrow[from=2-2, to=3-2]
	\arrow[from=2-3, to=3-3]
	\arrow[from=2-4, to=3-4]
	\arrow[from=3-4, to=4-4]
	\arrow[from=3-3, to=4-3]
	\arrow[from=3-2, to=4-2]
	\arrow[from=2-1, to=2-2]
	\arrow[from=3-1, to=3-2]
	\arrow[from=4-1, to=4-2]
	\arrow[from=4-4, to=4-5]
	\arrow[from=3-4, to=3-5]
	\arrow[from=2-4, to=2-5]
	\arrow[from=1-4, to=2-4]
	\arrow[from=1-3, to=2-3]
	\arrow[from=1-2, to=2-2]
	\arrow[from=4-2, to=5-2]
	\arrow[from=4-3, to=5-3]
	\arrow[from=4-4, to=5-4]
\end{tikzcd}\]}
\caption{The commutative diagrams used to prove Lemma \ref{lem:multiSchuriltrationSES}. In the first diagram, all rows are exact (the middle is exact by Lemma \ref{lem:theMultiSchurSES}. Likewise, the third column is clearly exact and the middle column is exact by Lemma \ref{lem:theMultiSchurSES}, so the first column is exact by the long exact sequence of homology. For the second diagram, all of the rows are exact by Lemma \ref{lem:theMultiSchurSES}, and the second and third columns are exact according to the inductive hypothesis. By the long exact sequence of homology, the first column must also be exact.}
\label{fig:proofOfMultiSES}
\end{sidewaysfigure}
\end{proof}

The next lemma can be seen as a ``first approximation" of the filtration in Theorem \ref{thm:multiSchurFactors}.

\begin{lemma}\label{lem:multiSchurFiltFirstStep}
    Let $\alpha^1 , \dots , \alpha^n$ be a sequence of $\ell$-partitioned compositions and $A^1 , \dots , A^n$ be a sequence of Koszul $R$-algebras and $N^i$ (resp. $M^i$) be Koszul left (resp. right) $A^i$-modules for $1 \leq i \leq n$. Then the multi-Schur module $\bbs^{\underline{\alpha}}_{\underline{M} , \underline{A} , \underline{N}}$ admits a canonical filtration with graded pieces of the form
    $$\bbs^{\sigma_I (\alpha^1) , \dots , \sigma_I (\alpha^{n-1})}_{(M^1,A^1,N^1) , \dots , (M^{n-1} , A^{n-1} , N^{n-1})} \otimes_R \bbs^{\nu_{[\ell] \backslash I} (\alpha^n)}_{(M^n,A^n,N^n)},$$
    where $I$ ranges over all subsets of $[\ell-1]$.
\end{lemma}

\begin{proof}
    Proceed by induction on $\ell$, where the base case $\ell = 1$ is trivial. Assume now that $\ell > 1$ and write each $\alpha^i = \beta^i \cdot \gamma^i$, where $\gamma^i := p_\ell (\alpha^i)$. By Lemma \ref{lem:multiSchuriltrationSES} there is a canonical short exact sequence
    $$0 \to \bbs^{\beta^1 \cdot \gamma^1 , \dots , \beta^{n-1} \cdot \gamma^{n-1} , \beta^n}_{\underline{M} , \underline{A} , \underline{N}} \otimes_R \bbs^{\gamma^n}_{M^n,A^n,N^n} $$
    $$\to \bbs^{\beta^1 \cdot \gamma^1 , \dots , \beta^n \cdot \gamma^n}_{\underline{M} , \underline{A} , \underline{N}} \to \bbs^{\beta^1 \odot \gamma^1, \dots , \beta^{n-1} \odot \gamma^{n-1} , \beta^n \cdot \gamma^n}_{\underline{M} , \underline{A} , \underline{N}} \to 0.$$
    By the inductive hypothesis, the multi-Schur module $\bbs^{\beta^1 \cdot \gamma^1 , \dots , \beta^{n-1} \cdot \gamma^{n-1} , \beta^n}_{\underline{M} , \underline{A} , \underline{N}}$ admits a filtration with graded pieces of the form
    $$\bbs^{\sigma_I (\alpha^1) , \dots , \sigma_I (\alpha^{n-1})}_{(M^1,A_1,N^1) , \dots ,(M^{n-1} ,  A^{n-1} , N^{n-1})} \otimes_R \bbs^{\nu_{[\ell-1] \backslash I} (\beta^n)}_{(M^n, A^n , N^n)},$$
    where $I \subset [\ell - 2]$. Likewise, the multi-Schur module $ \bbs^{\beta^1 \odot \gamma^1, \dots , \beta^{n-1} \odot \gamma^{n-1} , \beta^n \cdot \gamma^n}_{\underline{M} , \underline{A} , \underline{N}}$ admits a filtration with graded pieces of the form
    $$\bbs^{\sigma_J (\beta^1 \odot \gamma^1) , \dots , \sigma_J (\beta^{n-1} \odot \gamma^{n-1})}_{(M^1, A_1 , N^1) , \dots , (M^{n-1} , A^{n-1} , N^{n-1})} \otimes_R \bbs^{\nu_{[\ell-1] \backslash J} (\beta^n \cdot \gamma^n)}_{(M^n, A^n, N^n)},$$
    where $J \subset [\ell -2]$. Note that ranging over all $I \subset [\ell-2]$ as in the first case is the same as ranging over all $I \subset [\ell-1]$ with $\ell-1 \notin I$, and ranging over all $J$ as in the second case is equivalent to ranging over all $J \subset [\ell-1]$ with $\ell-1 \in J$. Combining both of these filtrations with Corollary \ref{cor:splicingFiltrations} yields the result. 
\end{proof}

\begin{lemma}\label{lem:tensorFiltration}
    Let $\alpha$ be any $\ell$-partitioned composition. Let $A$ be a Koszul $R$-algebra and $N$ (resp. $N$) any left (resp. right) Koszul $A$-module. For any $I \subset [\ell-1]$, the module $\bbs^{\nu_I (\alpha)}_{M,A,N}$ admits a canonical filtration with associated graded pieces of the form
    $$\bbs^{\sigma_J (\alpha)}_{M,A,N},$$
    where $J \subset [\ell-1]$ ranges over all subsets with $J \subseteq I$. 
\end{lemma}

\begin{proof}
    Proceed by induction on $|I|$, with base case $I = \varnothing$ being vacuous since there is an equality $\bbs^{\nu_{\varnothing} (\alpha)}_{M,A,N} = \bbs^{\sigma_\varnothing (\alpha)}_{M,A,N}$. For $|I| > 0 $, let $j \in I$ be the largest element of $I$. Write $\alpha = p_{< j} (\alpha) \cdot p_{\geq j} (\alpha)$ and consider the short exact sequence
    $$0 \to \bbs^{\nu_{I \backslash j} (\alpha)}_{M,A,N} \to \bbs^{\nu_{I} (\alpha)}_{M,A,N} \to \bbs^{\nu_{I \backslash j} ( p_{< j} (\alpha) \odot p_{\geq j} (\alpha))}_{M , A ,N} \to 0.$$
    Let us consider the filtrations of the outer two terms: by the inductive hypothesis, $\bbs^{\nu_{I \backslash j} (\alpha)}_{M,A,N}$ has a filtration with graded pieces of the form
    $$\bbs^{\sigma_{K} (\alpha)}_A,$$
    where $K \subset [\ell -1 ]$ ranges over all subsets with $K \subset I \backslash j$. Likewise, the term $\bbs^{\nu_{I \backslash j} ( p_{< j} (\alpha) \odot p_{\geq j} (\alpha))}_{M , A ,N}$ has a filtration with graded pieces of the form
    $$\bbs^{\sigma_{L} ( p_{< j} (\alpha) \odot p_{\geq j} (\alpha))}_{M , A ,N},$$
    where $L \subset [\ell - 2]$ ranges over all subsets $L \subset s_j (I \backslash j)$. Notice that $p_{< j} (\alpha) \odot p_{\geq j} (\alpha))$ is the same as $\sigma_j (\alpha)$, in which case ranging over all $K$ and $L$ as above is the same as just ranging over all subsets $J \subset [\ell -1]$ with $J \subset I$. The result thus follows from Corollary \ref{cor:splicingFiltrations}.
\end{proof}

\begin{example}
    Let $A = S(V)$, the symmetric algebra on some free $R$-module $V$. Then $A$ is a $\gl (V)$-representation and Lemma \ref{lem:tensorFiltration} implies that there is a $\gl (V)$-equivariant filtration of the tensor power $V^{\otimes d}$ with graded pieces of the form $\bbs^{\alpha}_A$, where $\alpha$ ranges over all compositions of $d$ (since $V^{\otimes d} = \bbs^{\nu_{[d-1]} (1^d)}_{S(V)}$).

    Assuming $R$ is a field of characteristic $0$, the ring $R [ \gl (V)]$ is semisimple and hence this yields a $\gl(V)$-equivariant direct sum decompositions
    $$V^{\otimes d} = \bigoplus_{|\alpha|=d} \bbs^{\alpha}_A.$$
\end{example}

Finally, we arrive at the main result of this section; this result furnishes the multi-Schur functors with a canonical filtration whose associated graded pieces are easily described as tensor products of the ribbon Schur functors of Section \ref{sec:ribbonSchurs}.

\begin{theorem}\label{thm:multiSchurFactors}
        Let $\alpha^1 , \dots , \alpha^n$ be a sequence of $\ell$-partitioned compositions and $A^1 , \dots , A^n$ be a sequence of Koszul $R$-algebras and $N^i$ (resp. $M^i$) be Koszul left (resp. right) $A^i$-modules for $1 \leq i \leq n$. Then the multi-Schur module $\bbs^{\underline{\alpha}}_{\underline{M} , \underline{A} , \underline{N}}$ admits a canonical filtration with graded pieces of the form
        $$\bbs^{\sigma_{I_1} (\alpha^1)}_{(M^1,A^1,N^1)} \otimes_R \bbs^{\sigma_{I_2} (\alpha^2)}_{(M^2, A^2, N^2)} \otimes_R \cdots \otimes_R \bbs^{\sigma_{I_n} (\alpha^n)}_{(M^n,A^n, N^n)},$$
        where the subsets $I_1, \dots , I_n \subset [\ell -1 ]$ range over all choices such that $I_1 \cap I_2 \cap \cdots \cap I_n = \varnothing$. 
        % \begin{itemize}
        %     \item Range over all \defi{non-strict} chains of inclusions $I_1 \subseteq I_2 \subseteq \cdots \subseteq I_{n-1} \subset [\ell -1]$, then
        %     \item For each $1 \leq i \leq n-1$, range over all subsets $J_i \subseteq I_i$.
        % \end{itemize}
\end{theorem}

\begin{remark}
    Notice that the parameterizing set for the associated graded pieces is inherently symmetric under permuting the tensor factors, which is expected by the invariance of the multi-Schur module on the ordering of the inputs.

    In general, the number of ways to choose $n$ subsets of $[\ell-1]$ with no common intersection is $(2^n-1)^{\ell-1}$, so the number of associated graded pieces grows exponentially with respect to both $n$ and $\ell$. 
\end{remark}

\begin{remark}
    The adjective ``canonical"  used in Theorem \ref{thm:multiSchurFactors} means that this filtration is really true at the level of functors; in other words, the functor $\bbs^{\underline{\alpha}}$ which takes as inputs $n$-tuples of Koszul $R$-algebras and outputs the associated multi-Schur module admits a canonical filtration whose associated graded pieces are given by the functors of the form
    $$\bbs^{\sigma_{I_1} (\alpha^1)} \otimes_R \bbs^{\sigma_{I_2} (\alpha^2)} \otimes_R \cdots \otimes_R \bbs^{\sigma_{I_n} (\alpha^n)}.$$
\end{remark}

% \begin{cor}
%     Let $\alpha$, $\beta$ be two $\ell$-partitioned compositions and $A$, $B$ compatibly split Koszul $R$-algebras. Then there is a canonical filtration of $\bbs^{\alpha , \beta}_{A,B}$ with associated graded pieces of the form
%     $$\bbs^{\sigma_I (\alpha)}_A \otimes_R \bbs^{\sigma_J (\beta)}_{B},$$
%     where $I$ and $J$ are any two disjoint subsets of $[\ell ]$.
% \end{cor}

% \begin{example}
%     The following example should help illustrate the selection of the subsets $J_i$ for each $1 \leq i \leq n-1$; in this case we will choose $n = 3$ and $\ell = 3$. Then we range over all length $2$ non-strict chains of $\{ 1 , 2 \}$ to obtain:
%     $$\varnothing \subseteq \varnothing, \quad \varnothing \subseteq \{ 1 \},$$
%     $$\varnothing \subseteq \{ 2 \}, \quad \varnothing \subseteq \{1 , 2 \},$$
%     $$\{ 1 \} \subseteq \{ 1 \}, \quad \{1 \} \subseteq \{ 1 , 2 \},$$
%     $$\{ 2 \} \subseteq \{2 \}, \quad \{ 2 \} \subseteq \{ 1 , 2 \},$$
%     $$\{ 1 , 2 \} \subseteq \{ 1 , 2 \}.$$
%     For each of the above chains, we then range over all subsets of each piece of the chain. One can easily count that this yields a total of $49$ associated graded pieces.
% \end{example}

\begin{proof}[Proof of Theorem \ref{thm:multiSchurFactors}]
     Proceed by induction on $n$ (the number of Koszul algebras), with base case $n = 2$. To prove the base case, we induct on $\ell$ (where the case $\ell = 1$ is vacuous). Let $\ell > 1$ and write $\alpha = \alpha' \cdot p_\ell (\alpha)$ and $\beta = \beta' \cdot p_\ell (\beta)$. By Lemma \ref{lem:multiSchuriltrationSES} there is a short exact sequence
    $$0 \to \bbs^{\alpha , \beta'}_{A,B} \otimes_R \bbs^{p_\ell (\beta)}_B \to \bbs^{\alpha,\beta}_{A,B} \to  \bbs^{\alpha' \odot p_\ell (\alpha) , \beta}_{A,B} \to 0.$$
    By the inductive hypothesis, the outer $2$ terms admit filtrations with graded pieces of the correct form. This establishes the base case.

    Assume now that $n > 2$. By Lemma \ref{lem:multiSchurFiltFirstStep} the multi-Schur module $\bbs^{\alpha^1 , \dots , \alpha^n}_{A^1 , \dots , A^n}$ admits a canonical filtration with associated graded pieces of the form
    $$\bbs^{\sigma_I (\alpha^1) , \dots , \sigma_I (\alpha^{n-1})}_{A^1 , \dots , A^{n-1}} \otimes_R \bbs^{\nu_{[\ell] \backslash I} (\alpha^n)}_{A^n},$$
    where $I$ ranges over all subsets of $[\ell-1]$. By the inductive hypothesis, each of the modules $\bbs^{\sigma_I (\alpha^1) , \dots , \sigma_I (\alpha^{n-1})}_{A^1 , \dots , A^{n-1}}$ admits a canonical filtration with associated graded pieces of the form
    $$\bbs^{\sigma_{I_1} (\alpha^1)}_{A^1} \otimes_R \bbs^{\sigma_{I_2} (\alpha^2)}_{A^2} \otimes_R \cdots \otimes_R \bbs^{\sigma_{I_{n-1}} (\alpha^{n-1})}_{A^{n-1}},$$
    where the above ranges over all choices of $I_1 , \dots , I_{n-1}$ with $I_1 \cap \cdots \cap I_{n-1} = I$. On the other hand, the module $\bbs^{\nu_{[\ell-1] \backslash I} (\alpha^n)}_{A^n}$ admits a canonical filtration with associated graded pieces of the form $\bbs^{\sigma_{I_n} (\alpha^n)}_{A^n}$, where $I_n$ ranges over all subsets $I_n \subset [\ell-1]$ such that $I_n \cap I = \varnothing$. This is overall the same thing as ranging over all choices $I_1 , \dots , I_n \subset [\ell -1]$ such that $I_1 \cap \cdots \cap I_n = \varnothing$. 
\end{proof}

\begin{example}
    Let $A$ and $B$ be Koszul algebras. Let us use the argument of Theorem \ref{thm:multiSchurFactors} to filter the multi-Schur module $\bbs^{(1^3),(1^3)}_{A,B}$, helping to illustrate the idea of the proof in a concrete setting. Recall by Example \ref{ex:theSESExample} that we have the short exact sequence
    $$\bbs^{(1^3),(1^2)}_{A , B} \otimes_R B_1 \to \bbs^{(1^3),(1^3)}_{A , B} \to \bbs^{(1,2),(1^3)}_{A, B} \to 0.$$
    Iteratively applying this sequence to the outer two terms yields
    $$0 \to A_3 \otimes_R (B_1)^{\otimes 3} \to \bbs^{(1^3),(1^2)}_{A , B} \to \bbs^{(2,1)}_A \otimes_R \bbs^{(1^2)}_B \otimes_R B_1 \to 0,$$
    $$0 \to \bbs^{(1,2)}_A \otimes_R B_1 \otimes_R \bbs^{(1^2)}_A \to \bbs^{(1,2),(1^3)}_{A, B} \to A_3 \otimes_R \bbs^{(1^3)}_B \to 0.$$
    The terms involving $B$ appearing on the ends of these short exact sequences are precisely the filtration factors of Lemma \ref{lem:multiSchurFiltFirstStep}, and are filtered further by Lemma \ref{lem:tensorFiltration}, yielding the filtration factors of Theorem \ref{thm:multiSchurFactors}.
\end{example}

\begin{remark}
    Let $2^{[\ell -1]}$ denote the Boolean poset on $n$ elements. Then the product $(2^{[\ell-1]})^{\times n}$ is naturally a ranked poset (in fact, isomorphic to the Boolean poset $2^{[\ell-1] \times [n]}$), and the set $\cat{S}$ of all tuples $(I_1 , \dots , I_n)$ with $I_1 \cap \cdots \cap I_n = \varnothing$ is a ranked subposet. Choosing any total order $<$ refining the partial order on this subposet, the filtration of the multi-Schur module $\bbs^{\alpha^1 , \dots , \alpha^n}_{A^1 , \dots , A^n}$ is parameterized by $<$. In other words, the filtration is of the form
    $$\{ F_{(I_1 , \dots , I_n)} \}_{(I_1 , \dots , I_n) \in \cat{S}},$$
    with
    $$F_{(I_1 , \dots , I_n)}/F_{\text{pred} (I_1 , \dots , I_n)} \cong \bbs^{\sigma_{I_1} (\alpha^1)}_{A^1} \otimes_R \cdots \otimes_R \bbs^{\sigma_{I_n} (\alpha^n)}_{A^n}.$$
    In the above, $\text{pred}$ denotes the predecessor function (ie, the largest element strictly smaller with respect to $<$).
\end{remark}

\begin{example}\label{ex:mSchurFiltrationEx1}
Let $\alpha = (1,2) \cdot (2) \cdot (1)$ and $\beta = (3) \cdot (4,2) \cdot (2)$ be two $3$-partitioned compositions and let us compute the composition factors of the multi-Schur module $\bbs^{\alpha,\beta}_{A,B}$ for any two Koszul algebras $A$ and $B$. The subposet (in fact, meet semilattice) of $(2^{[2]})^{\times 2}$ that parametrizes the filtration factors has Hasse diagram
    % https://q.uiver.app/?q=WzAsOSxbMSwyLCIoXFx2YXJub3RoaW5nLFxcdmFybm90aGluZykiXSxbMCwxLCIoXFx7IDEgXFx9ICwgXFx2YXJub3RoaW5nKSJdLFsxLDEsIiggXFx7IDIgXFx9ICwgXFx2YXJub3RoaW5nKSJdLFsyLDEsIihcXHZhcm5vdGhpbmcgLCBcXHsgMSBcXH0gKSJdLFszLDEsIihcXHZhcm5vdGhpbmcgLCBcXHsgMiBcXH0gKSJdLFswLDAsIihcXHsxLDJcXH0gLCBcXHZhcm5vdGhpbmcpIl0sWzEsMCwiKFxcezEgXFx9ICwgXFx7IDIgXFx9KSJdLFsyLDAsIihcXHsyIFxcfSAsIFxceyAxIFxcfSkiXSxbMywwLCIoXFx2YXJub3RoaW5nICwgXFx7MSwyIFxcfSkiXSxbMCwxLCIiLDAseyJzdHlsZSI6eyJoZWFkIjp7Im5hbWUiOiJub25lIn19fV0sWzAsMiwiIiwyLHsic3R5bGUiOnsiaGVhZCI6eyJuYW1lIjoibm9uZSJ9fX1dLFswLDMsIiIsMix7InN0eWxlIjp7ImhlYWQiOnsibmFtZSI6Im5vbmUifX19XSxbMCw0LCIiLDIseyJzdHlsZSI6eyJoZWFkIjp7Im5hbWUiOiJub25lIn19fV0sWzEsNSwiIiwwLHsic3R5bGUiOnsiaGVhZCI6eyJuYW1lIjoibm9uZSJ9fX1dLFsxLDYsIiIsMCx7InN0eWxlIjp7ImhlYWQiOnsibmFtZSI6Im5vbmUifX19XSxbNSwyLCIiLDAseyJzdHlsZSI6eyJoZWFkIjp7Im5hbWUiOiJub25lIn19fV0sWzIsNywiIiwwLHsic3R5bGUiOnsiaGVhZCI6eyJuYW1lIjoibm9uZSJ9fX1dLFszLDgsIiIsMix7InN0eWxlIjp7ImhlYWQiOnsibmFtZSI6Im5vbmUifX19XSxbNCw4LCIiLDEseyJzdHlsZSI6eyJoZWFkIjp7Im5hbWUiOiJub25lIn19fV0sWzQsNiwiIiwxLHsic3R5bGUiOnsiaGVhZCI6eyJuYW1lIjoibm9uZSJ9fX1dLFszLDcsIiIsMSx7InN0eWxlIjp7ImhlYWQiOnsibmFtZSI6Im5vbmUifX19XV0=
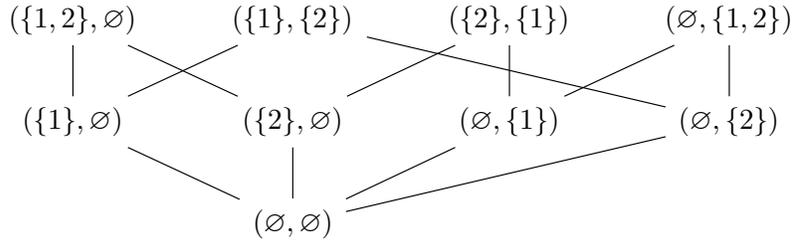
\[\begin{tikzcd}
	{(\{1,2\} , \varnothing)} & {(\{1 \} , \{ 2 \})} & {(\{2 \} , \{ 1 \})} & {(\varnothing , \{1,2 \})} \\
	{(\{ 1 \} , \varnothing)} & {( \{ 2 \} , \varnothing)} & {(\varnothing , \{ 1 \} )} & {(\varnothing , \{ 2 \} )} \\
	& {(\varnothing,\varnothing)}
	\arrow[no head, from=3-2, to=2-1]
	\arrow[no head, from=3-2, to=2-2]
	\arrow[no head, from=3-2, to=2-3]
	\arrow[no head, from=3-2, to=2-4]
	\arrow[no head, from=2-1, to=1-1]
	\arrow[no head, from=2-1, to=1-2]
	\arrow[no head, from=1-1, to=2-2]
	\arrow[no head, from=2-2, to=1-3]
	\arrow[no head, from=2-3, to=1-4]
	\arrow[no head, from=2-4, to=1-4]
	\arrow[no head, from=2-4, to=1-2]
	\arrow[no head, from=2-3, to=1-3]
\end{tikzcd}\]
This translates to filtration factors of the following form:
% https://q.uiver.app/?q=WzAsOSxbMSwyLCJcXGJic157KDEsMiwyLDEpfV9BIFxcb3RpbWVzX1IgXFxiYnNeeygzLDQsMiwyKX1fQiJdLFswLDEsIlxcYmJzXnsoMSw0LDEpfV9BIFxcb3RpbWVzX1IgXFxiYnNeeygzLDQsMiwyKX1fQiJdLFsxLDEsIlxcYmJzXnsoMSwyLDMpfV9BIFxcb3RpbWVzX1IgXFxiYnNeeygzLDQsMiwyKX1fQiJdLFsyLDEsIlxcYmJzXnsoMSwyLDIsMSl9X0EgXFxvdGltZXNfUiBcXGJic157KDcsMiwyKX1fQiJdLFszLDEsIlxcYmJzXnsoMSwyLDIsMSl9X0EgXFxvdGltZXNfUiBcXGJic157KDMsNCw0KX1fQiJdLFswLDAsIlxcYmJzXnsoMSw1KX1fQSBcXG90aW1lc19SIFxcYmJzXnsoMyw0LDIsMil9X0IiXSxbMSwwLCJcXGJic157KDEsNCwxKX1fQSBcXG90aW1lc19SIFxcYmJzXnsoMyw0LDQpfV9CIl0sWzIsMCwiXFxiYnNeeygxLDIsMyl9X0EgXFxvdGltZXNfUiBcXGJic157KDcsMiwyKX1fQiJdLFszLDAsIlxcYmJzXnsoMSwyLDIsMSl9X0FcXG90aW1lc19SIFxcYmJzXnsoNyw0KX1fQiJdLFswLDEsIiIsMCx7InN0eWxlIjp7ImhlYWQiOnsibmFtZSI6Im5vbmUifX19XSxbMCwyLCIiLDIseyJzdHlsZSI6eyJoZWFkIjp7Im5hbWUiOiJub25lIn19fV0sWzAsMywiIiwyLHsic3R5bGUiOnsiaGVhZCI6eyJuYW1lIjoibm9uZSJ9fX1dLFswLDQsIiIsMix7InN0eWxlIjp7ImhlYWQiOnsibmFtZSI6Im5vbmUifX19XSxbMSw1LCIiLDAseyJzdHlsZSI6eyJoZWFkIjp7Im5hbWUiOiJub25lIn19fV0sWzEsNiwiIiwwLHsic3R5bGUiOnsiaGVhZCI6eyJuYW1lIjoibm9uZSJ9fX1dLFs1LDIsIiIsMCx7InN0eWxlIjp7ImhlYWQiOnsibmFtZSI6Im5vbmUifX19XSxbMiw3LCIiLDAseyJzdHlsZSI6eyJoZWFkIjp7Im5hbWUiOiJub25lIn19fV0sWzMsOCwiIiwyLHsic3R5bGUiOnsiaGVhZCI6eyJuYW1lIjoibm9uZSJ9fX1dLFs0LDgsIiIsMSx7InN0eWxlIjp7ImhlYWQiOnsibmFtZSI6Im5vbmUifX19XSxbNCw2LCIiLDEseyJzdHlsZSI6eyJoZWFkIjp7Im5hbWUiOiJub25lIn19fV0sWzMsNywiIiwxLHsic3R5bGUiOnsiaGVhZCI6eyJuYW1lIjoibm9uZSJ9fX1dXQ==
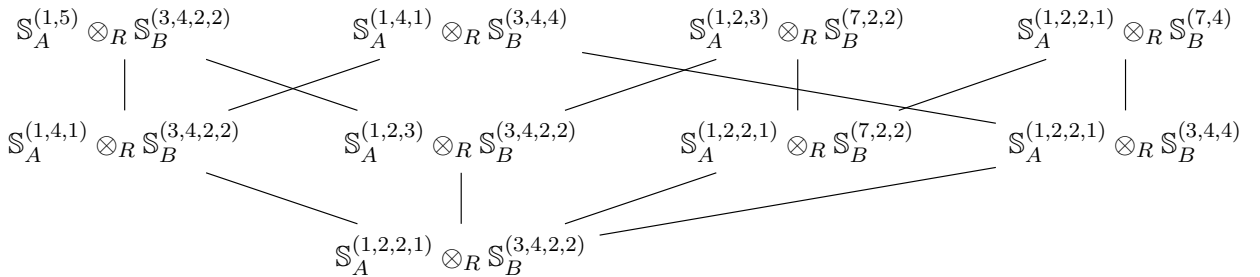
\[\begin{tikzcd}
	{\bbs^{(1,5)}_A \otimes_R \bbs^{(3,4,2,2)}_B} & {\bbs^{(1,4,1)}_A \otimes_R \bbs^{(3,4,4)}_B} & {\bbs^{(1,2,3)}_A \otimes_R \bbs^{(7,2,2)}_B} & {\bbs^{(1,2,2,1)}_A\otimes_R \bbs^{(7,4)}_B} \\
	{\bbs^{(1,4,1)}_A \otimes_R \bbs^{(3,4,2,2)}_B} & {\bbs^{(1,2,3)}_A \otimes_R \bbs^{(3,4,2,2)}_B} & {\bbs^{(1,2,2,1)}_A \otimes_R \bbs^{(7,2,2)}_B} & {\bbs^{(1,2,2,1)}_A \otimes_R \bbs^{(3,4,4)}_B} \\
	& {\bbs^{(1,2,2,1)}_A \otimes_R \bbs^{(3,4,2,2)}_B}
	\arrow[no head, from=3-2, to=2-1]
	\arrow[no head, from=3-2, to=2-2]
	\arrow[no head, from=3-2, to=2-3]
	\arrow[no head, from=3-2, to=2-4]
	\arrow[no head, from=2-1, to=1-1]
	\arrow[no head, from=2-1, to=1-2]
	\arrow[no head, from=1-1, to=2-2]
	\arrow[no head, from=2-2, to=1-3]
	\arrow[no head, from=2-3, to=1-4]
	\arrow[no head, from=2-4, to=1-4]
	\arrow[no head, from=2-4, to=1-2]
	\arrow[no head, from=2-3, to=1-3]
\end{tikzcd}\]
If for instance $A = B = S(V)$, the symmetric algebra, then the above filtration is also $\gl (V) \times \gl (V)$ equivariant (and in characteristic $0$ yields a direct sum decomposition).
\end{example}

\section{Applications}\label{sec:applications}

The following section is the reward for enduring the technical details of Sections \ref{sec:ribbonSchurs} and \ref{sec:multiSchurFunctors}; we are able to arrive at the other end with a rather robust theory that allows us to give elegant and simple closed form descriptions of higher derived invariants associated to (Veronese/Segre subalgebras of) Koszul algebras (see also \cite{BarcanescuManolache}). In subsection \ref{subsec:buildingKoszulModules} we show how to use this theory to build a large class of Koszul modules over an arbitrary Koszul algebra $A$, and specialize to the case of a polynomial ring to prove a uniform, characteristic-free regularity result for certain classes of vector bundles on projective space.

\subsection{Tor and Ext}\label{subsec:TorAndExt}

In this section we prove the following theorem, which gives some concise descriptions of Tor and Ext between pairs of Koszul modules in terms of ribbon Schur functors.

\begin{theorem}\label{thm:explicitTorExt}
    Let $A$ be a Koszul algebra and $M$ (resp. $N$) a Koszul right (resp. left) $A$-module of initial degree $t$ (resp. $s$). Then there is a canonical isomorphism of $A$-modules
    $$\tor_i^A (M,N) \cong \bbs^{(1^{i})}_{M,A,N} \quad \text{for all} \ i > 0$$
    If $M$ is instead a Koszul left $A$-module, then there is a canonical isomorphism
    $$\ext^i_A (M,N) \cong \bbs^{(i)}_{M^!,A^!,(N^*)^!} \quad \text{for all} \ i >0.$$ 
    In particular, both of the above modules are flat $R$-modules annihilated by $A_+$.
\end{theorem}

\begin{proof}
    By the definition of Tor combined with Theorem \ref{thm:thePriddyCxWorks}, the module $\tor_i^A (M,N)$ may be computed by looking at the homology of the complex
    $$\cdots \to M \otimes_R \bbs^{(1^{i+1})}_{A,N} \to M \otimes_R \bbs^{(1^i)}_{A,N} \to M \otimes_R \bbs^{(1^{i-1})}_{A,N} \to \cdots.$$
    Splitting this complex into graded pieces, there is a commutative diagram:
% https://q.uiver.app/?q=WzAsNCxbMCwwLCJNX2ogXFxvdGltZXNfUiBcXGJic157KDFeaSl9X3tBLE59Il0sWzIsMCwiTV97aisxfSBcXG90aW1lc19SIFxcYmJzXnsoMV57aS0xfSl9X3tBLE59Il0sWzEsMSwiTV9qIFxcb3RpbWVzX1IgQV8xIFxcb3RpbWVzX1IgXFxiYnNeeygxXntpLTF9KX1fe0EuTn0iXSxbMCwyLCJcXGJic157KDEpfV97TV97XFxnZXEgan0sQX0gXFxvdGltZXNfUiBcXGJic157KDFee2ktMX0pfV97QSxOfSJdLFswLDEsIihkX2kpX3tqK2krc30iXSxbMCwyLCIiLDIseyJzdHlsZSI6eyJ0YWlsIjp7Im5hbWUiOiJob29rIiwic2lkZSI6InRvcCJ9fX1dLFsyLDEsIiIsMix7InN0eWxlIjp7ImhlYWQiOnsibmFtZSI6ImVwaSJ9fX1dLFszLDIsIiIsMix7InN0eWxlIjp7InRhaWwiOnsibmFtZSI6Imhvb2siLCJzaWRlIjoidG9wIn19fV1d
\[\begin{tikzcd}
	{M_j \otimes_R \bbs^{(1^i)}_{A,N}} && {M_{j+1} \otimes_R \bbs^{(1^{i-1})}_{A,N}} \\
	& {M_j \otimes_R A_1 \otimes_R \bbs^{(1^{i-1})}_{A.N}} \\
	{\bbs^{(1)}_{M_{\geq j},A} \otimes_R \bbs^{(1^{i-1})}_{A,N}}
	\arrow["{(d_i)_{j+i+s}}", from=1-1, to=1-3]
	\arrow[hook, from=1-1, to=2-2]
	\arrow[two heads, from=2-2, to=1-3]
	\arrow[hook, from=3-1, to=2-2]
\end{tikzcd}\]
This implies that there is an equality
$$\ker (d_i)_{i+j+s} = \left( M_j \otimes_R \bbs^{(1^i)}_{A,N} \right) \cap \left( \bbs^{(1)}_{M_{\geq j} , A} \otimes_R \bbs^{(1^{i-1})}_{A,N} \right).$$
By Observation \ref{obs:intersectingSchurs}, this latter intersection is precisely $\bbs^{(1^i)}_{(M_{\geq j} , A , N)}$. On the other hand, by definition of the Priddy differential there is also a commutative diagram for all $j > s$:
% https://q.uiver.app/?q=WzAsNCxbMCwwLCJNX3tqLTF9IFxcb3RpbWVzX1IgXFxiYnNeeygxXntpKzF9KX1fe0EsTn0iXSxbMSwwLCJNX2ogXFxvdGltZXNfUiBcXGJic157KDFeaSl9X3tBLE59Il0sWzEsMSwiTV9qIFxcb3RpbWVzX1IgQV8xIFxcb3RpbWVzX1IgXFxiYnNeeygxXntpLTF9KX1fe0EsTn0iXSxbMCwxLCJcXGJic157KDFeaSl9X3tNX3tcXGdlcSBqfSxBLE59Il0sWzAsMSwiKGRfe2krMX0pX3tqK2krc30iXSxbMSwyLCIiLDAseyJzdHlsZSI6eyJ0YWlsIjp7Im5hbWUiOiJob29rIiwic2lkZSI6InRvcCJ9fX1dLFswLDMsIiIsMix7InN0eWxlIjp7ImhlYWQiOnsibmFtZSI6ImVwaSJ9fX1dLFszLDIsIiIsMix7InN0eWxlIjp7InRhaWwiOnsibmFtZSI6Imhvb2siLCJzaWRlIjoidG9wIn19fV1d
\[\begin{tikzcd}
	{M_{j-1} \otimes_R \bbs^{(1^{i+1})}_{A,N}} & {M_j \otimes_R \bbs^{(1^i)}_{A,N}} \\
	{\bbs^{(1^i)}_{M_{\geq j},A,N}} & {M_j \otimes_R A_1 \otimes_R \bbs^{(1^{i-1})}_{A,N}}
	\arrow["{(d_{i+1})_{j+i+s}}", from=1-1, to=1-2]
	\arrow[hook, from=1-2, to=2-2]
	\arrow[two heads, from=1-1, to=2-1]
	\arrow[hook, from=2-1, to=2-2]
\end{tikzcd}\]
This implies that there is also an equality
$$\im (d_{i+1})_{j+i+s} = \bbs^{(1^i)}_{M_{\geq j} , A , N}.$$
Putting both of the above equalities together, it follows that
$$\tor_i^A (M , N)_{i+j+s} = \begin{cases}
\bbs^{(1^i)}_{M , A , N} & \text{if} \ j = s, \\
0 & \text{otherwise}.
\end{cases}$$

To prove the isomorphism for Ext, recall first that there is a canonical isomorphism
$$\ext^i_A (M , N) \cong \left( \tor_i^A (N^* , M) \right)^*.$$
By Observation \ref{obs:gradedDualKoszul}, the graded dual $N^*$ is a Koszul right $A$-module, and by the isomorphism just proved for Tor there is an isomorphism
$$\tor_i^A (N^* , M) \cong \bbs^{(1^i)}_{N^*,A,M}.$$
Dualizing and using the isomorphism of Corollary \ref{cor:doubleModuleInputProperties}(3), the result follows immediately.
\end{proof}

\begin{example}
    Assume that $A$ is any Koszul algebra and recall that
    $$\bbs_{A,A_+^d}^{(1^i)} = \bbs^{(1^i,d)}_A.$$
    By Theorem \ref{thm:thePriddyCxWorks}, the minimal free resolution of $A_+^d$ thus has the form
    $$\cdots \to A \otimes_R \bbs^{(1^i,d)}_A \to A \otimes_R \bbs^{(1^{i-1},d)}_A \to \cdots \to A \otimes_R A_d \to A_+^d.$$
    The ribbon Schur module $\bbs^{(1^i,d)}_{A}$ may be presented as the cokernel of the composition
    $$(A^!)^*_{i+2} \otimes_R A_{d-2} \to (A^!)^*_{i+1} \otimes_R A_1 \otimes_R A_{d-2} \to (A^!)^*_{i+1} \otimes_R A_{d-1},$$
    in which case we see that there is an isomorphism with the ribbon Schur functor:
    $$\bbs^{(1^i,d)}_{A} \cong L^A_{i-1,d-1},$$
    where the module $L^A_{i+1,d-1}$ is defined as in \cite{faber2020canonical}. Thus Theorem \ref{thm:explicitTorExt} at least recovers the minimal free resolution of powers of the maximal ideal of a Koszul algebra constructed in \cite{faber2020canonical}.
\end{example}

\begin{cor}
    Let $A$ be any Koszul algebra such that $\bbs^{(1^i)}_{A} = 0$ for all $i > 1$. Then every Koszul module over $A$ is a flat $R$-module. If $A$ is $R$-projective, then every Koszul module over $A$ is $R$-projective.
\end{cor}

\subsection{Veronese Subalgebras}\label{subsec:VeroneseCase}

Recall the definition of the Veronese subalgebra as in Definition \ref{def:VeroneseAndSegre}. The following observation shows that the operation $(-)^{(d)}$ on compositions as defined in Definition \ref{def:operationsOfComps} interacts functorially with the formation of the ribbon Schur functor:

\begin{obs}\label{obs:VeroneseCommutes}
    Let $A$ be any Koszul algebra and $M$ (resp. $N$) any Koszul right (resp. left) $A$-module. For any integer $d > 0$ and integers $i,j \in \bbz$ there are isomorphisms
    $$\bbs^{\alpha}_{\veralg{d}} = \bbs^{\alpha^{(d)}}_{A}, \quad \text{and} \quad \bbs^{ \alpha }_{\veralgg{M}{d} , \veralg{d}, \veralgg{N}{d}} \cong \bbs^{ \alpha^{(d)}}_{M,A,N}.$$
\end{obs}

Combining Theorem \ref{thm:explicitTorExt} with Observation \ref{obs:VeroneseCommutes} immediately yields:

\begin{cor}\label{cor:veroneseDerivedInvs}
        Let $A$ be any Koszul algebra and $M$ (resp. $N$) any Koszul right (resp. left) $A$-module. For any integer $d > 0$ there is an isomorphism
        $$\tor_i^{\veralg{d}} ( \veralgg{M}{d} , \veralgg{N}{d}) = \bbs^{(d^i)}_{M,A,N},$$
        where $t$ (resp. $s$) is the initial degree of $M$ (resp. $N$). In particular, there are canonical isomorphisms
        $$\tor_i^{\veralg{d}} (\vermod{d}{r} , \vermod{d}{r'} ) \cong \bbs^{(r,d^i,r')}_{A}.$$
\end{cor}

\begin{remark}
    The isomorphism  
    $$\tor_i^{\veralg{d}} (\vermod{d}{r} , \vermod{d}{r'} ) \cong \bbs^{(r,d^i,r')}_{A}.$$
    was originally proved in the case $A = S(V)$ (the symmetric algebra) in the work \cite{almousa2022equivariant}. However, one notices that the original proof of this isomorphism does not invoke anything more than the Koszulness properties of the symmetric algebra (and its truncations), which leads to the generalization presented in Corollary \ref{cor:veroneseDerivedInvs}.
\end{remark}

\subsection{Segre Subalgebras}\label{subsec:SegreCase}

In this subsection, we apply the construction of multi-Schur modules and their filtrations to study Segre products of Koszul algebras. The following observation is a straightforward translation of the quadratic dual for a Koszul $R$-algebra:

\begin{obs}
    Let $A^1, \dots , A^n$ be a collection of Koszul algebras and $M^i$ a Koszul left $A^i$-module for $1 \leq i \leq n$. Then there are isomorphisms of $R$-modules
    $$\left( (A^1 \circ \cdots \circ A^n)^! \right)^*_i = \bbs^{(1^i)}_{A^1 , \dots , A^n}, \quad \text{and}$$
    $$\left( (M^1 \circ \cdots \circ M^n)^! \right)^*_i = \bbs^{(1^i)}_{(A^1,M^1) , \dots , (A^n , M^n)}.$$
\end{obs}

\begin{example}
    Let $A$ be any Koszul algebra and consider the filtration of Theorem \ref{thm:multiSchurFactors} applied to computing $(A^{[2] \ !})^*_3 = \bbs^{(1^3)}_{A,A}$. The poset parametrizing the filtration terms are the same as those of Example \ref{ex:mSchurFiltrationEx1}, which yields the filtration factors
    % https://q.uiver.app/?q=WzAsOSxbMSwyLCJcXGJic157KDFeMyl9X0EgXFxvdGltZXNfUiBcXGJic157KDFeMyl9X0EiXSxbMCwxLCJcXGJic157KDIsMSl9X0EgXFxvdGltZXNfUiBcXGJic157KDFeMyl9X0EiXSxbMSwxLCJcXGJic157KDEsMil9X0EgXFxvdGltZXNfUiBcXGJic157KDFeMyl9X0EiXSxbMiwxLCJcXGJic157KDFeMyl9X0EgXFxvdGltZXNfUiBcXGJic157KDIsMSl9X0EiXSxbMywxLCJcXGJic157KDFeMyl9X0EgXFxvdGltZXNfUiBcXGJic157KDEsMil9X0EiXSxbMCwwLCJcXGJic157KDMpfV9BIFxcb3RpbWVzX1IgXFxiYnNeeygxXjMpfV9BIl0sWzEsMCwiXFxiYnNeeygyLDEpfV9BIFxcb3RpbWVzX1IgXFxiYnNeeygxLDIpfV9BIl0sWzIsMCwiXFxiYnNeeygxLDIpfV9BIFxcb3RpbWVzX1IgXFxiYnNeeygyLDEpfV9BIl0sWzMsMCwiXFxiYnNeeygxXjMpfV9BIFxcb3RpbWVzX1IgXFxiYnNeeygzKX1fQSJdLFswLDEsIiIsMCx7InN0eWxlIjp7ImhlYWQiOnsibmFtZSI6Im5vbmUifX19XSxbMCwyLCIiLDIseyJzdHlsZSI6eyJoZWFkIjp7Im5hbWUiOiJub25lIn19fV0sWzAsMywiIiwyLHsic3R5bGUiOnsiaGVhZCI6eyJuYW1lIjoibm9uZSJ9fX1dLFswLDQsIiIsMix7InN0eWxlIjp7ImhlYWQiOnsibmFtZSI6Im5vbmUifX19XSxbMSw1LCIiLDAseyJzdHlsZSI6eyJoZWFkIjp7Im5hbWUiOiJub25lIn19fV0sWzEsNiwiIiwwLHsic3R5bGUiOnsiaGVhZCI6eyJuYW1lIjoibm9uZSJ9fX1dLFs1LDIsIiIsMCx7InN0eWxlIjp7ImhlYWQiOnsibmFtZSI6Im5vbmUifX19XSxbMiw3LCIiLDAseyJzdHlsZSI6eyJoZWFkIjp7Im5hbWUiOiJub25lIn19fV0sWzMsOCwiIiwyLHsic3R5bGUiOnsiaGVhZCI6eyJuYW1lIjoibm9uZSJ9fX1dLFs0LDgsIiIsMSx7InN0eWxlIjp7ImhlYWQiOnsibmFtZSI6Im5vbmUifX19XSxbNCw2LCIiLDEseyJzdHlsZSI6eyJoZWFkIjp7Im5hbWUiOiJub25lIn19fV0sWzMsNywiIiwxLHsic3R5bGUiOnsiaGVhZCI6eyJuYW1lIjoibm9uZSJ9fX1dXQ==
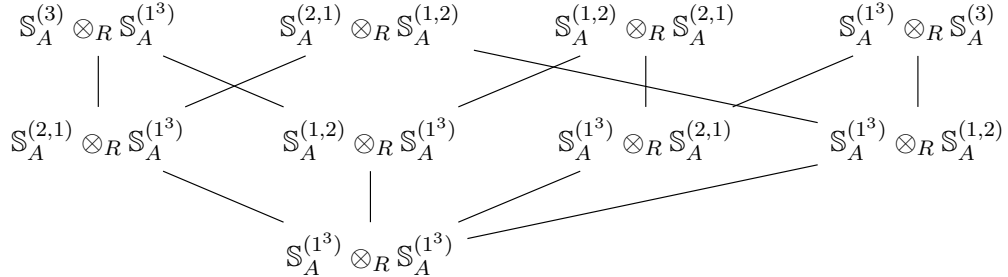
\[\begin{tikzcd}
	{\bbs^{(3)}_A \otimes_R \bbs^{(1^3)}_A} & {\bbs^{(2,1)}_A \otimes_R \bbs^{(1,2)}_A} & {\bbs^{(1,2)}_A \otimes_R \bbs^{(2,1)}_A} & {\bbs^{(1^3)}_A \otimes_R \bbs^{(3)}_A} \\
	{\bbs^{(2,1)}_A \otimes_R \bbs^{(1^3)}_A} & {\bbs^{(1,2)}_A \otimes_R \bbs^{(1^3)}_A} & {\bbs^{(1^3)}_A \otimes_R \bbs^{(2,1)}_A} & {\bbs^{(1^3)}_A \otimes_R \bbs^{(1,2)}_A} \\
	& {\bbs^{(1^3)}_A \otimes_R \bbs^{(1^3)}_A}
	\arrow[no head, from=3-2, to=2-1]
	\arrow[no head, from=3-2, to=2-2]
	\arrow[no head, from=3-2, to=2-3]
	\arrow[no head, from=3-2, to=2-4]
	\arrow[no head, from=2-1, to=1-1]
	\arrow[no head, from=2-1, to=1-2]
	\arrow[no head, from=1-1, to=2-2]
	\arrow[no head, from=2-2, to=1-3]
	\arrow[no head, from=2-3, to=1-4]
	\arrow[no head, from=2-4, to=1-4]
	\arrow[no head, from=2-4, to=1-2]
	\arrow[no head, from=2-3, to=1-3]
\end{tikzcd}\]
Suppose now that $A = S(V)$, the symmetric algebra on a vector space $V$. If $R$ is a field of characteristic $0$, this induces a $\gl (V) \times \gl (V)$-equivariant direct sum decomposition of $(A^{[2] \ !})^*_3$ into irreducibles:
$$(A^{[2] \ !})^*_3 = \left( \bigwedge^3 V \otimes \bigwedge^3 V \right) \oplus \left( \bigwedge^3 V \otimes_R \bbs^{(2,1)} (V) \right)^{\oplus 4}$$
$$\oplus \left( \bigwedge^3 V \otimes S_3 (V) \right)^{\oplus 2} \oplus \left( \bbs^{(2,1)} (V) \otimes \bbs^{(2,1)} (V) \right)^{\oplus 2}.$$
\end{example}

\begin{obs}\label{obs:multiFactorsForSameComp}
    Let $\alpha \in C(d)$ be any composition of some integer $d>0$ and $A^1 , \dots , A^n$ a sequence of Koszul $R$-algebras admitting a compatible $G^i$-action for each $1 \leq i \leq n$. Assume that $M^i$ (resp. $N^i$) is a Koszul right (resp. left) $A^i$-module admitting a compatible $G^i$-action. Then the multi-Schur module 
    $$\bbs^{\alpha}_{(M^1 , A^1 , N^1), \dots , (M^n,A^n,N^n)}$$
    admits a $G^1 \times \cdots \times G^n$-equivariant filtration with associated graded pieces of the form
    $$\bbs^{\alpha^1}_{(M^1,A^1,N^1)} \otimes_R \cdots \otimes_R \bbs^{\alpha^n}_{(M^n,A^n,N^n)},$$
    where $\alpha^1 , \dots , \alpha^n$ range over all compositions of $|\alpha|$ satisfying $\alpha^1 \w \cdots \w \alpha^n = \alpha$. 
    
    In particular, if each of the group rings $R[G^i]$ is semisimple, then there is a $G^1 \times \cdots \times G^n$-equivariant decomposition
    $$\bbs^{\alpha}_{(M^1 , A^1 , N^1), \dots , (M^n,A^n,N^n)} \cong \bigoplus_{\substack{(\alpha^1 , \dots , \alpha^n) \in C(d)^{\times n} \\ \alpha^1 \w \cdots \w \alpha^n = \alpha}} \bbs^{\alpha^1}_{(M^1,A^1,N^1)} \otimes_R \cdots \otimes_R \bbs^{\alpha^n}_{(M^n,A^n,N^n)}.$$
\end{obs}

\begin{proof}
    This is just a retranslation of Theorem \ref{thm:multiSchurFactors} combined with the fact that the refinement poset on $\alpha$ is isomorphic to the Boolean poset on $[\ell(\alpha)-1]$, and the meet operation corresponds to intersection under this isomorphism.
\end{proof}

\begin{remark}
    The appearance of compositions ranging over the refinement poset as filtration factors is likely related to the connection between Segre products and internal cohomomorphisms as discovered by Manin (see \cite{manin1987some} and \cite{manin2014topics}). Indeed, another perspective on the filtration \ref{thm:multiSchurFactors} is as a canonical filtration of the graded pieces of the cohomomorphism algebra. 
\end{remark}

Next, we use Theorem \ref{thm:multiSchurFactors} to prove a symmetric function identity; we first need to recall some notation related to Schur polynomials and establish some multi-index conventions:

\begin{notation}\label{not:symmetricFunctionNotation}
    Let $n \geq 1$ be any integer and consider sets of indeterminates $\xx^1, \dots , \xx^n$. Recall that the \defi{Schur polynomial} associated to a skew shape $\lambda / \mu$ is the polynomial
    $$s_{\lambda / \mu} (\xx) := \sum_{T \in \operatorname{SST} (\lambda / \mu)} \xx_T,$$
    where $\xx_T$ denotes the multigraded character of the semistandard tableau $T$. If $\alpha$ is a composition, the notation $s_\alpha (\xx)$ denotes the Schur polynomial corresponding to the ribbon shape determined by $\alpha$. The complete symmetric polynomial $h_d (\xx)$ is defined to be $s_{(d)} (\xx)$. 

    Given a tuple of compositions $\underline{\alpha} = (\alpha^1 , \dots , \alpha^n)$, use the notation
    $$s_{\underline{\alpha}} (\underline{\xx}) := s_{\alpha^1} (\xx^1) \cdot s_{\alpha^2} (\xx^2) \cdots s_{\alpha^n} (\xx^n).$$
    In particular, for a single composition $\alpha = (\alpha_1 , \dots , \alpha_n)$ there is the equality
    $$h_\alpha (\underline{\xx}) = h_{\alpha_1} (\xx^1) \cdot h_{\alpha_2} (\xx^2) \cdots h_{\alpha_n} (\xx^n).$$
\end{notation}

\begin{cor}\label{cor:theCharacterId}
    With notation as in Notation \ref{not:symmetricFunctionNotation}, there is an equality of symmetric polynomials
    $$h_{d^n + \alpha} (\underline{\xx}) = \sum_{i=1}^d  \sum_{\substack{\beta^1 , \dots , \beta^n \\ |\beta^1| = \cdots = |\beta^n| = i \\ \beta^1 \w \cdots \w \beta^n = 1^i}} (-1)^{i+1} h_{(d-i)^n} (\underline{\xx} ) \cdot s_{\underline{\beta} \cdot \alpha} ( \underline{\xx}).$$
\end{cor}

\begin{remark}
    In the statement of Corollary \ref{cor:theCharacterId}, if $\underline{\beta} = (\beta^1 , \dots , \beta^n)$ is a tuple of compositions and $\alpha = (\alpha_1 , \dots , \alpha_n)$ is a composition, then we use the convention $$\underline{\beta} \cdot \alpha := \left( \beta^1 \cdot (\alpha_1) , \beta^2 \cdot (\alpha_2) , \dots , \beta^n \cdot (\alpha_n)\right).$$
\end{remark}

\begin{proof}
    Assume $R = k$ is a field of characteristic $0$ and let $A = S(V)$ for some vector space $V$. Note that $A$ and $A_{\geq d} = A_+^d$ are polynomial functors for all $d \geq 1$. Taking Segre products, this means that $A^{[n]}$ and $A_+^\alpha := A_+^{\alpha_1} \circ A_+^{\alpha_2} \circ \cdots \circ A_+^{\alpha_n}$ are also polynomial functors, and hence there is an equality.
    \begin{equation}\label{eqn:charIdentity}
    \ch (A_+^\alpha) = \ch (A^{[n]}) \cdot \sum_{i \geq 0} (-1)^i \ch \left( \tor_i^{A^{[n]}} (A_+^\alpha , k) \right),\end{equation}
    where $\ch (-)$ denotes the multigraded character. By definition there are equalities
    $$\ch (A_+^\alpha) = \sum_{d \geq 1} h_{d^n + \alpha} (\underline{\xx}),$$
    $$\ch (A^{[n]} ) = \sum_{d \geq 0} h_{d^n} (\underline{\xx}).$$
    On the other hand, by Theorem \ref{thm:multiSchurFactors} there is an equality
    $$\ch \left( \tor_i^{A^{[n]}} (A_+^\alpha , k) \right) = \sum_{\substack{\beta^1 , \dots , \beta^n \\ |\beta^1| = \cdots = |\beta^n| = i \\ \beta^1 \w \cdots \w \beta^n = 1^i}} s_{\underline{\beta} \cdot \alpha} (\underline{\xx}).$$
    Combining all of these expressions and comparing degrees on each side of the equality (\ref{eqn:charIdentity}) yields the result.
\end{proof}

\begin{remark}
    When $n = 1$, this reduces to the well-known classical character identity
    $$h_{d+a} (\xx) = \sum_{i=1}^d (-1)^{i+1} h_{d-i} (\xx) s_{(1^i,a)} (\xx).$$
    Moreover, taking $s$th Veronese powers of the polynomial ring and performing an identical argument yields the same identity, but every composition is replaced with its $s$th rescaling (ie, apply the operation $(-)^{(s)}$ to all compositions).
\end{remark}

\begin{example}
    Let $n=2$, $d=2$, and $\alpha = (\alpha_1 , 
    \alpha_2)$ be any composition. Then the identity of Corollary \ref{cor:theCharacterId} reads
    $$h_{2+\alpha_1} (\xx^1) h_{2+\alpha_2} (\xx^2) = h_{d-1} (\xx^1) h_{d-1} (\xx^2) s_{(1 , \alpha_1)} (\xx^1) s_{(1,\alpha_2)} (\xx^2)$$
    $$- s_{(1^2,\alpha_1)} (\xx^1) s_{(1^2 , \alpha_2)} (\xx^2) - s_{(1^2,\alpha_1)} (\xx^1) s_{(2,\alpha_2)} (\xx^2) - s_{(2,\alpha_1)} (\xx^1) s_{(1^2,\alpha_2)} (\xx^2).$$
\end{example}

Our next corollary is an evident consequence of Theorem \ref{thm:explicitTorExt} combined with Observation \ref{obs:multiFactorsForSameComp} (by setting $\alpha = (1^i)$).

\begin{cor}\label{cor:SegreDerivedInvariants}
Let $A^1 , \dots , A^n$ a sequence of Koszul $R$-algebras admitting a compatible $G^i$-action for each $1 \leq i \leq n$. Assume that $M^i$ (resp. $N^i$) is a Koszul right (resp. left) $A^i$-module admitting a compatible $G^i$-action. Then the module
$$\tor_i^{A^1 \circ \cdots \circ A^n} (M^1 \circ \cdots \circ M^n , N^1 \circ \cdots \circ N^n)$$
admits a $G^1 \times \cdots \times G^n$-equivariant filtration with associated graded pieces of the form
$$\bbs^{\alpha^1}_{M^1,A^1,N^1} \otimes_R \cdots \otimes_R \bbs^{\alpha^n}_{M^n,A^n,N^n},$$
where the compositions range over all tuples $(\alpha^1 , \dots , \alpha^n) \in C(i)^{\times n}$ with $\alpha^1 \w \cdots \w \alpha^n = (1^{i})$. 

Likewise, the module
$$\tor_i^{(A^1 \circ \cdots \circ A^n)^{(d)}} ((M^1 \circ \cdots \circ M^n)^{(d)}, (N^1 \circ \cdots \circ N^n)^{(d)})$$
admits a $G^1 \times \cdots \times G^n$-equivariant filtration with associated graded pieces of the form
$$\bbs^{\alpha^1}_{M^1,A^1,N^1} \otimes_R \cdots \otimes_R \bbs^{\alpha^n}_{M^n,A^n,N^n},$$
where the compositions range over all tuples $(\alpha^1 , \dots , \alpha^n) \in C(di)^{\times n}$ with $\alpha^1 \w \cdots \w \alpha^n = (d^{i})$.
\end{cor}

\begin{example}
Let $A$ be any Koszul algebra and let us compute the filtration factors of $(A^{[3]!})^*_2 = \bbs^{(1^2)}_{A^{[3]}}$ (unfortunately, trying this for $\bbs^{(1^3)}_{A^{[3]}}$ yields $49$ filtration factors, which is too big of an example). The poset parametrizing the filtration factors is:
    % https://q.uiver.app/?q=WzAsNyxbMSwyLCIoXFx2YXJub3RoaW5nLFxcdmFybm90aGluZyxcXHZhcm5vdGhpbmcpIl0sWzAsMSwiKFxcezFcXH0gLCBcXHZhcm5vdGhpbmcgLCBcXHZhcm5vdGhpbmcpIl0sWzEsMSwiKFxcdmFybm90aGluZywgXFx7IDEgXFx9ICwgXFx2YXJub3RoaW5nKSJdLFsyLDEsIihcXHZhcm5vdGhpbmcgLCBcXHZhcm5vdGhpbmcgLCBcXHsgMSBcXH0gKSJdLFswLDAsIihcXHsxXFx9ICwgXFx7IDEgXFx9ICwgXFx2YXJub3RoaW5nKSJdLFsxLDAsIihcXHsxIFxcfSAsIFxcdmFybm90aGluZyAsIFxceyAxIFxcfSkiXSxbMiwwLCIoXFx2YXJub3RoaW5nLCBcXHsxIFxcfSAgLCBcXHsgMSBcXH0gKSJdLFswLDEsIiIsMCx7InN0eWxlIjp7ImhlYWQiOnsibmFtZSI6Im5vbmUifX19XSxbMCwyLCIiLDIseyJzdHlsZSI6eyJoZWFkIjp7Im5hbWUiOiJub25lIn19fV0sWzAsMywiIiwyLHsic3R5bGUiOnsiaGVhZCI6eyJuYW1lIjoibm9uZSJ9fX1dLFszLDYsIiIsMix7InN0eWxlIjp7ImhlYWQiOnsibmFtZSI6Im5vbmUifX19XSxbMyw1LCIiLDIseyJzdHlsZSI6eyJoZWFkIjp7Im5hbWUiOiJub25lIn19fV0sWzIsNiwiIiwxLHsic3R5bGUiOnsiaGVhZCI6eyJuYW1lIjoibm9uZSJ9fX1dLFsyLDQsIiIsMSx7InN0eWxlIjp7ImhlYWQiOnsibmFtZSI6Im5vbmUifX19XSxbMSw1LCIiLDAseyJzdHlsZSI6eyJoZWFkIjp7Im5hbWUiOiJub25lIn19fV0sWzEsNCwiIiwwLHsic3R5bGUiOnsiaGVhZCI6eyJuYW1lIjoibm9uZSJ9fX1dXQ==
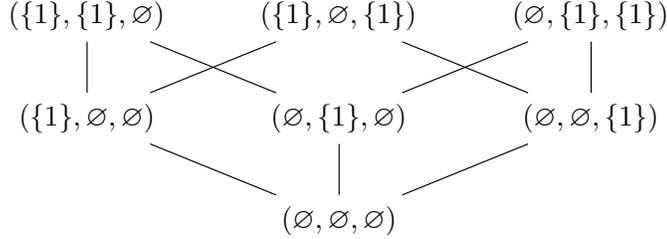
\[\begin{tikzcd}
	{(\{1\} , \{ 1 \} , \varnothing)} & {(\{1 \} , \varnothing , \{ 1 \})} & {(\varnothing, \{1 \}  , \{ 1 \} )} \\
	{(\{1\} , \varnothing , \varnothing)} & {(\varnothing, \{ 1 \} , \varnothing)} & {(\varnothing , \varnothing , \{ 1 \} )} \\
	& {(\varnothing,\varnothing,\varnothing)}
	\arrow[no head, from=3-2, to=2-1]
	\arrow[no head, from=3-2, to=2-2]
	\arrow[no head, from=3-2, to=2-3]
	\arrow[no head, from=2-3, to=1-3]
	\arrow[no head, from=2-3, to=1-2]
	\arrow[no head, from=2-2, to=1-3]
	\arrow[no head, from=2-2, to=1-1]
	\arrow[no head, from=2-1, to=1-2]
	\arrow[no head, from=2-1, to=1-1]
\end{tikzcd}\]
This yields filtration factors:
% https://q.uiver.app/?q=WzAsNyxbMSwyLCJcXGJic157KDFeMil9X0EgXFxvdGltZXNfUiBcXGJic157KDFeMil9X0EgXFxvdGltZXNfUiBcXGJic157KDFeMil9X0EiXSxbMCwxLCJcXGJic157KDIpfV9BIFxcb3RpbWVzX1IgXFxiYnNeeygxXjIpfV9BIFxcb3RpbWVzX1IgXFxiYnNeeygxXjIpfV9BIl0sWzEsMSwiXFxiYnNeeygxXjIpfV9BIFxcb3RpbWVzX1IgXFxiYnNeeygyKX1fQSBcXG90aW1lc19SIFxcYmJzXnsoMV4yKX1fQSJdLFsyLDEsIlxcYmJzXnsoMV4yKX1fQSBcXG90aW1lc19SIFxcYmJzXnsoMV4yKX1fQSBcXG90aW1lc19SIFxcYmJzXnsoMil9X0EiXSxbMCwwLCJcXGJic157KDIpfV9BIFxcb3RpbWVzX1IgXFxiYnNeeygyKX1fQSBcXG90aW1lc19SIFxcYmJzXnsoMV4yKX1fQSJdLFsxLDAsIlxcYmJzXnsoMil9X0EgXFxvdGltZXNfUiBcXGJic157KDFeMil9X0EgXFxvdGltZXNfUiBcXGJic157KDIpfV9BIl0sWzIsMCwiXFxiYnNeeygxXjIpfV9BIFxcb3RpbWVzX1IgXFxiYnNeeygyKX1fQSBcXG90aW1lc19SIFxcYmJzXnsoMil9X0EiXSxbMCwxLCIiLDAseyJzdHlsZSI6eyJoZWFkIjp7Im5hbWUiOiJub25lIn19fV0sWzAsMiwiIiwyLHsic3R5bGUiOnsiaGVhZCI6eyJuYW1lIjoibm9uZSJ9fX1dLFswLDMsIiIsMix7InN0eWxlIjp7ImhlYWQiOnsibmFtZSI6Im5vbmUifX19XSxbMyw2LCIiLDIseyJzdHlsZSI6eyJoZWFkIjp7Im5hbWUiOiJub25lIn19fV0sWzMsNSwiIiwyLHsic3R5bGUiOnsiaGVhZCI6eyJuYW1lIjoibm9uZSJ9fX1dLFsyLDYsIiIsMSx7InN0eWxlIjp7ImhlYWQiOnsibmFtZSI6Im5vbmUifX19XSxbMiw0LCIiLDEseyJzdHlsZSI6eyJoZWFkIjp7Im5hbWUiOiJub25lIn19fV0sWzEsNSwiIiwwLHsic3R5bGUiOnsiaGVhZCI6eyJuYW1lIjoibm9uZSJ9fX1dLFsxLDQsIiIsMCx7InN0eWxlIjp7ImhlYWQiOnsibmFtZSI6Im5vbmUifX19XV0=
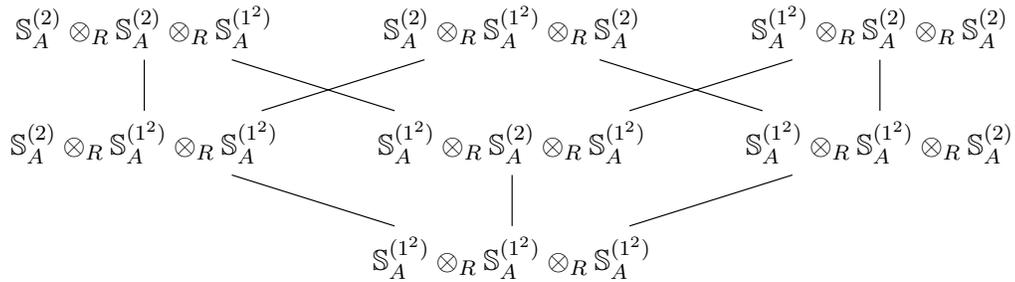
\[\begin{tikzcd}
	{\bbs^{(2)}_A \otimes_R \bbs^{(2)}_A \otimes_R \bbs^{(1^2)}_A} & {\bbs^{(2)}_A \otimes_R \bbs^{(1^2)}_A \otimes_R \bbs^{(2)}_A} & {\bbs^{(1^2)}_A \otimes_R \bbs^{(2)}_A \otimes_R \bbs^{(2)}_A} \\
	{\bbs^{(2)}_A \otimes_R \bbs^{(1^2)}_A \otimes_R \bbs^{(1^2)}_A} & {\bbs^{(1^2)}_A \otimes_R \bbs^{(2)}_A \otimes_R \bbs^{(1^2)}_A} & {\bbs^{(1^2)}_A \otimes_R \bbs^{(1^2)}_A \otimes_R \bbs^{(2)}_A} \\
	& {\bbs^{(1^2)}_A \otimes_R \bbs^{(1^2)}_A \otimes_R \bbs^{(1^2)}_A}
	\arrow[no head, from=3-2, to=2-1]
	\arrow[no head, from=3-2, to=2-2]
	\arrow[no head, from=3-2, to=2-3]
	\arrow[no head, from=2-3, to=1-3]
	\arrow[no head, from=2-3, to=1-2]
	\arrow[no head, from=2-2, to=1-3]
	\arrow[no head, from=2-2, to=1-1]
	\arrow[no head, from=2-1, to=1-2]
	\arrow[no head, from=2-1, to=1-1]
\end{tikzcd}\]
\end{example}

\begin{example}
    Let $A$ be any Koszul algebra and let $T(V)$ denote the tensor algebra on some projective $R$-module $V$. Then by definition the $R$-module $\bbs^{(1^i)}_{A,T(V)}$ admits a filtration with graded pieces of the form
    $$\bbs^{\alpha^1}_{A} \otimes_R \bbs^{\alpha^2}_{T(V)},$$
    with $\alpha^1 \w \alpha^2 = (1^i)$. Note that the only choice of $\alpha^2$ for which $\bbs^{\alpha^2}_{T(V)}$ is nonzero is for $\alpha^2 = (i)$. Thus $\alpha^1 = (1^i)$ and we find
    $$\bbs^{(1^i)}_{A,T(V)} = \bbs^{(1^i)}_A \otimes_R \bbs^{(i)}_{T(V)} = (A^!)^*_i \otimes_R V^{\otimes i}.$$
    In particular, after dualizing and collecting graded pieces, there is an isomorphism of $R$-algebras
    $$(A \circ T(V))^! = A^! \circ T(V^*).$$
    Of course, this could be verified using more direct methods, but the point is to demonstrate the utility of Theorem \ref{thm:multiSchurFactors}.
\end{example}

\begin{example}
    Let $A = S(V)$, $B = S(W)$ be symmetric algebras on free $R$-modules $V$ and $W$ both of rank $2$. Let $M = A_+^r \circ B$ and $N = A \circ B_+^{r'}$, viewed as modules over the Segre product $A \circ B$. By Corollary \ref{cor:SegreDerivedInvariants} there is an equality
    $$\tor^{A \circ B}_i (M , N) = \bbs^{(1^i)}_{(A,A_+^r) \circ (B,B_+^{r'})}.$$
    The module $\bbs^{(1^i)}_{(A,A_+^r) \circ (B,B_+^{r'})}$ has a $\gl (V) \times \gl (W)$-equivariant filtration with associated graded pieces of the form
    $$\bbs^{\alpha}_{A,A_+^r} \otimes_R \bbs^{\beta}_{B,B_+^{r'}},$$
    where $\alpha$ and $\beta$ range over all partitions with $\alpha \w \beta = (1^i)$. Notice that since $V$ and $W$ have rank $2$, we are really only ranging over all partitions $\alpha$ and $\beta$ such that all parts of $\alpha^t$ and $\beta^t$ are at most $2$, and such that $\alpha^t \vee \beta^t = (i)$. Retranslating this in terms of subsets of the Boolean poset, this is asking for all ways to partition the set $[i-1]$ into the union of two totally disconnected sets $I_\alpha \cup I_\beta$, where $1 \in I_{\beta}$ and $i-1 \in I_\alpha$. One quickly sees that there is no such decomposition if $i-1$ is odd and only $1$ such decomposition when $i-1$ is even:
    $$[i-1] = \{2, 4 , \dots , i-1 \} \cup \{ 1 , \dots , i-2 \}.$$
    Retranslating this in terms of compositions, we find there there is a $\gl (V) \times \gl (W)$-equivariant isomorphism:
    $$\tor^{A \circ B}_i (M , N) = \begin{cases}
    S_{r-1} (V) \otimes_R \det(V)^{\frac{i+1}{2}} \otimes_R S_{r'-1} (W) \otimes_R \det(W)^{\frac{i+1}{2}} & \text{if} \ i > 0 \ \text{is odd}, \\
    0 & \text{if} \ i > 0 \ \text{is even.}
    \end{cases}$$
\end{example}

\subsection{Koszul Modules Built From Ribbons and General Skew Shapes}\label{subsec:buildingKoszulModules}

In this section we apply the Koszulness criterion of Lemma \ref{lem:SESandDualityMods} to deduce that a large class of modules parametrized by ribbons are Koszul modules. This immediately yields interesting Koszul modules over any Koszul algebra (generalizing powers of $A_+$), and in the case of the symmetric algebra we are able to give a quick and much more general proof of the Koszulness of certain classes of modules formed by attaching rows to a fixed Schur functor associated to a skew-partition. We conclude with an application of these results that allows us to compute the regularity of the sheaf $\bbs^\lambda (\R)$ in arbitrary characteristic.

\begin{notation}
    Let $A$ be any Koszul algebra and $\alpha$ any fixed composition. Define the right $A$-module $\bbs^{\alpha \odot \bullet}_A$ via
    $$\bbs^{\alpha \odot \bullet}_A := \bigoplus_{d \geq 0} \bbs^{\alpha \odot (d)}_A,$$
    with right $A$-module action induced by the canonical surjections $\bbs^{\alpha}_A \otimes_R A_d \twoheadrightarrow \bbs^{\alpha \odot (d)}_A$. Similarly, the right $(A^!)^*$-comodule $\bbs^{\alpha \cdot (1^\bullet)}_A$ is defined via
    $$\bbs^{\alpha \cdot (1^\bullet)}_A := \bigoplus_{d \geq 0} \bbs^{\alpha \cdot (1^d)}_A,$$
    with comodule action induced by the canonical injections $\bbs^{\alpha \cdot (1^d)}_A \hookrightarrow \bbs^\alpha_A \otimes_R (A^!_d)^*$. The left $A$-modules $\bbs^{\bullet \odot \alpha}_A$ and left $(A^!)^*$-comodules $\bbs^{(1^\bullet) \cdot \alpha}_A$ are defined identically, but with the appropriate concatenation/near-concatenation appearing on the left.
\end{notation}

\begin{example}\label{ex:ribbonModuleExample}
    If $\alpha = (e)$ is a single integer, then the module $\bbs^{(e) \odot \bullet}_A$ is simply $(A_+)^e$, viewed as a right $A$-module. More generally, let $\Omega^i_{\geq j} (R)$ denote the $i$th syzygy of the $A$-module $R$,\footnote{That is, the image of the $i$th differential in any $A$-projective resolution of $R$.} truncated past degree $j$. Then
    $$\Omega^i_{\geq j} (R) = \bbs^{(1^{i-1},j) \odot \bullet}_A.$$
\end{example}

\begin{theorem}\label{thm:compKoszulMods}
    Let $A$ be any Koszul algebra and $\alpha$ any composition. Then the right $A$-module $\bbs^{\alpha \odot \bullet}_A$ is a Koszul $A$-module, and the minimal free resolution over $A$ has the form
    $$ \cdots \to \bbs^{\alpha \cdot (1^i)}_A \otimes_R A(-i) \to \bbs^{\alpha \cdot (1^{i-1})}_A \otimes_R A(-i+1) \to \cdots \to \bbs^{\alpha \cdot (1)} \otimes_R A(-1) \to \bbs^\alpha_A \otimes_R A.$$
    In particular, the quadratic dual of $\bbs^{\alpha \odot \bullet}_A$ is precisely the left $A^!$-module $\bbs^{\bullet \odot \rev (\alpha^t)}_{A^!}$. 
\end{theorem}

\begin{proof}
    Let $\alpha$ be any fixed partition and set $M := \bbs^{\alpha \odot \bullet}_A$. We use the criterion of Lemma \ref{lem:SESandDualityMods}(2). Let $d \geq 0$ be any integer and observe first that
    $$\bbs^\beta_{M_{\geq d} , A} = \bbs^{\alpha \odot d \cdot \beta}_A.$$
    Thus for any compositions $\beta, \gamma$ the exactness of the sequence
    $$0 \to \bbs^{\beta \cdot \gamma}_{M_{\geq d} , A} \to \bbs^{\beta}_{M_{\geq d} , A} \otimes_R \bbs^\gamma_A \to \bbs^{\beta \odot \gamma}_{M_{\geq d} , A} \to 0$$
    is equivalent to the exactness of the sequence
    $$0 \to \bbs^{\delta \cdot \gamma}_A \to \bbs^\delta_A \otimes_R \bbs^{\gamma}_A \to \bbs^{\delta \odot \gamma}_A \to 0,$$
    where $\delta = \alpha \odot (d) \cdot \beta$. This latter sequence is evidently exact, since the algebra $A$ is assumed to be Koszul, whence the module $M$ is Koszul. The latter statements are trivial consequences of the Priddy complex associated to a Koszul module (see Theorem \ref{thm:thePriddyCxWorks}). 
\end{proof}

The modules $\bbs^{\alpha \odot \bullet}_A$ are constructed in such a way that they are totally compatible with any kind of ambient group actions, and the naturality of this construction leads one to wonder if there are classes of Koszul modules in the literature that are ``secretly" of the form $\bbs^{\alpha \odot \bullet}_A$ for some $\alpha$. We pose this question formally:

\begin{question}
    Are there interesting examples of Koszul modules in the literature of the form $\bbs^{\alpha \odot \bullet}_A$ for some composition $\alpha$? (One such class of examples arises as in Example \ref{ex:ribbonModuleExample}.)
\end{question}

For the remainder of this subsection, we assume that $A = S^\bullet (V)$ (the symmetric algebra) or $\bigwedge^\bullet V$ (the exterior algebra), where $V$ is any free $R$-module ($R$ is still assumed to be a commutative ring). In this setting, we have access to the classically defined Schur functors $\bbs^D (V)$ of Akin-Buchsbaum-Weyman \cite{akin1982schur}, where $D = \lambda / \mu$ is a skew partition. For a composition $\alpha$, the notation $D \odot \alpha$ will denote the skew partition obtained by attaching the bottom row of $\alpha$ to the top row of the diagram $D$. Likewise, the notation $D \cdot \alpha$ is defined to be the skew partition obtained by concatenating the ribbon diagram associated with $\alpha$ to the top row of $D$.

\begin{example}
    Let $D = (3,3,2)/(1)$ and $\alpha = (2,2)$. Then
    $$D \odot \alpha = \ytableausetup{boxsize=1em} \ydiagram{4+2,1+4,3,2} \quad \text{and} \quad D \cdot \alpha = \ydiagram{3+2,2+2,1+2,3,2}$$
\end{example}

\begin{remark}
    The concatenation/near-concatenation of arbitrary diagrams $D$ and $D'$ is defined in \cite[Definition 3.4]{almousa2022equivariant}, but we will not need this level of generality here.
\end{remark}

\begin{definition}
    Let $D = \lambda / \mu$ be a skew-partition. The notation $\bbs^{D \odot \bullet} (V)$ will denote the $S^\bullet (V)$-module with
    $$\bbs^{D \odot \bullet} (V) := \bigoplus_{d \geq 0} \bbs^{D \odot (d)} (V),$$
    with multiplication induced by the canonical surjections
    $$\bbs^D (V) \otimes_R S^d (V) \twoheadrightarrow \bbs^{D \odot d} (V).$$
     Likewise, the notation $\bbs^{D \cdot (1^\bullet)} (V)$ denotes the $\bigwedge^\bullet V$-comodule with
    $$\bbs^{D \cdot (1^\bullet)} (V) := \bigoplus_{d \geq 0} \bbs^{D \cdot (1^d)} (V)$$
    with comultiplication induced by the natural inclusions
    $$\bbs^{D \cdot (1^d)} (V) \hookrightarrow \bbs^{D} (V) \otimes_R \bigwedge^d V.$$
    The analogous definitions for $\bbs^D$ replaced by the Weyl functors $\bbw^D$ will be used, with the tacit knowledge that the module $\bbw^{D \odot \bullet} (V)$ is instead a $\bigwedge^\bullet V$-module. Likewise, the modules $\bbs^{\bullet \odot D} (V)$ (resp. $\bbw^{\bullet \odot D} (V)$) and $\bbs^{(1^\bullet) \cdot D} (V)$ (resp. $\bbw^{(1^\bullet) \cdot D} (V)$) are defined analogously.
\end{definition}

\begin{remark}
    A simple way to define the module structure on $\bbs^{D \odot \bullet} (V)$ is to take advantage of the short exact sequence (see \cite[Proposition 3.6]{almousa2022equivariant}):
    $$0 \to \bbs^{D \cdot (1)} (V) \to \bbs^{D} \otimes_R V \to \bbs^{D \odot (1)} (V) \to 0,$$
    and then define $\bbs^{D \odot \bullet} (V)$ to be the quadratic $S^\bullet (V)$-module induced by the above short exact sequence (there is always a standard way to do this -- see, for instance, \cite{polishchuk2005quadratic}).
\end{remark}

\begin{theorem}\label{thm:KoszulSkewShapes}
    Let $A := S^\bullet (V)$, the symmetric algebra on a free $R$-module $V$. For any skew partition $D := \lambda / \mu$, the $S^\bullet (V)$-module $\bbs^{D \odot \bullet} (V)$ is Koszul, and the minimal free resolution over $A$ has the form
    $$\cdots \to \bbs^{D \cdot (1^i)} \otimes_R A(-i) \to \bbs^{D \cdot (1^{i-1})} (V) \otimes_R A(-i+1) \to \cdots \to \bbs^{D \cdot (1)} (V) \otimes_R A(-1) \to \bbs^{D} (V) \otimes_R A.$$
    In particular, the quadratic dual of $\bbs^{D \odot \bullet} (V)$ is precisely the $\bigwedge^\bullet V^*$-module $\bbw^{\bullet \odot D^t} (V)$. The analogous statement for the $A = \bigwedge^\bullet V$-module $\bbw^{D \odot \bullet} (V)$ also holds.
\end{theorem}

\begin{proof}
    Define $M := \bbs^{D \odot \bullet} (V)$ and notice that by identical reasoning to the proof of Theorem \ref{thm:compKoszulMods}, for any integer $d \geq 0$ and composition $\alpha$ there is an equality
    $$\bbs^\alpha_{M_{\geq d},A} = \bbs^{D \odot (d) \cdot \alpha} (V),$$
    whence the sequences of Lemma \ref{lem:SESandDualityMods} read
    $$0 \to \bbs^{D \odot (d) \cdot \alpha \cdot \beta} (V) \to \bbs^{D \odot (d) \cdot \alpha} (V) \otimes_R \bbs^\beta_A \to \bbs^{D \odot (d) \cdot \alpha \odot \beta} (V) \to 0.$$
    By \cite[Proposition 3.6]{almousa2022equivariant}, this sequence is exact (in the notation of \cite{almousa2022equivariant}, the diagram denoted $D$ is the diagram $D \odot (d) \cdot \alpha$ in our notation and $D'$ is the ribbon diagram associated to $\beta$). By Lemma \ref{lem:SESandDualityMods}(2), the module $M$ is Koszul, and the latter statements are again an immediate consequence of the Priddy complex (see Theorem \ref{thm:thePriddyCxWorks}). 
\end{proof}

As a further application, we conclude this section with a characteristic-free computation of the regularity of a certain class of vector bundles on projective space by using the resolution of Theorem \ref{thm:KoszulSkewShapes}.

\begin{notation}
    Let $V$ be any $\kk$-vector space and $\bbp (V)$ denote projective space on $V$. The \defi{tautological subbundle} $\R$ on $\bbp (V)$ is the twisted sheaf $\Omega (1)$, where $\Omega$ denotes the cotangent bundle. More concretely, $\R$ is defined via a twist of the Euler sequence
    $$0 \to \R \to V \otimes_{\kk} \OO_{\bbp(V)} \to \OO_{\bbp (V)} (1) \to 0.$$
\end{notation}

For convenience, recall that a sheaf $\cat{F}$ on some variety $X$ is $r$-\defi{regular} if 
$$H^i (X , \cat{F} (r-i)) = 0$$
for all $i > 0$. The \defi{regularity} of a sheaf $\cat{F}$ is defined to be the minimal integer $r$ such that $\cat{F}$ is $r$-regular. We conclude our applications with a characteristic-free regularity computation for a canonical class of vector bundles on projective space:

\begin{theorem}\label{thm:regularityThm}
    Let $V$ be a $\kk$-vector space and $\R$ the tautological subbundle on $\bbp (V)$. Given any partition $\lambda = (\lambda_1 , \dots , \lambda_n)$ (where $n = \dim V$), there is an exact sequence of vector bundles of $\bbp (V)$:
    \begin{equation}\label{eqn:theSchurRes}
        \cdots \to \bbs^{(\lambda_1 , \lambda) \cdot 1^i} (V) \otimes_\kk \OO_{\bbp (V)} (-\lambda_1 - i) \to \cdots \to \bbs^{(\lambda_1,\lambda)\cdot 1} (V) \otimes_\kk \OO_{\bbp (V)} (-\lambda_1-1)
    \end{equation}
    $$  \to \bbs^{(\lambda_1 , \lambda)} (V) \otimes_\kk \OO_{\bbp (V)} (-\lambda_1) \to \bbs^\lambda (\R) \to 0.$$
    Moreover, writing $\lambda^t = (\lambda_1^t , \dots , \lambda_m^t)$ (where $\lambda_m^t > 0$)\footnote{In other words, $\lambda_m^t$ is the length of the rightmost column of the tableau associated to $\lambda$.}, there is an equality
    $$H^{\lambda_m^t} \left( \bbp (V) , \bbs^\lambda (\R) (\lambda_1-\lambda_m^t-1) \right) = \bbs^{\left( (\lambda_1-1)^{\lambda_m^t+1} ,\lambda_2 , \dots , \lambda_n \right)} (V).$$
    In particular, the sheaf $\bbs^\lambda (\R)$ has regularity $\lambda_1$. 
\end{theorem}

\begin{remark}
    Theorem \ref{thm:regularityThm} generalizes a theorem of Gao-Raicu \cite[Theorem 2.2]{gao2022cohomology}, where the authors used the Kempf-Weyman geometric technique \cite{weyman2003} to construct the above resolution in the case $\lambda = (a)$; they used this to prove that the regularity of the sheaf $S^a (\R)$ is precisely $a$.

    The subtlety of needing to concatenate the partition $\lambda$ on the rightmost column was not present in \cite{gao2022cohomology}, since concatenating a column of the left/right of the diagram for the partition $(a,a)$ yields isomorphic representations.
\end{remark}

\begin{example} 
    Consider the partition $$\lambda := (3,3,2,1) = \ytableausetup{boxsize=1em} \ydiagram{3,3,2,1}
    $$
    so that $\lambda^t = (4,3,2)$. Then Theorem \ref{thm:regularityThm} implies that there is a resolution of the form
    $$\cdots \to  \ytableausetup{boxsize = 0.5em} \ydiagram{2+1,2+1,3,3,3,2,1} \otimes_\kk \OO_{\bbp (V)} (-5) \to  \ydiagram{2+1,3,3,3,2,1} \otimes_\kk \OO_{\bbp (V)} (-4) \to \ydiagram{3,3,3,2,1} \otimes_\kk \OO_{\bbp (V)} (-3) \to \bbs^{(3,3,2,1)} (\R),$$
    where in the above complex a Young diagram for shape $\lambda/\mu$ corresponds to the Schur module $\bbs^{\lambda / \mu} (V)$. By stripping off the rightmost column of the shapes appearing in the above complex, Theorem \ref{thm:regularityThm} also implies that there is an isomorphism 
    $$H^2 (\bbp (V) , \bbs^{(3,3,2,1)} (\R) ) = \bbs^{(2,2,2,2,1)} (V).$$
    This is particularly evident if the ambient vector space $V$ has dimension $4$, since the above resolution becomes the short exact sequence
    $$0 \to \bbs^{(2,2,2,2,1)} (V) \otimes_\kk \det (V) \otimes_\kk \OO_{\bbp (V)} (-4) \to \bbs^{(3,3,3,2,1)} (V) \otimes_\kk \OO_{\bbp (V)} (-3) \to \bbs^{(3,3,2,1)} (\R) \to 0.$$
\end{example}

\begin{proof}[Proof of Theorem \ref{thm:regularityThm}]
    The complex (\ref{eqn:theSchurRes}) arises by taking sheaves associated to the resolutions of Theorem \ref{thm:KoszulSkewShapes} for the partition $\lambda$. The sheaf associated to the module $\bbs^{\lambda \odot \bullet} (V)$ is precisely $\bbs^\lambda (\R)$, whence the sequence (\ref{eqn:theSchurRes}) is indeed an exact sequence of vector bundles.

    Observe that it is of no loss of generality to assume that $\lambda_n = 0$, since if $\lambda_n >  0$ we may write 
    $$\bbs^{\lambda} (\R) = \det (\R)^{\lambda_n} \otimes_{\OO_{\bbp (V)}} \bbs^{(\lambda_1-\lambda_n , \lambda_2 - \lambda_n , \dots , \lambda_{n-1} - \lambda_n,0)} (\R).$$
    Using the fact that $\det (\R) = \OO_{\bbp (V)} (-1)$, we see that it indeed suffices to prove the statement of Theorem \ref{thm:regularityThm} with $\lambda_n = 0$. 
    
    Twist the complex (\ref{eqn:theSchurRes}) by $\lambda_1-\lambda_m^t-1$. The cohomology of each of the terms 
    $$\bbs^{(\lambda_1, \lambda) \cdot 1^i} (V) \otimes_\kk \OO_{\bbp (V)} (\underbrace{-i-\lambda_m^t-1}_{= -\lambda_1 - i + (\lambda_1-\lambda_m^t - 1)})$$
    is $0$ unless $i = n-\lambda_m^t-1$, since if $i < n - \lambda_1^t-1$ then $0 >-i-\lambda_m^t-1 > -n$ and hence the twists $\OO_{\bbp (V)} (-i - \lambda_m^t-1)$ have $0$ cohomology identically. If $i > n - \lambda_m^t-1$, then the Schur module $\bbs^{(\lambda_1,\lambda) \cdot (1^{n- \lambda_m^t-1} )} (V)$ is identically $0$, since the rightmost column has length strictly greater than $n$ (which is the rank of $V$). It follows that
    $$H^j \left( \bbp (V) , \bbs^{(\lambda_1, \lambda) \cdot (1^{n - \lambda_m^t-1})} (V) \otimes_\kk \OO_{\bbp (V)} (-n) \right) = \begin{cases}
        \bbs^{(\lambda_1 , \lambda) \cdot 1^{n-\lambda_m^t - 1} } (V) \otimes_\kk \det (V^*) & \text{if} \ j=n-1, \\
        0 & \text{otherwise.}
    \end{cases}$$
    Combining the fact that $\bbs^{(\lambda_1, \lambda) \cdot 1^{n-\lambda_m^t - 1}} (V) = \bbs^{\left( (\lambda_1-1)^{\lambda_m^t+1},\lambda_2, \dots , \lambda_n \right)} (V) \otimes_\kk \det (V)$ with the above equality, the hypercohomology spectral sequence implies that 
    $$H^j (\bbp (V) , \bbs^\lambda (\R) (\lambda_1-\lambda_m^t-1) ) = \begin{cases}
        \bbs^{\left( (\lambda_1-1)^{\lambda_m^t+1},\lambda_1-1,\lambda_2,\dots , \lambda_n \right)} (V) & \text{if} \ j=\lambda_m^t, \\
        0 & \text{otherwise.}
    \end{cases}$$
    The fact that $H^i (\bbp (V) , \bbs^\lambda (\R) ( r - i) ) = 0$ for all $r \geq \lambda_1$ is an immediate consequence of the complex (\ref{eqn:theSchurRes}), since twisting by any $s > \lambda_1 - \lambda_m^t  - 1$ will yield a complex of vector bundles whose terms have at most global sections.
\end{proof}

\appendix
\section{Koszul Algebras And Modules Over Commutative Rings}\label{sec:generalKoszulAlg}

The purpose of this appendix is to define Koszul algebras/modules and their quadratic duals and recall Backelin's theorem in the generality established in Section \ref{sec:augmentedBarComplexes}. After developing the machinery of refinement complexes, we relate Backelin's theorem to the exactness properties of these complexes. Much of the material in this section follows from straightforward extensions of the material of \cite{polishchuk2005quadratic}, but since we assume that $R$ is an arbitrary commutative ring and our algebras are only flat $R$-modules in each homogeneous component, there are some additional details/technicalities to be verified.

\subsection{Generalities on Quadratic Algebras and Modules}\label{subsec:generalitiesOnQuad}

\begin{definition}
    Let $A$ be any quadratic algebra and $M$ any graded (left) $A$-module $M$ of initial degree $t$. There is a canonical multiplication map $A_1^{\otimes d} \otimes_R M_t \to M_{t+d}$ for every $d \geq 0$; the kernel of this map will be denoted $Q^M_{t+d}$.

    The module $M$ is called \defi{quadratic} if:
    \begin{enumerate}
        \item The canonical map $A_1^{\otimes d} \otimes_R M_t \to M_{t+d}$ is surjective for all $d \geq 0$, and
        \item for every $d \geq 0$, there is an equality
        $$Q^M_{d+t} = Q_2^A \otimes_R A_1^{\otimes d-2} \otimes_R M_t + \cdots + \underbrace{A_1^{\otimes i} \otimes_R Q_2^A \otimes_R A_1^{\otimes d-i-2} \otimes_R M_t}_{(i+1)\text{th position}} + \cdots + A_1^{\otimes d-1} \otimes_R Q_{t+1}^M.$$ 
    \end{enumerate}
\end{definition}

\begin{definition}[Quadratic Duals]\label{def:quadraticDual}
    Let $A$ be a quadratic $R$-algebra and $M$ any quadratic left $A$-module of initial degree $t$. The \defi{quadratic dual} $A^! \subset \ext^\bullet_A (R,R)$ is defined to be the subalgebra
    $$\bigoplus_{i \in \bbz} \ext^i_A (R,R)_i \subset \ext^\bullet_A (R,R).$$
    Notice that this is indeed a well-defined subalgebra, since the Yoneda product respects both the cohomological and internal grading. Viewing $\ext^i_A (M,R)$ as a right $A^!$-module (via Yoneda composition), define the \emph{quadratic dual} $M^! \subset \ext^i_A (M,R)$ to be the $A^!$-submodule
    $$\bigoplus_{i \in \bbz} \ext^i_A (M,R)_{i+t} \subset \ext^\bullet_A (M,R).$$
    The quadratic dual $M^!$ of a right $A$-module $M$ is defined analogously and is a left $A^!$-module.
\end{definition}

\begin{remark}
    The above definition is indeed well-defined for right $A$-modules, since a right $A$-module is equivalently a left $A^\op$-module, and it is evident that there is an isomorphism of algebras
    $$\ext^\bullet_{A^\op} (R,R) \cong \ext^\bullet_A (R,R)^\op.$$
    Thus $\ext^\bullet_A (M,R)$ is a right $\ext^\bullet_A (R,R)^\op$-module, and hence a left $\ext^\bullet_A (R,R)$-module.
\end{remark}

\begin{remark}
    The notion of a \emph{quadratic dual} is typically only reserved for Koszul algebras. The modules $A^!$ and $M^!$ as defined in Definition \ref{def:quadraticDual} are sometimes referred to as diagonal subalgebras and diagonal submodules of the Ext algebra/module, but in view of the observation below it seems appropriate to use the name quadratic dual for the general construction.
\end{remark}

\begin{obs}
    Let $A$ be a quadratic $R$-algebra and $M$ any quadratic left $A$-module of initial degree $t$. Then the quadratic dual $A^!$ is a quadratic algebra, and likewise the quadratic dual $M^!$ is a quadratic right $A^!$-module. 
\end{obs}

\begin{proof}
    Dualizing the Bar complex $\bar^A (A)$, the $n$th graded piece of the quadratic dual $A^!$ is by definition defined to be the cokernel of the map
    $$\bigoplus_{i = 1}^{n-1} A_1^{* \otimes i-1} \otimes_R A_2^* \otimes_R A_1^{*\otimes n-i-1} \to A_1^{* \otimes n}$$
    Moreover, the Yoneda product is induced by the tensor algebra product on the cobar construction, in which case the algebra $A^!$ is by definition a quadratic algebra. Similarly, dualizing the bar complex $\bar^A (M)$ implies that
    $$M^!_n := \coker \left( \begin{matrix} M_{t+1}^* \otimes_R A_1^{* \otimes n-t-1} \\ \oplus \\ \bigoplus_{i=1}^{n-t-1} M_t^* \otimes_R A_1^{* \otimes i-1} \otimes_R A_2^* \otimes_R A_1^{* \otimes n-t-i-1} \end{matrix} \to M_t^* \otimes_R A_1^{* \otimes n-t} \right).$$
    Again, the right Yoneda module structure is induced by the right tensor algebra structure on the cobar complex, in which case $M^!$ is a quadratic right $A^!$-module.
\end{proof}

Finally, we conclude this section by defining a Koszul algebra/module:

\begin{definition}
    Let $A$ be a quadratic $R$-algebra. The algebra $A$ is \defi{Koszul} if there is an isomorphism of $R$-algebras
    $$A^! \cong \ext^\bullet_A (R,R).$$
    Likewise, let $M$ be any left $A$-module. Then the module $M$ is \defi{Koszul} if there is an isomorphism of right $A^!$-modules 
    $$M^! \cong \ext^\bullet_A (M ,R).$$
    In other words: the inclusions of Definition \ref{def:quadraticDual} are actually equalities.
\end{definition}

\subsection{Refinement Complexes}\label{subsec:refinementComplexes}

In this section, we recall the notion of \emph{refinement complexes}; it should be noted here that the terminology is new, but such complexes (often unnamed) have been studied before (see for instance \cite[Chapter 2.8]{polishchuk2005quadratic}). These complexes may be understood as (subquotients of) homogeneous strands of the augmented bar complex associated to a quadratic algebra $A$. 

Throughout this section, we will use the notation for compositions and the standard operations between them established in Subsection \ref{subsec:compositionOps}.

\begin{definition}[Refinement Complexes]
    Let $\alpha = (\alpha_1 , \dots , \alpha_n)$ be a composition of some integer $d$. Define the \defi{refinement (chain) complex} $R_\bullet^{A,M} (\alpha)$ to be the chain complex with
    $$R^{A,M}_i (\alpha) = \bigoplus_{\substack{\beta \geq \alpha \\ \ell (\beta) - \ell (\alpha) = i}} \left( (A^!)^* \otimes_R (M^!)^*\right)_\beta,$$
    and differential induced by the cobar differential on $\cobar^{A^!} (M^!)$. Likewise, define the \defi{refinement (cochain) complex} $R_{A,M}^\bullet (\alpha)$ to be the cochain complex with
    $$R^i_{A,M} (\alpha) = \bigoplus_{\substack{\beta \geq \alpha \\ \ell (\beta) - \ell (\alpha) = i}} (A \otimes_R M)_\beta,$$
    and differential induced by the bar differential on $\bar^A (M)$. The notation $R^A_i (\alpha)$ and $R_i^A (\alpha)$ will be shorthand for $R^{A,A_+}_i (\alpha)$ and $R^i_{A,A_+} (\alpha)$, respectively.
\end{definition}

\begin{example}
    If $A = S(V)$ is the symmetric algebra on some free $R$-module $V$, then
% https://q.uiver.app/?q=WzAsNCxbMCwwLCJSXntTKFYpfV9cXGJ1bGxldCAoMywyLDQpOiJdLFsxLDAsIlxcYmlnd2VkZ2VeOSBWIl0sWzIsMCwiXFxiZWdpbnttYXRyaXh9IFxcYmlnd2VkZ2VeNSBWIFxcb3RpbWVzIFxcYmlnd2VkZ2VeNCBWIFxcXFwgXFxvcGx1cyBcXFxcIFxcYmlnd2VkZ2VeMyBWIFxcb3RpbWVzIFxcYmlnd2VkZ2VeNiBWIFxcZW5ke21hdHJpeH0iXSxbMywwLCJcXGJpZ3dlZGdlXjMgViBcXG90aW1lcyBcXGJpZ3dlZGdlXjIgViBcXG90aW1lcyBcXGJpZ3dlZGdlXjQgViJdLFsxLDJdLFsyLDNdXQ==
\[\begin{tikzcd}
	{R^{S(V)}_\bullet (3,2,4):} & {\bigwedge^9 V} & {\begin{matrix} \bigwedge^5 V \otimes \bigwedge^4 V \\ \oplus \\ \bigwedge^3 V \otimes \bigwedge^6 V \end{matrix}} & {\bigwedge^3 V \otimes \bigwedge^2 V \otimes \bigwedge^4 V.}
	\arrow[from=1-2, to=1-3]
	\arrow[from=1-3, to=1-4]
\end{tikzcd}\]
Likewise, using the notation of Definition \ref{def:gradedCompsConventions}:
% https://q.uiver.app/?q=WzAsOSxbMCwxLCJSXlxcYnVsbGV0X0EgKDMsMiw0LDMpIDoiXSxbMSwxLCJBX3soMywyLDQsMyl9Il0sWzIsMCwiQV97KDUsNCwzKX0iXSxbMiwxLCJBX3soMyw2LDMpfSJdLFsyLDIsIkFfeygzLDIsNyl9Il0sWzMsMCwiQV97KDksMyl9Il0sWzMsMSwiQV97KDUsNyl9Il0sWzMsMiwiQV97KDMsOSl9Il0sWzQsMSwiQV97MTJ9Il0sWzIsMywiXFxiaWdvcGx1cyIsMSx7InN0eWxlIjp7ImJvZHkiOnsibmFtZSI6Im5vbmUifSwiaGVhZCI6eyJuYW1lIjoibm9uZSJ9fX1dLFs0LDMsIlxcYmlnb3BsdXMiLDEseyJzdHlsZSI6eyJib2R5Ijp7Im5hbWUiOiJub25lIn0sImhlYWQiOnsibmFtZSI6Im5vbmUifX19XSxbNSw2LCJcXGJpZ29wbHVzIiwxLHsic3R5bGUiOnsiYm9keSI6eyJuYW1lIjoibm9uZSJ9LCJoZWFkIjp7Im5hbWUiOiJub25lIn19fV0sWzYsNywiXFxiaWdvcGx1cyIsMSx7InN0eWxlIjp7ImJvZHkiOnsibmFtZSI6Im5vbmUifSwiaGVhZCI6eyJuYW1lIjoibm9uZSJ9fX1dLFsxLDNdLFszLDZdLFs2LDhdXQ==
\[\begin{tikzcd}
	&& {A_{(5,4,3)}} & {A_{(9,3)}} \\
	{R^\bullet_A (3,2,4,3) :} & {A_{(3,2,4,3)}} & {A_{(3,6,3)}} & {A_{(5,7)}} & {A_{12}} \\
	&& {A_{(3,2,7)}} & {A_{(3,9)}}
	\arrow["\bigoplus"{description}, draw=none, from=1-3, to=2-3]
	\arrow["\bigoplus"{description}, draw=none, from=3-3, to=2-3]
	\arrow["\bigoplus"{description}, draw=none, from=1-4, to=2-4]
	\arrow["\bigoplus"{description}, draw=none, from=2-4, to=3-4]
	\arrow[from=2-2, to=2-3]
	\arrow[from=2-3, to=2-4]
	\arrow[from=2-4, to=2-5]
\end{tikzcd}\]
\end{example}

\subsection{Koszulness and Distributivity}\label{subsec:KoszulnessAndDistr}

In this section, we recall Backelin's theorem for Koszul algebras. For convenience, we state explicitly the following equivalent conditions for Koszulness, which are trivial retranslations of the definition.

\begin{obs}
    Let $A$ be any quadratic algebra. Then the following are equivalent:
    \begin{enumerate}
        \item The algebra $A$ is Koszul.
        \item For all $j >i$, one has $\ext^i_A (R,R)_j = 0$.
        \item For all $j > i$, one has $\tor_i^A (R,R)_j = 0$. 
    \end{enumerate}
\end{obs}

\begin{definition}\label{def:KoszulSubmodcollection}
    Let $A$ be a quadratic $R$-algebra. Given positive integers $n , i > 0$, use the notation
    $$S_{A,i}^n := A_{1}^{\otimes i-1} \otimes_R Q_2^A \otimes_R A_1^{\otimes n-i-1} \subset A_1^{\otimes n}.$$
\end{definition}

The following crucial observation ties all the machinery introduced in Subsection \ref{subsec:refinementComplexes} to the theory of distributivity developed in Section \ref{sec:qdistributivity}.

\begin{obs}\label{obs:subspaceIsomorphisms}
    Let $A$ be a Koszul $R$-algebra with $n,i >0$ positive integers. Then there is an isomorphism of $R$-modules
    $$\frac{A_1^{\otimes n}}{\vee_{i \in I} S_{A,i}^n} \cong A_{\phi (I)}.$$
    In particular, with notation as in Definition \ref{def:booleanAndRefinement} and Construction \ref{cons:theDistributivityCx} there are isomorphisms of complexes
    $$R^A_\bullet (\alpha) = C_\bullet^{\phi^{-1} (\alpha)} (A_1^{\otimes |\alpha|} ; S_{A,1}^n , \dots , S_{A,n-1}^n), \quad \text{and}$$
    $$R^\bullet_A (\alpha) = C^\bullet_{\phi^{-1} (\alpha)} (A_1^{\otimes |\alpha|} ; S_{A,1}^n , \dots , S_{A,n-1}^n).$$
\end{obs}

Finally, we state and prove the generalization of Backelin's theorem. Notice that the proof is deceptively short, but relies on the entirety of the material developed thus far in the paper.

\begin{theorem}\label{thm:KoszulnessIsDistributive}
    Let $A$ be any quadratic $R$-algebra. Then:
    $$A \ \text{is Koszul}$$
    $$\iff$$
    $$ \text{the collection} \ S_{A,1}^n , \dots , S_{A,n-1}^n \subset A_1^{\otimes n} \ \text{is distributive for all} \  n > 0. $$
    In particular, the refinement complexes $R^\bullet_A (\alpha)$ and $R^A_\bullet (\alpha)$ are exact in positive (co)homological degrees for all compositions $\alpha$.
\end{theorem}

\begin{remark}
    Notice that it is clear that distributivity implies that $A$ is Koszul, since this means that the complex $R^\bullet_{A} (1^d)$ is in particular exact in positive cohomological degrees. It is not obvious at all that Koszulness should be sufficient to imply exactness of the refinement complexes for \emph{all} choices of compositions.
\end{remark}

\begin{proof}
    This is an immediate consequence of Backelin's theorem combined with Observation \ref{obs:subspaceIsomorphisms}.
\end{proof}

\begin{remark}
Notice that with the distributivity perspective of Koszul algebras,
    $$A_n \longleftrightarrow A^!_n \quad \text{corresponds to}$$
    $$\frac{A_1^{\otimes n}}{S_{1,A}^n + \cdots + S_{n-1,A}^n} \longleftrightarrow \frac{A_1^{* \otimes n}}{{S_{1,A}^n}^\vee + \cdots + {S_{n-1,A}^n}^\vee}.$$
    This is another quick way to see that $A \cong (A^!)^!$ as $R$-algebras.
\end{remark}

\begin{cor}\label{cor:VeroneseKoszulness}
    Let $A$ be a quadratic $R$-algebra. Then $A$ is Koszul if and only if $\veralg{d}$ is Koszul for every $d > 0$. 
\end{cor}

\begin{remark}
    Of course, the nontrivial direction of Corollary \ref{cor:VeroneseKoszulness} is the fact that $A$ is Koszul implies that $\veralg{d}$ is Koszul for all $d > 0$. The distributivity criterion makes this a trivial consequence; it is worth mentioning that this was originally proved by Barcanescu and Manolache \cite{BarcanescuManolache}.
\end{remark}

\subsection{Koszul Modules Over Koszul Algebras}\label{subsec:KoszulMods}

The following section makes some additional observations about Koszul modules that will be useful to reference explicitly in earlier sections. We start with an analogous observation on equivalent conditions for Koszulness of a module:

\begin{obs}
    Let $A$ be a Koszul algebra and $M$ any quadratic left $A$-module of initial degree $t$. Then the following are equivalent:
    \begin{enumerate}
        \item The module $M$ is Koszul.
        \item For all $j > i$, one has $\ext^i_A (M,R)_{t+j} = 0$.
        \item For all $j > i$, one has $\tor_i^A (M,R)_{t+j} = 0$.
    \end{enumerate}
\end{obs}

The following submodule collections are the evident analog of the collection in Definition \ref{def:KoszulSubmodcollection} for quadratic modules:

\begin{definition}\label{def:koszulModSubmodCollection}
    Let $A$ be a quadratic algebra and $M$ any quadratic left $A$-module of initial degree $t$. Given positive integers $n,i > 0$, use the notation
    $$S_{A,M,i}^n := \begin{cases} A_1^{\otimes i-1} \otimes_R Q_2 \otimes_R A_1^{n-i-2} \otimes_R M_t & \text{if} \ i < n-1, \\
    A_1^{\otimes n-1} \otimes_R Q_{t+1}^M & \text{if} \ i = n-1 \\ 
    \end{cases} \subset A_1^{\otimes n-1} \otimes_R M_t.$$
    The submodules $S_{M,A,i}^n \subset M_t \otimes_R A_1^{\otimes n-1}$ for a right $A$-module $M$ are defined analogously.
\end{definition}

\begin{theorem}\label{thm:koszulModuleDistr}
    Let $M$ be a left (resp. right) $A$-module of initial degree $t$, where $A$ is a Koszul $R$-algebra. Then
    $$M \ \text{is Koszul}$$
    $$\iff$$
    $$ \text{the collection} \ S_{A,M,1}^n , \dots , S_{A,M,n-1}^n \subset A_1^{\otimes n} \otimes_R M_t \ \text{is distributive for all} \  n > 0. $$
    The analogous statement for right $A$-modules holds as well. In particular, the refinement complexes $R^\bullet_{A,M} (\alpha)$ and $R^{A,M}_\bullet (\alpha)$ are exact in positive (co)homological degrees for all compositions $\alpha$.
\end{theorem}

\begin{proof}
    The proof is identical to the proof of Theorem \ref{thm:KoszulnessIsDistributive}.
\end{proof}

The following observation is immediate upon applying $- \otimes_R M$ to the augmented bar complex $\bar^A (A)$:

\begin{obs}\label{obs:ProjectivesAreKoszul}
    Let $A$ be a Koszul $R$-algebra. Then any flat $R$-module is a Koszul $A$-module by viewing the $R$-module as a left (or right) $A$-module concentrated in some fixed degree.
\end{obs}

% \begin{proof}
%     Assume that $M$ has initial degree $t$. There is a spectral sequence
%     $$E^2_{p,q} = \tor_p^A (R , \tor_q^R (A , M)) \implies \tor_{p+q} (R , M).$$
%     Since $M$ is assumed to be $R$-flat, $\tor_q^R (A,M) = 0$ unless $q=0$ and we find
%     $$E^2_{p,q} = \begin{cases} \tor_p^A (R , A \otimes_R M) & \text{if} \ q = 0, \\
%     0 & \text{otherwise.} \end{cases}$$
%     The $A$-module $A \otimes_R M$ is $A$-flat
% \end{proof}

\begin{cor}\label{cor:truncationsAreKoszul}
    If $M$ is a Koszul left (resp. right) $A$-module, then the truncation $M_{\geq d}$ is a Koszul left (resp. right) $A$-module for all $d \in \bbz$. 
\end{cor}

\begin{proof}
    Let $t$ denote the initial degree of $M$, and proceed by induction on the difference $d-t$. When $d-t = 0$, it is by assumption that $M_{\geq t} = M$ is Koszul.

    Let $d-t > 0$. There is a short exact sequence
    $$0 \to M_{\geq d-t} \to M_{\geq d-t-1} \to M_{d-t-1} \to 0,$$ 
    where $M_{d-t-1}$ is a flat $R$-module viewed as being concentrated in degree $d-t-1$. By the inductive hypothesis, the truncation $M_{\geq d-t-1}$ is Koszul and by Observation \ref{obs:ProjectivesAreKoszul} the $A$-module $M_{d-t-1}$ is Koszul. By the long exact sequence of cohomology, the truncation $M_{\geq d-t}$ must also be Koszul.
\end{proof}

\begin{obs}\label{obs:gradedDualKoszul}
    If $M$ is a Koszul left $A$-module over a Koszul algebra $A$, then the graded dual $M^*$ is a Koszul right $A$-module.
\end{obs}

\begin{proof}
    It is clear by definition that each graded component of $M^*$ is a flat $R$-module, so it remains to prove that $\ext^\bullet_{A^\op} (M^* , R)$ is generated in minimal internal degree. There is a string of isomorphisms of algebras:
    $$\ext^\bullet_{A} (M,R)^\op \cong \ext^\bullet_{A^\op} (R,M) \cong \tor_\bullet^{A^\op} (R,M^*)^*.$$
    Since $\ext^\bullet_A (M,R)$ is generated in minimal internal degrees, so is $\tor_\bullet^{A^\op} (R,M^*)$. This implies that $M^*$ is Koszul.
\end{proof}

We conclude this subsection with a statement about the Koszulness of tensor products of Koszul algebras. This will be most useful when dealing with multi-Schur functors.

\begin{cor}\label{cor:tensorsAreKoszul}
    Let $A$ and $B$ be two Koszul algebras and let $M$ (resp. $N$) a left $A$ (resp. $B$)-module. Then the tensor product $A \otimes_R B$ is a Koszul algebra and $M \otimes_R N$ is a Koszul left $A \otimes_R B$-module. 
\end{cor}

\begin{proof}
    Recall that for any (left) $A$-module $M$ and $B$-module $N$, the shuffle product
    $$\nabla: \bar^A (M) \otimes_R \bar^B (N) \to \bar^{A \otimes_R B} (M \otimes_R N)$$
    is a well-defined morphism of complexes. The complex $\bar^{A \otimes_R B} (M \otimes_R N)$ is a flat resolution of $M \otimes_R N$, and by the assumption that all the modules $A$, $B$, $M$, and $N$ are flat as $R$-modules the tensor product $\bar^A (M) \otimes_R \bar^B (N)$ is also a flat resolution of $M \otimes_R N$. Since $R \otimes_A \bar^A (M)$ and $R \otimes_B \bar^B (N)$ both have homology concentrated in minimal degrees, so does $R \otimes_{A \otimes_R B} \bar^{A \otimes_R B} (M \otimes_R N)$. 
    
    % is a homotopy equivalence of complexes that respects the internal grading, whence upon dualizing and restricting to graded components we have a quasi-isomorphism
    % $$\nabla^*_n : \bar^{A \otimes_R B} (M \otimes_R N)^*_n \to \left( \bar^A (M)^* \otimes_R \bar^B (N)^* \right)_n.$$
    % By definition of the tensor product of complexes there is an equality
    % $$\left( \bar^A (M)^* \otimes_R \bar^B (N)^* \right)_n = \bigoplus_{i+j = n} \bar^A (M)_i^* \otimes_R \bar^B (N)_j^*.$$
    % By the assumption that all of the modules $A$, $B$, $M$, and $N$ are projective as $R$-modules, each of the complexes $\bar^A (M)_i^* \otimes_R \bar^B (N)_j^*$ are exact in homological degrees $< n$. This immediately implies Koszulness \keller{add more detail here, }
\end{proof}

\subsection{The Priddy Complex}\label{subsec:PriddyCx}

In this section, we recall the well-known construction of the Priddy complex (originally proved in the work of Priddy \cite{priddy1970koszul}) for our setting. The definition is identical to that of the original definition over fields, but we include the definition along with the relevant properties for completeness.

\begin{construction}\label{cons:PriddyCons}
    Let $A$ be a Koszul $R$-algebra and $M$ a left Koszul module. For each $i \geq 0$, there is a canonical inclusion of left $A$-modules
    $$A \otimes_R (M^!)^*_{t+i} \hookrightarrow A \otimes_R A_1^{\otimes i} \otimes_R M_t = \bar^A_i (M)_i,$$
    where the left $A$-module structure on $A \otimes_R (M^!)^*$ comes from only allowing $A$ to act on the leftmost tensor factor. This inclusion thus lifts to an inclusion of complexes
    $$A \otimes_R (M^!)^*_{\bullet} \hookrightarrow \bar^A (M).$$
    The differential on the complex $A \otimes_R (M^!)^*_{\bullet}$ is induced by the Bar complex differential (this is indeed well-defined).
\end{construction}

\begin{theorem}\label{thm:thePriddyCxWorks}
    Let $A$ be a quadratic algebra and $M$ any left $A$-module. Then the following are equivalent:
    \begin{enumerate}
        \item The left $A$-module $M$ is Koszul.
        \item The complex $A \otimes_R (M^!)^*$ is a flat resolution of $M$ over $A$. 
    \end{enumerate}
\end{theorem}

\begin{proof}
    $(2) \implies (1)$: This implication is clear, since if $A \otimes_R (M^!)^*$ is a flat resolution of $M$ over $A$, tensoring with $R$ over $A$ implies that $\tor^A_i (R,M)$ is concentrated in degree $i + t$ (recall that Tor may be computed using flat resolutions). 

    $(1) \implies (2)$: Let $\iota : A \otimes_R (M^!)^*$ be the inclusion of Construction \ref{cons:PriddyCons} and consider the mapping cone $\cone (\iota)$. There is the tautological short exact sequence of complexes:
    $$0 \to R \otimes_A \bar^A (M) \to \cone (R \otimes_A \iota) \to (M^!)^*[-1] \to 0.$$
    The assumption that $M$ is Koszul implies that $R \otimes_A \iota$ is a quasi-isomorphism, so that $\cone (R \otimes_A \iota)$ is an exact complex of flat $R$-modules. However, this implies by Nakayama's lemma $\cone (\iota)$ must be an exact complex, whence $A \otimes_R (M^!)^*$ is an $A$-flat resolution of $M$. 
\end{proof}

\bibliographystyle{amsalpha}
\bibliography{biblio}
\addcontentsline{toc}{section}{Bibliography}

\end{document}